\documentclass[12pt]{article}
\usepackage{amsmath,amssymb,amscd}
\usepackage{ulem}
\usepackage{xcolor}
\allowdisplaybreaks

\def\bilin#1#2{\left\langle#1,\,#2\right\rangle}

 \newtheorem{theorem}{Theorem}[section]
 \newtheorem{proposition}[theorem]{Proposition}
 \newtheorem{corollary}[theorem]{Corollary}
 \newtheorem{lemma}[theorem]{Lemma}
 \newtheorem{definition}[theorem]{Definition}
 \newtheorem{example}[theorem]{Example}
 \newtheorem{remark}[theorem]{Remark}

\numberwithin{equation}{section}

\newenvironment{proof}{\smallskip\par{\sc Proof.}\enspace}%
 {{\unskip\nobreak\hfil\penalty50\hskip2em
          \hbox{}\nobreak\hfil{\rule[-1pt]{5pt}{10pt}}
          \parfillskip=0pt\finalhyphendemerits=0
          \par\medskip}} 

\begin{document}

\vspace*{.3in}

\begin{center}
\LARGE
{\sf Stochastic Partial Differential Equations
Associated with Pseudo-Differential Operators
and Hilbert Space-Valued Gaussian Processes
}
\end{center}

\bigskip
\begin{center}
Un Cig Ji \\
Department of Mathematics\\
Institute for Industrial and Applied Mathematics\\
Chungbuk National University\\
Cheongju 28644, Korea \\
\texttt{E-Mail:uncigji@chungbuk.ac.kr }
\end{center}

\begin{center}
Jae Hun Kim \\
Department of Mathematics\\
Chungbuk National University\\
Cheongju 28644, Korea \\
\texttt{E-Mail:jaehunkim@chungbuk.ac.kr }
\end{center}

\begin{abstract}
In this paper, we prove the unique existence and investigate the $L^{p}$-regularity of solutions
to stochastic partial differential equations in Hilbert spaces associated with pseudo-differential operators,
driven by Hilbert space-valued Gaussian processes that satisfy certain regularity conditions 
for the covariance kernels of the Gaussian processes.
For our purposes, we develop an $L^{p}$-regularity framework for the solutions
to the stochastic partial differential equations associated with pseudo-differential operators.
As the main tools, we establish the $p$-th moment maximal inequality for stochastic integrals
with respect to a Hilbert space-valued Gaussian process
and a Littlewood-Paley type inequality for Banach space-valued functions.
Additionally, during our study, we improved the sufficient conditions for Fourier multipliers 
and examined the covariance kernels for Gaussian processes.
\end{abstract}



\newpage
\tableofcontents

\bigskip
\noindent
\section{Introduction}
Since the stochastic partial differential equations (SPDEs) introduced by Walsh \cite{Walsh 1986},
the theory of SPDEs has been developed extensively with broad applications.
The study of $L^{p}$-theory for SPDEs in the Sobolev space
was initiated by Krylov \cite{Krylov 1994,Krylov 1996,Krylov 1999,Krylov 2006}
dealing with the SPDEs associated with a sequence of independent one-dimensional Wiener processes,
and then the $L^{p}$-theory for SPDEs has been studied by several authors
\cite{Balan 2009,Balan 2011,Chen 2015,Han 2021, I. Kim 2019} and references cited therein.
In \cite{Balan 2009,Balan 2011}, the author established $L^{2}$ and $L^{p}$-theory ($p\ge2$)
for the stochastic heat equation associated with a sequence of fractional Brownian motions (fBm), respectively.
The stochastic heat equations driven by general single Gaussian noise
have been studied by several authors in \cite{Balan 2008, Bardina 2010, Hu 2015,Huang 2020,Mocioalca 2007}.

On the other hand, SPDEs driven by Hilbert space-valued Wiener process has been studied by many authors
(see, e.g., \cite{Da Prato 2014, Kruse 2014, Liu 2010, Rockner 2024, Prevot 2007}, and the references cited theirin).
In \cite{Duncan 2006, Duncan 2002,Grecksch 1999, Grecksch 2009,  Maslowski 2003},
the authors considered the stochastic integration and differential equations with respect to infinite dimensional fractional Brownian motion (fBm). More precisely, in \cite{Duncan 2002}, the authors deal with stochastic integration for deterministic operator-valued integrands, and they investigate some properties of the solution of a linear stochastic differential equation in Hilbert space with fBm using the semigroup approach. Since then, in \cite{Duncan 2006}, the authors extended the stochastic integral constructed in \cite{Duncan 2002} to random integrands, using Malliavin calculus. More precisely, in \cite{Duncan 2006}, the authors constructed certain suitable spaces for the Malliavin derivative 
and defined the stochastic integral (that is, the Skorohod integral) as the adjoint operator of Malliavin derivative.

Let $(K,\bilin{\cdot}{\cdot}_{K})$ and $(U,\bilin{\cdot}{\cdot}_{U})$ be real separable Hilbert spaces, $Q$ be a positive self-adjoint trace class operator on $U$, and let $U_{0}=Q^{\frac{1}{2}}(U)$.
Let $\boldsymbol{\beta}=\{\boldsymbol{\beta}_{t}\}_{t\in [0,T]}$ (for $0<T<\infty$) be a $Q$-Gaussian process satisfying that
\[
\mathbb{E}[\bilin{\boldsymbol{\beta}_{t}}{u}_{U}\bilin{\boldsymbol{\beta}_{s}}{v}_{U}]
=R(t,s)\bilin{Qu}{v}_{U}
\]
for a kernel function $R$.
Throughout this paper, we assume that
\begin{itemize}
  \item [{\upshape \textbf{(R1)}}] the partial derivative $\frac{\partial^{2}R}{\partial t \partial s}(t,s)$ exists and
\begin{align*}
\frac{\partial^{2}R}{\partial t \partial s}(t,s)\geq 0 \quad \text{for all }t,s\in [0,T],
\end{align*}
\item [{\upshape \textbf{(R2)}}]there exists $\boldsymbol{r}\geq 1$ such that the integral operator $K_{R}$ with the kernel $R$, defined by
\begin{align*}
\left(K_{R} f\right)(t)=\int_{0}^{T}f(s)\frac{\partial^{2}R}{\partial t \partial s}(t,s)ds,\quad t\in [0,T]
\end{align*}
for real valued bounded measurable function $f$ on $[0,T]$, is a bounded operator from $L^{\boldsymbol{r}}([0,T])$ into $L^{\boldsymbol{s}}([0,T])$
for the conjugate number $\boldsymbol{s}$ of $\boldsymbol{r}$ ($\frac{1}{\boldsymbol{r}}+\frac{1}{\boldsymbol{s}}=1$),
i.e.,
\begin{align*}
\|K_{R} f\|_{L^{\boldsymbol{s}}([0,T])}\leq C_{R}\|f\|_{L^{\boldsymbol{r}}([0,T])}
\end{align*}
for some $C_{R}>0$ (see Section \ref{sec: examples} for examples).
\end{itemize}
Motivated by \cite{Duncan 2006}, in this paper, we will define the $K$-valued stochastic integral $\int_{0}^{T}\Phi_{s}\delta\boldsymbol{\beta}_{s}$ for $K\otimes U_{0}\cong \mathcal{L}_{\rm HS}(U_{0},K)$-valued integrand $\Phi$, where $\mathcal{L}_{\rm HS}(U_{0},K)$ denotes the space of all Hilbert-Schmidt operators from $U_{0}$ into $K$.

For a complex-valued measurable function $\psi$ defined on $[0,\infty)\times \mathbb{R}^{d}$,
and for each $t\geq 0$, the pseudo-differential operator $L_{\psi}(t)$ with the symbol $\psi$
is defined by
\begin{align*}
L_{\psi}(t)f(x)=\mathcal{F}^{-1}(\psi(t,\xi)\mathcal{F}f(\xi))(x)
\end{align*}
for suitable functions $f$ defined on $\mathbb{R}^{d}$,
where $\mathcal{F}$ and $\mathcal{F}^{-1}$ are the Fourier transform
and Fourier inverse transform, respectively (see \eqref{eqn:FTand FIT}).

In this paper,
we prove the unique existence of a solution of the following SPDE in $K$:
\begin{equation}\label{eqn :SPDE with pseudo diff op}
du(t,x)=(L_{\psi}(t)u(t,x)+f(t,x))dt+g(t,x)\delta\boldsymbol{\beta}_{t},
\end{equation}
For this purpose, we will use two main tools.

One of the main tools is the $p$-th moment maximal inequality for the stochastic integral $\int_{0}^{T}\Phi_{s}\delta\boldsymbol{\beta}_{s}$.
More precisely, under certain conditions, there exists a constant $C>0$ such that
for any $u\in\mathbb{D}_{\boldsymbol{\beta}}^{1,p}(| \mathcal{H}_{K\otimes U_{0}}|) $ (see \eqref{eqn:D beta absolute}),
\begin{align}\label{eqn : cor of max.ineq for single gaussian process-0}
\mathbb{E}\sup_{t\leq T}\left\|\int_{0}^{t}u_{s}\delta \boldsymbol{\beta}_{s}\right\|_{K}^{p}\nonumber
&\leq C\left\{\mathbb{E}\left(\int_{0}^{T}\|u_{s}\|_{K\otimes U_{0}}^{q}ds\right)^{\frac{p}{q}}\right.\\
&\,\,\left.+\mathbb{E}\left(\int_{0}^{T}\left(\int_{0}^{T}\|D_{\theta}^{\boldsymbol{\beta}}u_{s}\|_{K\otimes U_{0}\otimes U_{0}}^{\boldsymbol{r}}d\theta\right)^{\frac{q}{\boldsymbol{r}}}ds\right)^{\frac{p}{q}}\right\}
\end{align}
if the right hand side is finite,
where $D^{\boldsymbol{\beta}}$ is the Malliavin derivative with respect to $\boldsymbol{\beta}$.
In \cite{Balan 2011}, the author proved the $p$-th moment maximal inequality
for the case of Hilbert space valued fBm with Hurst parameter $1/2<H<1$
(see Corollary 2.3 in \cite{Balan 2011}), in such case, $\boldsymbol{r}$ is corresponding to $1/H$.

Another one of the main tools is the generalized Littlewood-Paley type inequality.
For each $t>s$, we define the operator $\mathcal{T}_{\psi}(t,s)$ by
\begin{equation*}
\mathcal{T}_{\psi}(t,s)f(x)=p_{\psi}(t,s,\cdot)*f(x)
\end{equation*}
for suitable functions $f$, where
\begin{align*}
p_{\psi}(t,s,x):=\mathcal{F}^{-1}\left(\exp\left(\int_{s}^{t}\psi(r,\xi)dr\right)\right)(x)
\end{align*}
if the right hand side exists, and so we have
\begin{align*}
\mathcal{T}_{\psi}(t,s)f(x)
=\mathcal{F}^{-1}\left(\exp\left(\int_{s}^{t}\psi(r,\xi)dr\right)\mathcal{F}f(\xi)\right)(x),
\end{align*}
which will be used to represent the unique solution of \eqref{eqn :SPDE with pseudo diff op}
and appeared in the generalized Littlewood-Paley type inequalities \eqref{ineq:LP-ineq intro} and \eqref{ineq: 2nd LP ineq intro}.

Recently, in \cite{Ji-Kim 2025-1}, the authors proved the following inequality
under certain conditions: for any $q\geq 2$ and $p\geq q$, there exists a constant $C>0$
such that for any $f\in L^{p}((a,b)\times \mathbb{R}^{d};K)$,
\begin{align}\label{ineq:LP-ineq intro}
&\int_{\mathbb{R}^{d}}\int_{a}^{b}\left(\int_{a}^{t}(t-s)^{\frac{q\gamma_{\psi_{1}}}{\gamma_{\psi_{2}}}-1}
\|L_{\psi_{1}}
\mathcal{T}_{\psi_{2}}(t,s)f(s,\cdot)(x)\|_{K}^{q}ds\right)^{\frac{p}{q}}dtdx\nonumber\\
&\qquad\leq C\int_{a}^{b}\int_{\mathbb{R}^{d}} \|f(t,x)\|_{K}^{p}dxdt,
\end{align}
where $\gamma_{\psi_{1}},\gamma_{\psi_{2}}$ are corresponding to the orders of the symbols $\psi_{1}$ and $\psi_{2}$, respectively.
In Theorem \ref{cor: 2nd LP ineq}, we will prove the following inequality, which is an extension of
the inequality \eqref{ineq:LP-ineq intro}
to the case of  Banach space $L^{r}(\mathbb{R};H)$-valued functions:
for any $r\geq 1$, $q\geq \max\{2,r\}$ and $p\geq q$, there exists a constant $C>0$
such that for any $f\in C_{\rm c}^{\infty}((a,b)\times \mathbb{R}^{d};L^{r}(\mathbb{R};H))$,
\begin{align}\label{ineq: 2nd LP ineq intro}
&\int_{\mathbb{R}^{d}}\int_{a}^{b}\left[\int_{a}^{t}(t-s)^{\frac{q\gamma_{\psi_{1}}}{\gamma_{\psi_{2}}}-1}\left(\int_{\mathbb{R}}
\|L_{\psi_{1}}
\mathcal{T}_{\psi_{2}}(t,s)f(s,\cdot,\theta)(x)\|_{H}^{\boldsymbol{r}}d\theta\right)^{\frac{q}{\boldsymbol{r}}} ds\right]^{\frac{p}{q}}dtdx\nonumber\\
&\qquad\leq C_{1}\int_{a}^{b}\left[\int_{\mathbb{R}}\left(\int_{\mathbb{R}^{d}}
\|f(t,x,\theta)\|_{H}^{p}dx\right)^{\frac{\boldsymbol{r}}{p}}d\theta\right]^{\frac{p}{\boldsymbol{r}}}dt
\end{align}
if the right hand side is finite.
In \cite{Krylov 1994}, the author generalized the Littlewood-Paley inequality 
to the case of Hilbert space-valued functions (Theorem 1.1 in \cite{Krylov 1994}).
After, in \cite{Balan 2011}, the author extended Theorem 1.1 in \cite{Krylov 1994} to the case of $L^{1/H}((\alpha,\beta);V)$-valued functions
(Theorem A.2 in \cite{Balan 2011}),
where $H$ is the Hurst index with $1/2<H<1$ of fBm and $V$ is a Hilbert space.  
In Theorem 3.1 of \cite{I. Kim K.-H. Kim 2016}, the authors extend Theorem 1.1 in \cite{Krylov 1994} to the case of operator $\mathcal{T}_{\psi}(t,s)$.

This paper is organized as follows.
In Section \ref{sec: Malliavin}, we give some basic definitions and results of the Malliavin calculus and we prove the $p$-th moment maximal inequality with respect to the $Q$-Gaussian processes $\boldsymbol{\beta}$.
In Section \ref{sec:Fourier multi}, we examine the Fourier multipliers on $L^{p}(\mathbb{R}^{d};K)$ for a Hilbert space $K$.
In Section \ref{sec:stochastic Banach}, we introduce some stochastic Banach spaces and study their properties.
In Section \ref{sec: GLPI}, we prove the generalized Littlewood-Paley type inequalities.
In Section \ref{sec: main result}, we prove the existence and uniqueness of a solution of the equation
given in \eqref{eqn :SPDE with pseudo diff op}.
In Section \ref{sec: examples}, we discuss some examples of $Q$-Gaussian process $\boldsymbol{\beta}$
for which Assumptions \textbf{(R1)} and \textbf{(R2)} hold.
In Appendix \ref{sec:A Proof of LPI-Banach space}, we prove Theorem \ref{cor: 2nd LP ineq}.
Finally, in Appendix \ref{sec: potential}, we prove Lemma \ref{lem:L psi potential}.

\section{Malliavin Calculus}\label{sec: Malliavin}

Throughout this paper, $(\Omega,\mathcal{F}, P)$ is a complete probability space.
In this section, we study the Malliavin calculus for Gaussian processes in a separable Hilbert space.
Let $(U,\bilin{\cdot}{\cdot}_{U})$ be a real separable Hilbert space and $Q$ be a positive, self-adjoint and trace class operator on $U$.
Then motivated by the definition of $Q$-fractional Brownian motion in \cite{Duncan 2002}
(see also \cite{Alos-Mazet-Nualart2001} and \cite{Da Prato 2014}), we define a $Q$-Gaussian process as follows.

\begin{definition}
\upshape
A $U$-valued stochastic process $\boldsymbol{\beta}=\{\boldsymbol{\beta}_{t}\,|\,t\in [0,T]\}$
is called a $Q$-Gaussian process with the covariance kernel $R:[0,T]^{2}\to\mathbb{R}$ if it satisfies
\begin{itemize}
  \item [(i)]$\boldsymbol{\beta}_{0}=0$,
  \item [(ii)]$\boldsymbol{\beta}$ has continuous trajectories,
  \item [(iii)]for any $u,v\in U$, $\bilin{\boldsymbol{\beta}_{t}}{u}_{U}$ is a real valued Gaussian process and
      \begin{align}\label{eqn: kernel R}
      \mathbb{E}[\bilin{\boldsymbol{\beta}_{t}}{u}_{U}\bilin{\boldsymbol{\beta}_{s}}{v}_{U}]
      =R(t,s)\bilin{Qu}{v}_{U}.
      \end{align}
\end{itemize}
From \eqref{eqn: kernel R}, it is obvious that $R$ is symmetric.
\end{definition}

Let $\boldsymbol{\beta}$ be a $U$-valued $Q$-Gaussian process with the covariance kernel $R$.
Then we can see that there exist an orthonormal sequence $\{e_{j}\}_{j\in J}$ in $U$ and a sequence $\{\lambda_{j}\}_{j\in J}$
of positive real numbers such that $Qe_{j}=\lambda_{j}e_{j}$ for all $j\in J$, $\sum_{j\in J}\lambda_j<\infty$ and
$\boldsymbol{\beta}$ has the expression:
\begin{align*}
\boldsymbol{\beta}_{t}=\sum_{j\in J}\sqrt{\lambda_{j}}\beta_{j}(t)e_{j},\quad \beta_{j}(t)=\frac{1}{\sqrt{\lambda_{j}}}\bilin{\boldsymbol{\beta}_{t}}{e_{j}}_{U},
\end{align*}
where $\{\{\beta_{j}(t)\}_{t\ge0}\}_{j\in J}$ becomes a family of independent, identically distributed Gaussian random variables.
In fact, for any $j,k\in J$, we have
\begin{align*}
\mathbb{E}[\beta_{j}(t)\beta_{k}(s)]
=\frac{1}{\sqrt{\lambda_{j}\lambda_{k}}}
        \mathbb{E}[\bilin{\boldsymbol{\beta}_{t}}{e_{j}}_{U}\bilin{\boldsymbol{\beta}_{s}}{e_{k}}_{U}]
=R(t,s)\delta_{jk}.
\end{align*}
For the proof, we refer to \cite{Da Prato 2014}.

From now on, to avoid trivial complications we assume that $Q$ is strictly positive, i.e., $Qu>0$ if $u\ne0$,
and then it holds that $J=\mathbb{N}$ and $\{e_j\}_{j=1}^{\infty}$ is a complete orthonormal basis of $U$.
Let $U_{0}=Q^{\frac{1}{2}}(U)$ and let $Q^{-\frac{1}{2}}$ be the pseudo inverse of $Q^{\frac{1}{2}}$
(see, e.g., \cite{Prevot 2007}).
If we define the inner product
\begin{align*}
 \bilin{u}{v}_{U_{0}}=\bilin{Q^{-\frac{1}{2}}u}{Q^{-\frac{1}{2}}v}_{U}\qquad \text{for}\quad u,v\in U_{0},
\end{align*}
then $(U_{0},\bilin{\cdot}{\cdot}_{U_{0}})$ is a separable Hilbert space
and $\left\{\sqrt{\lambda_{j}}e_{j}\right\}_{j=1}^{\infty}$ is an orthonormal basis of $U_{0}$.

Let $\mathcal{E}$ be the linear span of all indicator functions of the form $1_{[0,t]},\,t\in [0,T]$.
We define the inner product
$\bilin{\cdot}{\cdot}_{\mathcal{H}}$ on $\mathcal{E}$ satisfying that
\begin{align*}
\bilin{1_{[0,t]}}{1_{[0,s]}}_{\mathcal{H}}=R(t,s).
\end{align*}
We define the reproducing kernel Hilbert space $\mathcal{H}$
by the closure of $\mathcal{E}$ with respect to the inner product $\bilin{\cdot}{\cdot}_{\mathcal{H}}$.
Put
\begin{align*}
\mathcal{E}_{U_{0}}=\left\{\phi :[0,T]\rightarrow U_{0} \,\left|\, \phi(t)=\sum_{i=1}^{m}1_{(t_{i-1},t_{i}]}(t)\varphi_{i},\, 0\leq t_{0}<\cdots <t_{m}\leq T,\, \varphi_{i}\in U_{0}\right.\right\}.
\end{align*}
Let $u\in U_0$ and $t\ge0$ be given. Then for any $m,n\in\mathbb{N}$, we obtain that
\begin{align*}
\mathbb{E}\left[\left|\sum_{j=m}^{n}\sqrt{\lambda_j}\beta_{j}(t)\bilin{u}{e_j}_{U_0}\right|^2\right]
&=R(t,t)\sum_{j=m}^{n}\bilin{u}{\sqrt{\lambda_j}e_j}_{U_0}^2\\
&\to 0\text{ as } m,n\to\infty,
\end{align*}
which implies that the inner product $\bilin{u}{\boldsymbol{\beta}_t}_{U_0}$ is well-defined as follows:
\begin{align*}
\bilin{u}{\boldsymbol{\beta}_t}_{U_0}
=\sum_{j=1}^{\infty}\sqrt{\lambda_j}\beta_{j}(t)\bilin{u}{e_j}_{U_0},
\end{align*}
where the series converges in $L^2(\Omega,P)$.
For $\phi=\sum_{i=1}^{m}1_{(t_{i-1},t_{i}]}\varphi_{i}\in \mathcal{E}_{U_{0}}$, we define
\begin{align*}
\int_{0}^{T}\phi(s)d\boldsymbol{\beta}_{s}
=\sum_{i=1}^{m} \bilin{\varphi_{i}}{\boldsymbol{\beta}_{t_{i}}-\boldsymbol{\beta}_{t_{i-1}}}_{U_{0}},
\end{align*}
and then by direct computation, we see that for any $\phi,\psi\in \mathcal{E}_{U_{0}}$,
\begin{align*}
\mathbb{E}\left[\int_{0}^{T}\phi(s)d\boldsymbol{\beta}_{s}\int_{0}^{T}\psi(s)d\boldsymbol{\beta}_{s}\right]
=\int_{0}^{T}\int_{0}^{T}\bilin{\phi(s)}{\psi(t)}_{U_{0}}\frac{\partial^{2}R}{\partial s\partial t}(s,t)dsdt.
\end{align*}
Define the inner product $\bilin{\cdot}{\cdot}_{\mathcal{H}_{U_{0}}}$ on $\mathcal{E}_{U_{0}}$ by
\begin{align*}
\bilin{\phi}{\psi}_{\mathcal{H}_{U_{0}}}
=\int_{0}^{T}\int_{0}^{T}\bilin{\phi(s)}{\psi(t)}_{U_{0}}\frac{\partial^{2}R}{\partial s\partial t}(s,t)dsdt,
\quad \phi,\psi\in \mathcal{E}_{U_0},
\end{align*}
and let $\mathcal{H}_{U_{0}}$ be the completion of $\mathcal{E}_{U_{0}}$
with respect to $\bilin{\cdot}{\cdot}_{\mathcal{H}_{U_{0}}}$.
Note that $\mathcal{H}_{U_{0}}$ is unitarily isomorphic to $\mathcal{H}\otimes U_{0}$.
For $h\in\mathcal{H}_{U_{0}}$ we denote $\boldsymbol{\beta}(h)=\int_{0}^{T}h(s)d\boldsymbol{\beta}_{s}=\int_{0}^{T}\bilin{h(s)}{d\boldsymbol{\beta}_{s}}_{U_{0}}$,
which is equal to
\begin{align}\label{eqn: series repn for Wiener integral}
\sum_{j=1}^{\infty}\sqrt{\lambda_{j}}\int_{0}^{T}\bilin{h(s)}{e_{j}}_{U_{0}}d\beta_{j}(s),
\end{align}
where $\int_{0}^{T}\bilin{h(s)}{e_{j}}_{U_{0}}d\beta_{j}(s)$ is an analogue of the Wiener integral (see \cite{Tudor2005}).
The process $\boldsymbol{\beta}=\{\boldsymbol{\beta}(\phi)\}_{\phi\in\mathcal{H}_{U_{0}}}$
is a centered Gaussian process with the covariance kernel:
\begin{align}\label{covariance of gaussian process B(phi)}
\mathbb{E}[\boldsymbol{\beta}(\phi)\boldsymbol{\beta}(\psi)]=\bilin{\phi}{\psi}_{\mathcal{H}_{U_{0}}}, \quad\phi,\psi\in \mathcal{H}_{U_{0}}.
\end{align}

If $V_{1},V_{2}$ are separable Hilbert spaces, we denote by $\mathcal{H}_{V_{1}}\otimes \mathcal{H}_{V_{2}}$
the completion of $\mathcal{E}_{V_{1}}\otimes \mathcal{E}_{V_{2}}$ with respect to the inner product $\bilin{\cdot}{\cdot}_{\mathcal{H}_{V_{1}}\otimes \mathcal{H}_{V_{2}}}$ defined by
\begin{align}\label{eqn: inner product on HV otimes HV}
\bilin{\phi}{\psi}_{\mathcal{H}_{V_{1}}\otimes \mathcal{H}_{V_{2}}}
:=\int_{[0,T]^4}\bilin{\phi(s,\theta)}{\psi(t,\eta)}_{V_{1}\otimes V_{2}}
    \frac{\partial^2 R}{\partial s \partial t}(s,t)\frac{\partial^2 R}{\partial \theta \partial \eta}(\theta,\eta)d\theta d\eta dsdt.
\end{align}

If $V$ is a Banach space, then we define $|\mathcal{H}_{V}|$ by the completion of $\mathcal{E}_{V}$ with respect to the norm
\begin{align*}
\|\phi\|^{2}_{|\mathcal{H}_{V}|}:=\int_{0}^{T}\int_{0}^{T}\|\phi(s)\|_{V}\|\phi(t)\|_{V}\frac{\partial^2 R}{\partial s \partial t}(s,t)dsdt.
\end{align*}
In the case of $V=\mathbb{R}$, we write $|\mathcal{H}|=|\mathcal{H}_{V}|$.
For Banach spaces $V_{1}$ and $V_{2}$, we denote $|\mathcal{H}_{V_{1}}|\otimes |\mathcal{H}_{V_{2}}|$
the space of all strongly measurable functions $\phi : [0,T]^{2}\rightarrow V_{1}\otimes V_{2}$
with $\|\phi\|_{|\mathcal{H}_{V_{1}}|\otimes |\mathcal{H}_{V_{2}}|}<\infty$,
where
\begin{align*}
\|\phi\|^2_{|\mathcal{H}_{V_{1}}|\otimes |\mathcal{H}_{V_{2}}|}
:=\int_{[0,T]^4}\|\phi(s,\theta)\|_{V_{1}\otimes V_{2}}\|\phi(t,\eta)\|_{V_{1}\otimes V_{2}}\frac{\partial^2 R}{\partial s \partial t}(s,t)
  \frac{\partial^2 R}{\partial \theta \partial \eta}(\theta,\eta)d\theta d\eta dsdt.
\end{align*}
If $V$, $V_{i}$ for $i=1,2$ are Hilbert spaces, then by direct computation, we see that
\begin{align}
\|\phi\|_{\mathcal{H}_{V}}&\leq \|\phi\|_{\vert\mathcal{H}_{V}\vert},\label{ineq: H norm leq absolute H norm}\\
\|\phi\|_{\mathcal{H}_{V_{1}}\otimes \mathcal{H}_{V_{2}}}&\leq \|\phi\|_{\vert\mathcal{H}_{V_{1}}\vert\otimes \vert\mathcal{H}_{V_{2}}\vert}\label{ineq: Hx H norm leq absolute H x absolute H norm}.
\end{align}

For the covariance kernel $R$, we define the integral kernel operator $K_{R}$ by
\begin{align*}
K_{R}f(t)=\int_{0}^{T}f(s)\frac{\partial^{2}R}{\partial s\partial t}(s,t)ds
\end{align*}
for any real valued measurable function $f$ on $[0,T]$ whenever the integral is well-defined.

\begin{remark}
\upshape
(i) \enspace
Let $1<p,r<\infty$ and let $s$ be the conjugate number of $r$, i.e., $\frac{1}{r}+\frac{1}{s}=1$.
If $\frac{\partial^{2}R}{\partial s\partial t}\in L^{p}([0,T],L^{s}([0,T]))$,
then the integral kernel operator $K_{R}$ is bounded from $L^{r}([0,T])$ into $L^{p}([0,T])$.
In fact, by H\"{o}lder's inequality, we obtain that
      \begin{align*}
      \|K_{R}f\|_{L^{p}([0,T])}
      &=\left(\int_{0}^{T}\left\vert\int_{0}^{T}f(s)\frac{\partial^{2}R}{\partial s\partial t}(s,t)ds\right\vert^{p}dt\right)^{\frac{1}{p}}\\
      &\leq C\|f\|_{L^{r}([0,T])} ~~\text{ with }~~
\quad C=\left(\int_{0}^{T}\left(\int_{0}^{T}
            \left\vert\frac{\partial^{2}R}{\partial s\partial t}(s,t)\right\vert^{s}ds\right)^{\frac{p}{s}}dt\right)^{\frac{1}{p}}.
      \end{align*}
Therefore, the condition \textbf{(R2)} (in Introduction) with $\boldsymbol{r}=r$ and $\boldsymbol{s}=p$ is satisfied.

(ii) \enspace
If $\frac{\partial^{2}R}{\partial s \partial t}\in L^{\infty}([0,T]^{2})$,
then for any $f\in L^{1}([0,T])$, we have
  \begin{align*}
  \|K_{R}f\|_{L^{\infty}([0,T])}\leq \left\|\frac{\partial^{2}R}{\partial s \partial t}\right\|_{L^{\infty}([0,T]^{2})}\|f\|_{L^{1}([0,T])},
  \end{align*}
  i.e., the integral kernel operator $K_{R}$ is bounded from $L^{1}([0,T])$ into $L^{\infty}([0,T])$.
Therefore, the condition \textbf{(R2)} (in Introduction)  with $\boldsymbol{r}=1$ and $\boldsymbol{s}=\infty$ is satisfied.

\end{remark}

Now, we discuss some basic notions of the Malliavin calculus with respect to an isonormal Gaussian process $\{\boldsymbol{\beta}(\phi)\}_{\phi\in\mathcal{H}_{U_{0}}}$.
For the study for the Malliavin calculus with respect to a Wiener process, we refer to \cite{Nualart 2006}.

Let $K$ be a real separable Hilbert space and let $C_{b}^{\infty}(\mathbb{R}^{d})$
be the space of all infinitely differentiable bounded functions on $\mathbb{R}^{d}$
such that all partial derivatives are bounded.
Let $\mathcal{C}_{\boldsymbol{\beta}}(K)$ be the subspace of $L^{2}(\Omega;K)$ of elements $F$ of the form
\begin{align*}
F=\sum_{j=1}^{m}f_{j}(\boldsymbol{\beta}(\phi_{j1}),\cdots,\boldsymbol{\beta}(\phi_{jn_{j}}))k_{j},
\end{align*}
where $m,n_{1},\cdots,n_{m}\in\mathbb{N}$, $k_{1}\cdots,k_{m}\in K$,
\begin{align*}
f_{j}\in C_{b}^{\infty}(\mathbb{R}^{n_{j}}), \quad\phi_{jl}\in \mathcal{H}_{U_{0}}, \quad j=1,\cdots,m.
\end{align*}
In the case of $K=\mathbb{R}$, we write $\mathcal{C}_{\boldsymbol{\beta}}=\mathcal{C}_{\boldsymbol{\beta}}(\mathbb{R})$.

For $F\in\mathcal{C}_{\boldsymbol{\beta}}(K)$, the Malliavin derivative is defined by
\begin{align*}
D^{\boldsymbol{\beta}}F
&=\sum_{j=1}^{m}\sum_{l=1}^{n_{j}}
   \frac{\partial f_{j}}{\partial x_{l}}(\boldsymbol{\beta}(\phi_{j1}),\cdots,\boldsymbol{\beta}(\phi_{jn_{j}}))k_{j}\otimes \phi_{jl},\\
D_{t}^{\boldsymbol{\beta}}F
&=\sum_{j=1}^{m}\sum_{l=1}^{n_{j}}
  \frac{\partial f_{j}}{\partial x_{l}}(\boldsymbol{\beta}(\phi_{j1}),\cdots,\boldsymbol{\beta}(\phi_{jn_{j}}))\phi_{jl}(t)k_{j}
\end{align*}
and, for all $\phi\in \mathcal{H}_{U_{0}}$,
\begin{align*}
D_{\phi}^{\boldsymbol{\beta}}F
:=D^{\boldsymbol{\beta}}F(\phi)
:=\sum_{j=1}^{m}\sum_{l=1}^{n_{j}}\frac{\partial f_{j}}{\partial x_{l}}(\boldsymbol{\beta}(\phi_{j1}),\cdots,\boldsymbol{\beta}(\phi_{jn_{j}}))
 \bilin{\phi_{jl}}{\phi}_{\mathcal{H}_{U_{0}}}k_{j}.
\end{align*}
Then it holds that $D^{\boldsymbol{\beta}}F\in L^{2}(\Omega; K\otimes\mathcal{H}_{U_{0}})$
and $D^{\boldsymbol{\beta}}F(\phi)\in L^{2}(\Omega; K)$.
The operators $D^{\boldsymbol{\beta}}$ is a (unbounded) closable operator
from $L^{2}(\Omega;K)$ into $L^{2}(\Omega;K\otimes\mathcal{H}_{U_{0}})$.
Also, for fixed $\phi\in\mathcal{H}_{U_{0}}$,
$D_{\phi}^{\boldsymbol{\beta}}=D^{\boldsymbol{\beta}}(\cdot)(\phi)$ is a closable operator
from $L^{2}(\Omega;K)$ into $L^{2}(\Omega;K)$ (see, e.g., \cite{Grorud 1992}).

For $p\geq 1$, we denote by $\mathbb{D}_{\boldsymbol{\beta}}^{1,p}(K)$
the closure of $\mathcal{C}_{\boldsymbol{\beta}}(K)$ with respect to the norm
\begin{align*}
\|F\|_{\mathbb{D}_{\boldsymbol{\beta}}^{1,p}(K)}^{p}=\mathbb{E}\|F\|_{K}^{p}+\mathbb{E}\|D^{\boldsymbol{\beta}}F\|_{K\otimes \mathcal{H}_{U_{0}}}^{p}.
\end{align*}

\begin{definition}\label{def: skorohod integral}
\upshape
Let $\boldsymbol{\beta}$ be a $U$-valued $Q$-Gaussian process with the covariance kernel $R$.
\begin{itemize}
  \item [(i)] Let $u\in L^{2}(\Omega;K\otimes\mathcal{H}_{U_{0}})$. We say that $u\in\mathrm{Dom}(\delta^{\boldsymbol{\beta}})$ if there exists a positive constant $C$ such that:
      \[
      \left|\mathbb{E}\left[\bilin{D^{\boldsymbol{\beta}}F}{u}_{K\otimes\mathcal{H}_{U_{0}}}\right]\right|\leq C(\mathbb{E}\left[\|F\|_{K}^{2}\right])^{\frac{1}{2}}, \quad F\in\mathcal{C}_{\boldsymbol{\beta}}(K).
      \]
  \item [(ii)] For $u\in\mathrm{Dom}(\delta^{\boldsymbol{\beta}})$ the map $F\mapsto \mathbb{E}\left[\bilin{D^{\boldsymbol{\beta}}F}{u}_{K\otimes\mathcal{H}_{U_{0}}}\right]$ is linear and continuous from $\mathcal{C}_{\boldsymbol{\beta}}(K)$ into $\mathbb{R}$ for the topology of $L^{2}(\Omega;K)$.
      Then by the Riesz representation theorem, there exists a unique element,
      denoted by $\delta^{\boldsymbol{\beta}}(u)$ and called the Skorohod integral of $u$, in $L^{2}(\Omega;K)$ such that
\begin{align*}
\mathbb{E}\left[\bilin{F}{\delta^{\boldsymbol{\beta}}(u)}_{K}\right]
=\mathbb{E}\left[\bilin{D^{\boldsymbol{\beta}}F}{u}_{K\otimes\mathcal{H}_{U_{0}}}\right],\quad F\in\mathcal{C}_{\boldsymbol{\beta}}(K).
\end{align*}
\end{itemize}
\end{definition}

\begin{remark}
\upshape
The Skorohod integral with respect to the $Q$-Gaussian process defined in Definition \ref{def: skorohod integral}
is motivated by the Skorohod integrals studied in \cite{Grorud 1992} and \cite{Duncan 2006} with respect to the $Q$-Wiener process
and $Q$-fractional Brownian motion in a Hilbert space, respectively.
\end{remark}

Note that $\delta^{\boldsymbol{\beta}}$ is a closed operator
because it is the adjoint of the densely defined (unbounded) operator $D^{\boldsymbol{\beta}}$. For $u\in\mathrm{Dom}(\delta^{\boldsymbol{\beta}})$, we use the notation:
\[
\delta^{\boldsymbol{\beta}}(u)=\int_{0}^{T}u_{t}\delta\boldsymbol{\beta}_{t}.
\]

\begin{lemma}
\begin{itemize}
  \item [\rm{(i)}] For any $F\in\mathcal{C}_{\boldsymbol{\beta}}$ and $\phi\in \mathcal{H}_{U_{0}}$, it holds that
\begin{align}\label{eqn:IF of QD}
\mathbb{E}\left[\boldsymbol{\beta}(\phi)F\right]=\mathbb{E}\left[D_{\phi}^{\boldsymbol{\beta}}F\right].
\end{align}
  \item [\rm{(ii)}] For any $F,G\in\mathcal{C}_{\boldsymbol{\beta}}(K)$ and $\phi\in \mathcal{H}_{U_{0}}$, it holds that
\begin{align}
\mathbb{E}\left[\bilin{D_{\phi}^{\boldsymbol{\beta}}F}{G}_{K}\right]
&=\mathbb{E}\left[\boldsymbol{\beta}(\phi)\bilin{F}{G}_{K}-\bilin{F}{D_{\phi}^{\boldsymbol{\beta}}G}_{K}\right],
\label{eqn:QD}\\
\delta^{\boldsymbol{\beta}}(G\otimes\phi)
&=\boldsymbol{\beta}(\phi)G-D_{\phi}^{\boldsymbol{\beta}}G.\label{eqn:Skorohod integral of G-x-phi}
\end{align}
\end{itemize}
\end{lemma}

\begin{proof}
The proof of (i) is a simple modification of Lemma 1.2.1 in \cite{Nualart 2006}.
The proof of \eqref{eqn:QD} is immediate from the derivation property of $D_{\phi}^{\boldsymbol{\beta}}$ and (i).
The proof of \eqref{eqn:Skorohod integral of G-x-phi} is immediate from \eqref{eqn:QD}
by applying the identity
\begin{align*}
\mathbb{E}\left[\bilin{D^{\boldsymbol{\beta}}F}{G\otimes \phi}_{K\otimes\mathcal{H}_{U_{0}}}\right]
&=\mathbb{E}\left[\bilin{D_{\phi}^{\boldsymbol{\beta}}F}{G}_{K}\right].
\end{align*}
\end{proof}

Let $K$ and $U$ be Hilbert spaces.
Then we identify $K\otimes U$ with $\mathcal{L}_{\rm HS}(U,K)$
whose identification is given by the map
\begin{align*}
\kappa_{K,U}:K\otimes U\ni k\otimes u\mapsto \kappa_{K,U}(k\otimes u)=k\otimes u^{*}\in\mathcal{L}_{\rm HS}(U,K),
\end{align*}
where
\begin{align*}
(k\otimes u^{*})(v)=\langle u,v\rangle k,\quad v\in U.
\end{align*}
For a notational convenience, we write $\kappa_{K}:=\kappa_{K,K}$.
Then for any $k,h\in K$ and $u,v\in U$,
it is obvious that
\begin{align*}
\kappa_{K,U}(k\otimes u)\left(\kappa_{K,U}(h\otimes v)\right)^{*}
&=\left(k\otimes u^{*}\right)\left(v\otimes h^{*}\right)
=\langle u,v\rangle k\otimes h^{*},
\end{align*}
which implies that
\begin{align*}
\langle u,v\rangle k\otimes h
&=\kappa_{K}^{-1}\left(\langle u,v\rangle k\otimes h^{*}\right)
=\kappa_{K}^{-1}\left(\kappa_{K,U}(k\otimes u)\left(\kappa_{K,U}(h\otimes v)\right)^{*}\right).
\end{align*}
Also, we define the operator $S_{U}$ on $K\otimes U\otimes U$ by
\begin{align*}
S_{U}(k\otimes u_{1}\otimes u_{2})=k\otimes u_{2}\otimes u_{1}.
\end{align*}
Since $\mathbb{D}_{\boldsymbol{\beta}}^{1,2}(K\otimes \mathcal{H}_{U_{0}})\subset L^{2}(\Omega;K\otimes \mathcal{H}_{U_{0}})$,
for each $v\in \mathbb{D}_{\boldsymbol{\beta}}^{1,2}(K\otimes \mathcal{H}_{U_{0}})$,
$\kappa_{K,\mathcal{H}_{U_{0}}}(v)$ is considered as an element in
$L^{2}(\Omega;\mathcal{L}_{\rm HS}(\mathcal{H}_{U_{0}},K))$.

\begin{proposition}\label{prop: isometry formula of skorohod}
We have $\mathbb{D}_{\boldsymbol{\beta}}^{1,2}(K\otimes \mathcal{H}_{U_{0}})\subset\mathrm{Dom}(\delta^{\boldsymbol{\beta}})$
and $\delta^{\boldsymbol{\beta}}$ is a bounded linear operator
from $\mathbb{D}_{\boldsymbol{\beta}}^{1,2}(K\otimes \mathcal{H}_{U_{0}})$ into $L^2(\Omega;K)$.
Furthermore, for any $u,v\in\mathbb{D}_{\boldsymbol{\beta}}^{1,2}(K\otimes \mathcal{H}_{U_{0}})$, we have
\begin{align}
\mathbb{E}[\delta^{\boldsymbol{\beta}}(u)\otimes\delta^{\boldsymbol{\beta}}(v)]
&=\mathbb{E}[\kappa_{K}^{-1}(\kappa_{K,\mathcal{H}_{U_{0}}}(u)\kappa_{K,\mathcal{H}_{U_{0}}}(v)^{*})]\nonumber\\
 & \quad+\mathbb{E}\left[\kappa_{K}^{-1}\left(\kappa_{K,\mathcal{H}_{U_{0}}^{\otimes 2}}(D^{\boldsymbol{\beta}}u)
    \left[\kappa_{K,\mathcal{H}_{U_{0}}^{\otimes 2}}\left(S_{\mathcal{H}_{U_{0}}}\left(D^{\boldsymbol{\beta}}v\right)\right)\right]^{*}\right)\right],
    \label{eqn: expectation of tensor product of Skorohod}\\
    \mathbb{E}\left[\bilin{\delta^{\boldsymbol{\beta}}(u)}{\delta^{\boldsymbol{\beta}}(v)}_{K}\right]
&=\mathbb{E}\left[\bilin{u}{v}_{K\otimes\mathcal{H}_{U_{0}}}\right]+\mathbb{E}
\left[\bilin{D^{\boldsymbol{\beta}}u}
{S_{\mathcal{H}_{U_{0}}}\left(D^{\boldsymbol{\beta}}v\right)}_{K\otimes\mathcal{H}_{U_{0}}^{\otimes 2}}\right],
\label{eqn: expectation of inner product of Skorohod}\\
    \mathbb{E}\left[\|\delta^{\boldsymbol{\beta}}(u)\|_{K}^{2}\right]
    &=\mathbb{E}\left[\|u\|_{K\otimes\mathcal{H}_{U_{0}}}^{2}\right]
    +\mathbb{E}\left[\bilin{D^{\boldsymbol{\beta}}u}{S_{\mathcal{H}_{U_{0}}}
    \left(D^{\boldsymbol{\beta}}u\right)}_{K\otimes\mathcal{H}_{U_{0}}^{\otimes 2}}\right].
    \label{eqn: expectation of norm square of Skorohod}
\end{align}
\end{proposition}

\begin{proof}
The proof of the equality given in \eqref{eqn: expectation of tensor product of Skorohod}
is a simple modification of the proof of Proposition 3.2 in \cite{Grorud 1992}.
In fact, for any
\begin{align}\label{eqn:u and v}
u=Fk_1\otimes\phi_1,\quad v=Gk_2\otimes\phi_2\in \mathcal{C}_{\boldsymbol{\beta}}(K\otimes \mathcal{H}_{U_{0}})
\end{align}
with $F,G\in \mathcal{C}_{\boldsymbol{\beta}}$, $k_1,k_2\in K$ and $\phi_1,\phi_2\in \mathcal{H}_{U_{0}}$,
by applying \eqref{eqn:Skorohod integral of G-x-phi} and \eqref{eqn:IF of QD},
we obtain that
\begin{align*}
\mathbb{E}\left[\delta^{\boldsymbol{\beta}}(u)\otimes\delta^{\boldsymbol{\beta}}(v)\right]
&=\mathbb{E}\left[\left(\boldsymbol{\beta}(\phi_1)F-D_{\phi_1}^{\boldsymbol{\beta}}F\right)
  \left(\boldsymbol{\beta}(\phi_2)G-D_{\phi_2}^{\boldsymbol{\beta}}G\right)k_1\otimes k_2\right]\\
&=\mathbb{E}\left[\bilin{\phi_1}{\phi_2}_{\mathcal{H}_{U_{0}}}FG\right]k_1\otimes k_2\\
&\qquad +\mathbb{E}\left[D_{\phi_2}^{\boldsymbol{\beta}}\left(D_{\phi_1}^{\boldsymbol{\beta}}\left(FG\right)\right)
-D_{\phi_1}^{\boldsymbol{\beta}}\left[F(D_{\phi_2}^{\boldsymbol{\beta}}G)\right]
-D_{\phi_2}^{\boldsymbol{\beta}}\left[(D_{\phi_1}^{\boldsymbol{\beta}}F)G\right] \right.\\
&\qquad\left. +(D_{\phi_1}^{\boldsymbol{\beta}}F)(D_{\phi_2}^{\boldsymbol{\beta}}G)\right]k_1\otimes k_2\\
&=\mathbb{E}\left[\bilin{\phi_1}{\phi_2}_{\mathcal{H}_{U_{0}}}FG
+(D_{\phi_2}^{\boldsymbol{\beta}}F)(D_{\phi_1}^{\boldsymbol{\beta}}G)\right]k_1\otimes k_2.
\end{align*}
Thus \eqref{eqn: expectation of tensor product of Skorohod} proved for $u,v$ having form given in \eqref{eqn:u and v}.
Therefore, for any $u,v\in \mathcal{C}_{\boldsymbol{\beta}}(K\otimes \mathcal{H}_{U_{0}})$
having the form given in \eqref{eqn:u and v},
the equality given in \eqref{eqn: expectation of inner product of Skorohod}
can be obtained by taking the trace from \eqref{eqn: expectation of tensor product of Skorohod},
and the equality given in \eqref{eqn: expectation of norm square of Skorohod}
is obtained by taking $u=v$ in \eqref{eqn: expectation of inner product of Skorohod}.
Therefore, for any $u\in \mathcal{C}_{\boldsymbol{\beta}}(K\otimes \mathcal{H}_{U_{0}})$,
by applying \eqref{eqn: expectation of norm square of Skorohod}
and Cauchy-Schwarz inequality, we obtain that
\begin{align*}
\mathbb{E}[\|\delta^{\boldsymbol{\beta}}(u)\|_{K}^{2}]
&\leq \mathbb{E}\left[\|u\|_{K\otimes\mathcal{H}_{U_{0}}}^{2}\right]
  +\mathbb{E}\left[\|D^{\boldsymbol{\beta}}u\|_{K\otimes\mathcal{H}_{U_{0}}^{\otimes 2}}
    \left\|S_{\mathcal{H}_{U_{0}}}\left(D^{\boldsymbol{\beta}}u\right)
            \right\|_{K\otimes\mathcal{H}_{U_{0}}^{\otimes 2}}\right]\\
&=\mathbb{E}\left[\|u\|_{K\otimes\mathcal{H}_{U_{0}}}^{2}\right]
   +\mathbb{E}\left[\|D^{\boldsymbol{\beta}}u\|_{K\otimes\mathcal{H}_{U_{0}}^{\otimes 2}}^{2}\right]\\
&=\|u\|_{\mathbb{D}_{\boldsymbol{\beta}}^{1,2}(K\otimes\mathcal{H}_{U_{0}})}^{2},
\end{align*}
from which we see that $\delta^{\boldsymbol{\beta}}$ is a bounded linear operator
from $\mathcal{C}_{\boldsymbol{\beta}}(K\otimes \mathcal{H}_{U_{0}})$ into $L^2(\Omega;K)$.
Therefore, $\delta^{\boldsymbol{\beta}}$ can be extended to $\mathbb{D}_{\boldsymbol{\beta}}^{1,2}(K\otimes \mathcal{H}_{U_{0}})$
as a bounded linear operator into $L^2(\Omega;K)$,
and so $\mathbb{D}_{\boldsymbol{\beta}}^{1,2}(K\otimes \mathcal{H}_{U_{0}})\subset\mathrm{Dom}(\delta^{\boldsymbol{\beta}})$.
Then by applying the approximation method, for any $u,v\in\mathbb{D}_{\boldsymbol{\beta}}^{1,2}(K\otimes \mathcal{H}_{U_{0}})$,
the equalities given in \eqref{eqn: expectation of tensor product of Skorohod}--\eqref{eqn: expectation of norm square of Skorohod}
can be proved.
\end{proof}

As the Meyer's inequality for processes in a Hilbert space (see \cite{Sugita 1985}),
the following proposition shows that the Skorohod integral $\delta^{\boldsymbol{\beta}}$ is a bounded operator from $\mathbb{D}_{\boldsymbol{\beta}}^{1,p}(K\otimes \mathcal{H}_{U_{0}})$ into $L^{p}(\Omega;K)$.

\begin{proposition}\label{prop:Lp boundedness of Skorohod integral}
Let $p> 1$. Then there exists a constant $C>0$ depending on $p$
such that
\begin{equation}\label{eqn:ineq for continuity of skorohod integral}
\mathbb{E}\left[\|\delta^{\boldsymbol{\beta}}(u)\|_{K}^{p}\right]
\leq C\left(\|\mathbb{E}[u]\|_{K\otimes \mathcal{H}_{U_{0}}}^{p}
+\mathbb{E}\left[\|D^{\boldsymbol{\beta}}u\|_{K\otimes\mathcal{H}_{U_{0}}^{\otimes 2}}^{p}\right]\right)
\end{equation}
for any $u\in\mathbb{D}_{\boldsymbol{\beta}}^{1,p}(K\otimes \mathcal{H}_{U_{0}})$.
\end{proposition}
\begin{proof}
The proof is a modification of the proof of Theorem 3.3 in \cite{Sugita 1985}
(see also Proposition 1.5.4 in \cite{Nualart 2006}).
\end{proof}

We denote by $\mathbb{D}_{\boldsymbol{\beta}}^{1,p}(|\mathcal{H}_{K\otimes U_{0}}|)$
the set of all elements $u\in\mathbb{D}_{\boldsymbol{\beta}}^{1,p}(\mathcal{H}_{K\otimes U_{0}})$
such that $u\in  |\mathcal{H}_{K\otimes U_{0}}|$ a.s.,
$D^{\boldsymbol{\beta}}u \in |\mathcal{H}_{U_{0}}|\otimes|\mathcal{H}_{K\otimes U_{0}}|$ a.s. and $\|u\|_{\mathbb{D}_{\boldsymbol{\beta}}^{1,p}(|\mathcal{H}_{K\otimes U_{0}}|)}<\infty$, where
\begin{align}\label{eqn:D beta absolute}
\|u\|_{\mathbb{D}_{\boldsymbol{\beta}}^{1,p}(|\mathcal{H}_{K\otimes U_{0}}|)}^{p}:=\mathbb{E}\|u\|_{|\mathcal{H}_{K\otimes U_{0}}|}^{p}+\mathbb{E}\|D^{\boldsymbol{\beta}}u\|_{ |\mathcal{H}_{U_{0}}|\otimes|\mathcal{H}_{K\otimes U_{0}}|}^{p}.
\end{align}
\begin{proposition}\label{prop: relation boldbeta  and betaj}
Let $\Phi\in\mathbb{D}_{\boldsymbol{\beta}}^{1,2}(|\mathcal{H}_{K\otimes U_{0}}|)$.
Then for all $t\in [0,T]$, it holds that
\begin{align}
D_{s}^{\boldsymbol{\beta}}\Phi_{t}
&=\sum_{j=1}^{\infty}(D_{s}^{\beta_{j}}\Phi_{t})\otimes\sqrt{\lambda_{j}}e_{j},\label{eqn:MD of Phi-t}\\
\|D_{s}^{\boldsymbol{\beta}}\Phi_{t}\|_{K\otimes U_{0}\otimes U_{0}}^{2}
&=\sum_{j=1}^{\infty}\|D_{s}^{\beta_{j}}\Phi_{t}\|_{K\otimes U_{0}}^{2}\label{eqn: norm of D bold beta}.
\end{align}
Moreover, it holds that
\begin{equation}\label{ineq: beta j and boldbeta}
\sum_{j=1}^{\infty}\mathbb{E}\|D^{\beta_{j}}\Phi\|_{|\mathcal{H}|\otimes|\mathcal{H}_{K\otimes U_{0}}|}^{2}
\leq \mathbb{E}\|D^{\boldsymbol{\beta}}\Phi\|_{|\mathcal{H}_{U_{0}}|\otimes|\mathcal{H}_{K\otimes U_{0}}|}^{2}
\end{equation}
\end{proposition}

\begin{proof}
Let $\Phi=F\phi$ with $F=f(\boldsymbol{\beta}(h_{1}),\cdots,\boldsymbol{\beta}(h_{n}))\in\mathcal{C}_{\boldsymbol{\beta}}$ and $\phi\in\mathcal{H}_{K\otimes U_{0}}$, and let $t\in [0,T]$ be given.
Then by applying \eqref{eqn: series repn for Wiener integral},
we obtain that
\begin{align*}
D_{s}^{\boldsymbol{\beta}}\Phi_{t}
&=\sum_{i=1}^{n}\frac{\partial f}{\partial x_{i}}(\boldsymbol{\beta}(h_{1}),\cdots,\boldsymbol{\beta}(h_{n}))\phi(t)\otimes h_{i}(s)\\
&=\sum_{j=1}^{\infty}\left(\sum_{i=1}^{n}
   \frac{\partial f}{\partial x_{i}}(\boldsymbol{\beta}(h_{1}),\cdots,\boldsymbol{\beta}(h_{n}))
    \sqrt{\lambda_{j}}\bilin{h_{i}(s)}{e_{j}}_{U_{0}}\right)\phi(t)\otimes\sqrt{\lambda_{j}}e_{j}\\
&=\sum_{j=1}^{\infty}(D_{s}^{\beta_{j}}\Phi_{t})\otimes\sqrt{\lambda_{j}}e_{j},
\end{align*}
from which we have \eqref{eqn:MD of Phi-t} and we obtain that
\begin{align*}
\|D_{s}^{\boldsymbol{\beta}}\Phi_{t}\|_{K\otimes U_{0}\otimes U_{0}}^{2}
&=\sum_{j,l=1}^{\infty}\bilin{(D_{s}^{\beta_{j}}\Phi_{t})\otimes\sqrt{\lambda_{j}}e_{j}}
                             {(D_{s}^{\beta_{l}}\Phi_{t})\otimes\sqrt{\lambda_{l}}e_{l}}_{K\otimes U_{0}\otimes U_{0}}\\
&=\sum_{j=1}^{\infty}\|D_{s}^{\beta_{j}}\Phi_{t}\|_{ K\otimes U_{0}}^{2}.
\end{align*}
Therefore, we prove the equality given in \eqref{eqn: norm of D bold beta}.
By applying the Cauchy-Schwarz inequality, we obtain that
\begin{align*}
&\sum_{j=1}^{\infty}\mathbb{E}\|D^{\beta_{j}}\Phi\|_{|\mathcal{H}|\otimes |\mathcal{H}_{K\otimes U_{0}}|}^{2}\\
&=\mathbb{E}\int_{[0,T]^{4}}\sum_{j=1}^{\infty}\|D_{s}^{\beta_{j}}\Phi_{\theta}\|_{K\otimes U_{0}}\|D_{t}^{\beta_{j}}\Phi_{\eta}\|_{ K\otimes U_{0}}\frac{\partial^{2}R}{\partial s\partial t}(s,t)\frac{\partial^{2}R}{\partial \theta\partial \eta}(\theta,\eta)dsdtd\theta d\eta\\
&\leq \mathbb{E}\int_{[0,T]^{4}}\left(\sum_{j=1}^{\infty}\|D_{s}^{\beta_{j}}\Phi_{\theta}\|_{K\otimes  U_{0}}^{2}\right)^{\frac{1}{2}}\left(\sum_{j=1}^{\infty}\|D_{t}^{\beta_{j}}\Phi_{\eta}\|_{K\otimes  U_{0}}^{2}\right)^{\frac{1}{2}}\frac{\partial^{2}R}{\partial s\partial t}(s,t)\frac{\partial^{2}R}{\partial \theta\partial \eta}(\theta,\eta)dsdtd\theta d\eta\\
&=\mathbb{E}\int_{[0,T]^{4}}\|D_{s}^{\boldsymbol{\beta}}\Phi_{\theta}\|_{K\otimes U_{0}\otimes U_{0}}^{2}\|D_{t}^{\boldsymbol{\beta}}\Phi_{\eta}\|_{K\otimes U_{0}\otimes U_{0}}^{2}\frac{\partial^{2}R}{\partial s\partial t}(s,t)\frac{\partial^{2}R}{\partial \theta\partial \eta}(\theta,\eta)dsdtd\theta d\eta\\
&=\mathbb{E}\|D^{\boldsymbol{\beta}}\Phi\|_{|\mathcal{H}_{U_{0}}|\otimes |\mathcal{H}_{ K\otimes U_{0}}|}^{2},
\end{align*}
which implies \eqref{ineq: beta j and boldbeta}.
\end{proof}

\begin{remark}
\rm
{
By the inequality \eqref{ineq: beta j and boldbeta}, for all $j\in\mathbb{N}$ we have
\begin{align*}
\|\Phi\|_{\mathbb{D}_{\beta_{j}}^{1,2}(|\mathcal{H}_{K\otimes U_{0}}|)}^{2}
&=\mathbb{E}\|\Phi\|_{|\mathcal{H}_{K\otimes U_{0}}|}^{2}+\mathbb{E}\|D^{\beta_{j}}\Phi\|_{|\mathcal{H}|\otimes|\mathcal{H}_{K\otimes U_{0}}|}^{2}\\
&\leq \mathbb{E}\|\Phi\|_{|\mathcal{H}_{K\otimes U_{0}}|}^{2}+\mathbb{E}\|D^{\boldsymbol{\beta}}\Phi\|_{|\mathcal{H}_{U_{0}}|\otimes |\mathcal{H}_{K\otimes U_{0}}|}^{2}\\
&=\|\Phi\|_{\mathbb{D}_{\boldsymbol{\beta}}^{1,2}(|\mathcal{H}_{K\otimes U_{0}}|)}^{2}.
\end{align*}
}
\end{remark}

\begin{theorem}
Let $\Phi\in\mathbb{D}_{\boldsymbol{\beta}}^{1,2}(|\mathcal{H}_{K\otimes U_{0}}|)$. Then
\begin{align}\label{eqn:SE of SI}
\int_{0}^{T}\Phi_{s}\delta\boldsymbol{\beta}_{s}
=\sum_{j=1}^{\infty}\int_{0}^{T}\bilin{\Phi_s}{\sqrt{\lambda_j}e_j}_{U_0}\delta\beta_{j}(s).
\end{align}
\end{theorem}

\begin{proof}
By applying the Parseval's identity and the closedness of $\delta^{\boldsymbol{\beta}}$, we have
\begin{align}\label{eqn:SE of SI in proof}
\int_{0}^{T}\Phi_{s}\delta\boldsymbol{\beta}_{s}
&=\delta^{\boldsymbol{\beta}}(\Phi)
=\sum_{j=1}^{\infty}\delta^{\boldsymbol{\beta}}\left(\bilin{\Phi}{\sqrt{\lambda_j}e_j}_{U_0}\otimes \sqrt{\lambda_j}e_j\right).
\end{align}
On the other hand, for any $F\in\mathcal{C}_{\boldsymbol{\beta}}$ and $k\in K$, we obtain that
\begin{align*}
\bilin{\delta^{\boldsymbol{\beta}}\left(\bilin{\Phi}{\sqrt{\lambda_j}e_j}_{U_0}\otimes \sqrt{\lambda_j}e_j\right)}{Fk}
&=\bilin{\bilin{\Phi}{\sqrt{\lambda_j}e_j}_{U_0}\otimes \sqrt{\lambda_j}e_j}{D^{\boldsymbol{\beta}}(Fk)}\\
&=\bilin{\bilin{\Phi}{\sqrt{\lambda_j}e_j}_{U_0}}{D_{\sqrt{\lambda_j}e_j}^{\boldsymbol{\beta}}(Fk)}\\
&=\bilin{\delta^{\beta_j}\left(\bilin{\Phi}{\sqrt{\lambda_j}e_j}_{U_0}\right)}{Fk},
\end{align*}
which implies that for each $j\in\mathbb{N}$,
\begin{align*}
\delta^{\boldsymbol{\beta}}\left(\bilin{\Phi}{\sqrt{\lambda_j}e_j}_{U_0}\otimes \sqrt{\lambda_j}e_j\right)
=\delta^{\beta_j}\left(\bilin{\Phi}{\sqrt{\lambda_j}e_j}_{U_0}\right).
\end{align*}
Therefore, by \eqref{eqn:SE of SI in proof}, we have \eqref{eqn:SE of SI}.
\end{proof}

From  the Assumption \textbf{(R2)}
and the inequalities given in \eqref{ineq: H norm leq absolute H norm}
and \eqref{ineq: Hx H norm leq absolute H x absolute H norm},
by applying the H\"older's inequality, we can see that
\begin{align}
\|\phi\|_{\mathcal{H}_{U_{0}}}
&\leq C_{R}^{\frac{1}{2}}\|\phi\|_{L^{\boldsymbol{r}}([0,T];U_{0})},\label{ineq: HU norm leq Lr norm}\\
\|\psi\|_{\mathcal{H}_{U_{0}}\otimes \mathcal{H}_{K\otimes U_{0}}}
&\leq C_{R}\|\psi\|_{L^{\boldsymbol{r}}([0,T]^{2};K\otimes U_{0}\otimes U_{0})}.\label{ineq: HUxHU norm leq Lr norm}
\end{align}

The following theorem is a generalization of Theorem 2.2 in \cite{Balan 2011}
to the case of $U$-valued $Q$-Gaussian process with the covariance kernel $R$ satisfying \textbf{(R1)} and \textbf{(R2)}.

\begin{theorem}\label{thm : maximal inequality for single gaussian process}
Let $\boldsymbol{r}$ be the constant given in Assumption \textbf{(R2)}.
For any $p>\boldsymbol{r}$, there exists a constant $C>0$ depending on $p,\boldsymbol{r}$ and $T$
such that for any $0<\epsilon<\frac{1}{\boldsymbol{r}}-\frac{1}{p}$ and  $u\in\mathbb{D}_{\boldsymbol{\beta}}^{1,p}(|\mathcal{H}_{K\otimes U_{0}}|)$
for which the right hand side of (\ref{eqn: maximal inequality for single gaussian process}) is finite,
\begin{align}
\mathbb{E}\left[\sup_{t\leq T}\left\|\int_{0}^{t}u_{s}\delta \boldsymbol{\beta}_{s}\right\|_{K}^{p}\right]\nonumber
&\leq C\left\{\left(\int_{0}^{T}\|\mathbb{E}[u_{s}]\|_{K\otimes U_{0}}^{\frac{\boldsymbol{r}}{1-\epsilon \boldsymbol{r}}}ds\right)^{p\left(\frac{1}{\boldsymbol{r}}-\epsilon\right)}\right.\\
&\left.\,\,+\mathbb{E}\left(\int_{0}^{T}\left(\int_{0}^{T}\|D_{\theta}^{\boldsymbol{\beta}}u_{s}\|_{K\otimes U_{0}\otimes U_{0}}^{\boldsymbol{r}}d\theta\right)^{\frac{1}{1-\epsilon \boldsymbol{r}}}ds\right)^{p\left(\frac{1}{\boldsymbol{r}}-\epsilon\right)}\right\}.
\label{eqn: maximal inequality for single gaussian process}
\end{align}
\end{theorem}

\begin{proof}
By applying the inequalities \eqref{eqn:ineq for continuity of skorohod integral},
\eqref{ineq: HU norm leq Lr norm} and \eqref{ineq: HUxHU norm leq Lr norm},
the proof is similar to the proof of Theorem 4 in \cite{Nualart 2003}
(see also the proof of Theorem 2.2 in \cite{Balan 2011}).
In fact, for $\alpha=1-\frac{1}{p}-\epsilon$, by applying the constant $c_{\alpha}$ satisfying
$\frac{1}{c_{\alpha}}=\int_r^t \frac{(\theta-r)^{\alpha-1}}{(t-\theta)^{\alpha}}d\theta$ (for all $r<t$)
and the stochastic Fubini's theorem, we obtain that
\begin{align*}
\int_0^{t}u_s\delta \boldsymbol{\beta}_{s}
&=c_{\alpha}\int_0^{t} u_s\left(\int_{s}^{t}\frac{(r-s)^{\alpha-1}}{(t-r)^{\alpha}}dr\right)\delta \boldsymbol{\beta}_{s}\\
&=c_{\alpha}\int_0^{t}\frac{1}{(t-r)^{\alpha}}\left(\int_{0}^{r} (r-s)^{\alpha-1}u_s\,\delta\boldsymbol{\beta}_{s}\right) dr.
\end{align*}
Then by applying the H\"older's inequality, we can see that
\begin{align*}
\left\|\int_0^{t}u_s\delta \boldsymbol{\beta}_{s}\right\|_{K}^{p}
&\le M_{\alpha,p,T}\int_0^{t}\left\|\int_{0}^{r} (r-s)^{\alpha-1}u_s\,\delta\boldsymbol{\beta}_{s} \right\|_{K}^{p} dr
\end{align*}
for some constant $M_{\alpha,p,T}$, from which we have
\begin{align*}
\mathbb{E}\left[\sup_{0\le t\le T}\left\|\int_0^{t}u_s\delta \boldsymbol{\beta}_{s}\right\|_{K}^{p}\right]
&\le M_{\alpha,p,T}\int_0^{T}\mathbb{E}\left[\left\|\int_{0}^{r}
   (r-s)^{\alpha-1}u_s\,\delta\boldsymbol{\beta}_{s} \right\|_{K}^{p}\right] dr\\
&=M_{\alpha,p,T}\int_0^{T} \mathbb{E}\left[\|\delta^{\boldsymbol{\beta}}\left(v_r\right)\|_{K}^{p}\right] dr,
\end{align*}
where $v_r:=\boldsymbol{1}_{[0,r]}(\cdot)(r-\cdot)^{\alpha-1}u_{(\cdot)}$.
Now by applying \eqref{eqn:ineq for continuity of skorohod integral}, we have
\begin{align}\label{eqn:maximal Lp-estimate in proof}
\mathbb{E}\left[\sup_{0\le t\le T}\left\|\int_0^{t}u_s\delta \boldsymbol{\beta}_{s}\right\|_{K}^{p}\right]
&\le M_{\alpha,p,T} c_{p}(J_1+J_2),
\end{align}
where
\begin{align*}
J_1:=\int_0^{T}\|\mathbb{E}[v_r]\|_{K\otimes \mathcal{H}_{U_{0}}}^{p}dr,
\quad
J_2:=\int_0^{T}\mathbb{E}\left[\|D^{\boldsymbol{\beta}}v_x\|_{K\otimes\mathcal{H}_{U_{0}}\otimes\mathcal{H}_{U_{0}}}^{p}\right]dx.
\end{align*}
Then by applying \eqref{ineq: HU norm leq Lr norm} and \eqref{ineq: HUxHU norm leq Lr norm} to $J_1$ and $J_2$, respectively, we obtain that
\begin{align*}
J_1&\le C_{R}^{\frac{p}{2}}\int_0^{T} \left(\int_0^{x}\|\mathbb{E}[u_{s}]\|_{K\otimes U_{0}}^{\boldsymbol{r}} (x-s)^{\boldsymbol{r}(\alpha-1)}ds\right)^{\frac{p}{\boldsymbol{r}}}dx,\\
J_2&\le C_{R}^{p}\int_0^{T} \left(\int_0^{x}
  \left(\int_0^{T}\|D_{\theta}^{\boldsymbol{\beta}}u_{s}\|_{K\otimes U_{0}\otimes U_{0}}^{\boldsymbol{r}}d\theta\right)
           (x-s)^{\boldsymbol{r}(\alpha-1)}ds\right)^{\frac{p}{\boldsymbol{r}}}dx,
\end{align*}
from which by applying the Hardy–Littlewood inequality
(see p. 119 of \cite{Stein 1970} and the proof of Theorem 4 in \cite{Nualart 2003}),
and finally by applying \eqref{eqn:maximal Lp-estimate in proof},
we prove \eqref{eqn: maximal inequality for single gaussian process}.
\end{proof}

\begin{theorem}\label{thm : cor of max.ineq for single gaussian process}
Let $\boldsymbol{r}$ be the constant given in Assumption \textbf{(R2)}
and let $q\in\mathbb{R}$ with $q>\boldsymbol{r}$.
For any $p\geq q$, there exists a constant $C>0$ depending on $p,\boldsymbol{r},T$ and $q$ such that
for any $u\in\mathbb{D}_{\boldsymbol{\beta}}^{1,p}(|\mathcal{H}_{K\otimes U_{0}}|)$
for which the RHS of (\ref{eqn : cor of max.ineq for single gaussian process}) is finite,
\begin{align}\label{eqn : cor of max.ineq for single gaussian process}
\mathbb{E}\left[\sup_{t\leq T}\left\|\int_{0}^{t}u_{s}\delta \boldsymbol{\beta}_{s}\right\|_{K}^{p}\right]\nonumber
&\leq C\left\{\left(\int_{0}^{T}\|\mathbb{E}[u_{s}]\|_{K\otimes U_{0}}^{q}ds\right)^{\frac{p}{q}}\right.\\
&\,\,\left.+\mathbb{E}\left(\int_{0}^{T}\left(\int_{0}^{T}\|D_{\theta}^{\boldsymbol{\beta}}u_{s}\|_{K\otimes U_{0}\otimes U_{0}}^{\boldsymbol{r}}d\theta\right)^{\frac{q}{\boldsymbol{r}}}ds\right)^{\frac{p}{q}}\right\}.
\end{align}
In particular, we have
\begin{align}\label{eqn: cor maximal inequality for single gaussian process-0}
\mathbb{E}\left[\sup_{t\leq T}\left\|\int_{0}^{t}u_{s}\delta \boldsymbol{\beta}_{s}\right\|_{K}^{p}\right]\nonumber
&\leq C\left\{\int_{0}^{T}\|\mathbb{E}[u_{s}]\|_{K\otimes U_{0}}^{p}ds\right.\\
&\left.\,\,+\mathbb{E}\int_{0}^{T}\left(\int_{0}^{T}\|D_{\theta}^{\boldsymbol{\beta}}u_{s}\|_{K\otimes U_{0}\otimes U_{0}}^{\boldsymbol{r}}d\theta\right)^{\frac{p}{\boldsymbol{r}}}ds\right\}
\end{align}
\end{theorem}

\begin{proof}
Let's take $\epsilon$ such that $0<\epsilon<\frac{1}{\boldsymbol{r}}-\frac{1}{q}$.
Then $q\left(\frac{1-\boldsymbol{r}\epsilon}{\boldsymbol{r}}\right)>1$,
and so by applying the H\"{o}lder's inequality to $\phi \in L^{q}([0.T])$,  we have
\begin{equation*}
\|\phi\|_{L^{\frac{\boldsymbol{r}}{1-\boldsymbol{r}\epsilon}}([0,T])}
\leq C^{(1)}\|\phi\|_{L^{q}([0.T])}
\end{equation*}
for some constant $C^{(1)}=C_{q,\boldsymbol{r},T}^{(1)}>0$,
which is applied to the cases of $\phi (s)=\|\mathbb{E}[u_{s}]\|_{K\otimes U_{0}}$
and $\phi (s)=\left(\int_{0}^{T}\|D_{\theta}^{\boldsymbol{\beta}}u_{s}\|_{K\otimes U_{0}\otimes U_{0}}^{\boldsymbol{r}}d\theta\right)^{1/\boldsymbol{r}}$,
and then we have \eqref{eqn : cor of max.ineq for single gaussian process}
from the inequality \eqref{eqn: maximal inequality for single gaussian process}.
By taking $p=q$ in \eqref{eqn : cor of max.ineq for single gaussian process},
we have \eqref{eqn: cor maximal inequality for single gaussian process-0}.
\end{proof}

\begin{remark}
\rm
{
Theorem \ref{thm : cor of max.ineq for single gaussian process} is a generalization of Corollary 2.3 in \cite{Balan 2011}.
}
\end{remark}

By applying the Jensen's inequality
to the first terms of the RHSs of the inequalities \eqref{eqn : cor of max.ineq for single gaussian process}
 and \eqref{eqn: cor maximal inequality for single gaussian process-0}, we have the following corollary.

\begin{corollary}\label{cor: max ineq for single gaussian process}
Under the assumptions given in Theorem \ref{thm : cor of max.ineq for single gaussian process},
there exists a constant $C=C_{p,\boldsymbol{r},T}>0$ such that
for any $u\in\mathbb{D}_{\boldsymbol{\beta}}^{1,p}(|\mathcal{H}_{K\otimes U_{0}}|)$
for which the RHS of (\ref{eqn: cor maximal inequality for single gaussian process}) is finite,
\begin{align}\label{eqn : cor of max.ineq for single gaussian process+1}
\mathbb{E}\left[\sup_{t\leq T}\left\|\int_{0}^{t}u_{s}\delta \boldsymbol{\beta}_{s}\right\|_{K}^{p}\right]\nonumber
&\leq C\left\{\mathbb{E}\left(\int_{0}^{T}\|u_{s}\|_{K\otimes U_{0}}^{q}ds\right)^{\frac{p}{q}}\right.\\
&\,\,\left.+\mathbb{E}\left(\int_{0}^{T}\left(\int_{0}^{T}\|D_{\theta}^{\boldsymbol{\beta}}u_{s}\|_{K\otimes U_{0}\otimes U_{0}}^{\boldsymbol{r}}d\theta\right)^{\frac{q}{\boldsymbol{r}}}ds\right)^{\frac{p}{q}}\right\}.
\end{align}
In particular, we have
\begin{align}\label{eqn: cor maximal inequality for single gaussian process}
\mathbb{E}\left[\sup_{t\leq T}\left\|\int_{0}^{t}u_{s}\delta \boldsymbol{\beta}_{s}\right\|_{K}^{p}\right]\nonumber
&\leq C\left\{\mathbb{E}\int_{0}^{T}\|u_{s}\|_{K\otimes U_{0}}^{p}ds\right.\\
&\left.\,\,+\mathbb{E}\int_{0}^{T}\left(\int_{0}^{T}\|D_{\theta}^{\boldsymbol{\beta}}u_{s}\|_{K\otimes U_{0}\otimes U_{0}}^{\boldsymbol{r}}d\theta\right)^{\frac{p}{\boldsymbol{r}}}ds\right\}.
\end{align}
\end{corollary}

\section{Fourier Multiplier}\label{sec:Fourier multi}
Let $K$ be a Hilbert space.
We denote by $L_{K}^{p}=L^{p}(\mathbb{R}^{d};K)$
the Banach space of all strongly measurable functions $u:\mathbb{R}^{d}\rightarrow K$ such that
\begin{align*}
\|u\|_{L_{K}^{p}}^{p}=\int_{\mathbb{R}^{d}}\|u(x)\|_{K}^{p}dx<\infty.
\end{align*}
In the case of $K=\mathbb{R}$, we write $L^{p}= L_{\mathbb{R}}^{p}$.

For $f\in L_{K}^{1}$, the Fourier transform and inverse Fourier transform of $f$ is defined by
\begin{align}\label{eqn:FTand FIT}
\mathcal{F}(f)(\xi):=\frac{1}{(2\pi)^{d/2}}\int_{\mathbb{R}^{d}}e^{-i x\cdot\xi}f(x)dx,
\quad\mathcal{F}^{-1}(f)(\xi):=\frac{1}{(2\pi)^{d/2}}\int_{\mathbb{R}^{d}}e^{i x\cdot\xi}f(\xi)d\xi.
\end{align}

\begin{definition}[\cite{Grafakos 2008}, Definition 2.5.11]
\upshape
Let $1\leq p<\infty$. A complex valued bounded function $m$ defined on $\mathbb{R}^{d}$ is called a
Fourier multiplier on $L^{p}$ if
it holds that
\begin{align*}
\|\mathcal{F}^{-1}(m(\xi)\mathcal{F}(f)(\xi))\|_{L^{p}}
\leq C\|f\|_{L^{p}},\quad f\in L^{p}
\end{align*}
for some constant $C>0$.
\end{definition}

Let $\mathbb{N}_{0}=\mathbb{N}\cup \{0\}$. For $\alpha=(\alpha_{1},\cdots,\alpha_{d})\in\mathbb{N}_{0}^{d}$
and $\xi=(\xi_{1},\cdots,\xi_{d})\in\mathbb{R}^{d}$, we denote $|\alpha|=\alpha_{1}+\cdots+\alpha_{d}$ and
$\partial_{\xi}^{\alpha}=\partial_{\xi_{1}}^{\alpha_{1}}\cdots\partial_{\xi_{d}}^{\alpha_{d}}$,
where  $\partial_{\xi_{i}}^{\alpha_{i}}$ is the $\alpha_{i}$-th derivative in the variable $\xi_{i}$.

Note that $\mathbb{R}$ is divided into its two half-lines, and $\mathbb{R}^{2}$ is divided into its four-quadrants,
and generally $\mathbb{R}^{d}$ is divided into its $2^{d}$ ``octants''. Thus the first octant in $\mathbb{R}^{d}$ will be
\[
\{(\xi_{1},\xi_{2},\cdots,\xi_{d})\,|\, \xi_{i}>0 \quad\text{for all}\quad i=1,2,\cdots, d.\}.
\]
We shall assume that $m(\xi)$ is defined on each octant and is
continuous together with its partial derivatives up to and including order $d$. Thus
$m$ may be left undefined on the set of points where one or more coordinate variables vanishes.

\begin{theorem}[Marcinkiewicz multiplier theorem]\label{thm: Lp multiplier}
Let $m$ be a bounded function defined on $\mathbb{R}^{d}$ of type described.
Suppose that
\begin{itemize}
  \item [\rm{(i)}] there exists a constant $C>0$ such that for any $0\leq k\leq d$,
  \begin{align*}
  \sup_{\xi_{k+1},\cdots,\xi_{d}}\int_{A}\left|\partial_{ \xi_{1}}\partial_{\xi_{2}}\cdots\partial_{\xi_{k}}m(\xi)\right|d\xi_{1} d\xi_{2}\cdots d\xi_{k}\leq C
  \end{align*}
  as $A$ ranges over dyadic rectangles of $\mathbb{R}^{k}$, (If $k=d$, the ``sup'' sign is omitted.)
  \item [\rm{(ii)}] The condition analogous to $\rm{(i)}$ is valid for every one of the $d!$ permutations of variables $\xi_{1},\xi_{2},\cdots,\xi_{d}$.
\end{itemize}
Then $m$ is a Fourier multiplier on $L^{p}$.
\end{theorem}

\begin{proof}
See, e.g., Theorem 4.6.6$^\prime$ in p.109 of \cite{Stein 1970}.
\end{proof}

\begin{theorem}\label{thm: Lp multiplier 2}
Let $m$ be a complex-valued bounded function on $\mathbb{R}^{d}\setminus \{0\}$.
Assume that there exists a constant $C>0$ such that
\begin{align*}
\left|\partial_{\xi}^{\alpha}m(\xi)\right|\leq C\prod_{i=1}^{d}|\xi_{i}|^{-\alpha_{i}},
\quad \xi=(\xi_{1},\cdots,\xi_{d})\in\mathbb{R}^{d}\setminus\{0\}
\end{align*}
for all multi-indices $\alpha=(\alpha_{1},\cdots,\alpha_{d})\in\mathbb{N}_{0}^{d}$
with $|\alpha|\leq \lfloor\frac{d}{2}\rfloor+1$, where $\lfloor\eta\rfloor$ is the integer part of $\eta\in\mathbb{R}$.
Then $m$ is a Fourier multiplier on $L^{p}$.
\end{theorem}

\begin{proof}
We check the conditions of Theorem \ref{thm: Lp multiplier}.
For $\alpha=(\alpha_{1},\cdots,\alpha_{d})\in\mathbb{N}_{0}^{d}$,
we assume that $\alpha_{i}=1$ for $i=j_{1},\cdots, j_{n}$
and $\alpha_{i}=0$ for $i\in \{1,\cdots,d\}\setminus J$, where $J=\{j_{1},\cdots, j_{n}\}$.
Then for any dyadic rectangle $A=\prod_{k=1}^{n}[2^{l_{k}},2^{l_{k}+1}]$, we obtain that
\begin{align*}
\int_{A}\left|\partial_{\xi}^{\alpha}m(\xi)\right| d\xi_{j_{1}}\cdots d\xi_{j_{n}}
&\leq C\int_{A}\prod_{k=1}^{n}|\xi_{j_{k}}|^{-1}d\xi_{j_{1}}\cdots d\xi_{j_{n}}\\
&=C\prod_{k=1}^{n}\left(\int_{2^{l_{k}}}^{2^{l_{k}+1}}\xi_{j_{k}}^{-1}d\xi_{j_{k}}\right)\\
&=C(\ln 2)^{n}.
\end{align*}
Thus, if we put $B=\{\xi_{j}\,|\, j\notin J\}$, then we obtain
\begin{align*}
\sup_{B}\int_{A}\left|\partial_{\xi}^{\alpha}m(\xi)\right| d\xi_{j_{1}}\cdots d\xi_{j_{n}}
\leq C (\ln 2)^{n}.
\end{align*}
Hence by Theorem \ref{thm: Lp multiplier}, $m$ is Fourier multiplier on $L^{p}$.
\end{proof}

\begin{example}[\cite{Stein 1970}]
\upshape
Let $a_{1}, a_{2},\cdots,a_{d}>0$ and let $a=a_{1}+a_{2}+\cdots+a_{d}$. For $\xi=(\xi_{1},\xi_{2},\cdots,\xi_{d})\in\mathbb{R}^{d}$, the functions
\begin{align*}
m_{1}(\xi)=\frac{\xi_{1}}{\xi_{1}+i(\xi_{2}^{2}+\xi_{3}^{2}+\cdots+\xi_{d}^{2})},\quad
m_{2}(\xi)
=\frac{|\xi_{1}|^{a_{1}}|\xi_{2}|^{a_{2}}\cdots|\xi_{d}|^{a_{d}}}{(\xi_{1}^2+\xi_{2}^{2}+\cdots+\xi_{d}^{2})^{\frac{a}{2}}}
\end{align*}
satisfy the condition in Theorem \ref{thm: Lp multiplier 2}.
\end{example}

\begin{corollary}[Mihlin's multiplier theorem]\label{cor:Mihlin}
Let $m$ be a complex-valued bounded function on $\mathbb{R}^{d} \setminus \{0\}$.
Assume that there exists a constant $C>0$ such that
\begin{equation}\label{eqn: Mihlin's condition}
|\partial_{\xi}^{\alpha}m(\xi)|\leq C|\xi|^{-|\alpha|},\quad \xi\in\mathbb{R}^{d}\setminus\{0\}
\end{equation}
for all multi-indices $|\alpha| \leq \lfloor\frac{d}{2}\rfloor+1$.
Then $m$ is a Fourier multiplier on $L^{p}$.
\end{corollary}

\begin{proof}
For $\xi=(\xi_{1},\cdots,\xi_{d})$, since $|\xi_{i}|\leq |\xi|$ for all $i=1,\cdots,d$, we have
\begin{align*}
|\partial_{\xi}^{\alpha}m(\xi)|
\leq C|\xi|^{-|\alpha|}
=C|\xi|^{-\alpha_{1}}\cdots|\xi|^{-\alpha_{d}}
\leq C|\xi_{1}|^{-\alpha_{1}}\cdots|\xi_{d}|^{-\alpha_{d}}.
\end{align*}
By Theorem \ref{thm: Lp multiplier 2}, $m$ is a Fourier multiplier on $L^{p}$.
\end{proof}

\begin{remark}\label{rmk: multiplier}
\upshape
(i)\enspace The condition \eqref{eqn: Mihlin's condition} is called the Mihlin's condition. The following condition is called the H\"{o}rmander's condition: there exists a constant $A$ such that
\begin{equation*}
\sup_{R>0}R^{-d+2|\alpha|}\int_{R<|\xi|<2R}|\partial_{\xi}^{\alpha}m(\xi)|^2d\xi
\leq A^2<\infty
\end{equation*}
for all multi-indices $|\alpha|\leq \lfloor \frac{d}{2}\rfloor+1$.
If $m$ satisfies the H\"{o}rmander's condition, then $m$ is a Fourier multiplier on $L^{p}$.
For the proof, see, e.g., Theorem 5.2.7 in \cite{Grafakos 2008}.

(ii)\enspace Let $K$ be a Hilbert space. For $1<p<\infty$, let $m$ be a Fourier multiplier on $L^{p}$.
Then it is obvious that
\begin{align*}
\|\mathcal{F}^{-1}(m(\xi)\mathcal{F}(f)(\xi))\|_{L_{K}^{p}}
\leq C\|f\|_{L_{K}^{p}}, \quad f\in L_{K}^{p}
\end{align*}
for some constant $C>0$. Therefore, Theorems \ref{thm: Lp multiplier}, \ref{thm: Lp multiplier 2}
and Corollary \ref{cor:Mihlin} hold for $L_{K}^{p}$.
\end{remark}

For each $\gamma>0$,
we denote by $\mathfrak{M}_{\gamma}(\mathbb{R}^{d})$
the set of all positive functions $\psi: \mathbb{R}^{d}\rightarrow \mathbb{R}$
satisfying that there exist positive constants $\kappa=\kappa_{\psi}$ and $\mu=\mu_{\psi}$ such that
for any $\xi\in\mathbb{R}^{d}$ with $\xi\neq 0$
and multi index $\alpha\in\mathbb{N}_{0}^{d}$ with $|\alpha|\leq \lfloor\frac{d}{2}\rfloor+1$,
\begin{align}
\psi(\xi)&\geq \kappa |\xi|^{\gamma},\label{eqn: sym cond 1}\\
|\partial_{\xi}^{\alpha}\psi(\xi)|&\leq \mu |\xi|^{\gamma-|\alpha|}.\label{eqn: sym cond 2}
\end{align}

\begin{lemma}\label{lem: xi f in mathfrak m}
Let $f:\mathbb{R}^{d}\rightarrow\mathbb{R}$ be a function.
Assume that there exist positive constants $m$ and $C$ such that
for any $\xi\in\mathbb{R}^{d}$ with $\xi\neq 0$
and multi index $\alpha\in\mathbb{N}_{0}^{d}$ with $|\alpha|\leq \lfloor\frac{d}{2}\rfloor+1$,
\begin{align}
f(\xi)&\geq m,\label{eqn: f geq m}\\
|\partial_{\xi}^{\alpha}f(\xi)|&\leq C|\xi|^{-|\alpha|}.\label{eqn: partial f}
\end{align}
Let $\gamma>0$ be given. Then the function $\psi(\xi)=|\xi|^{\gamma}f(\xi)\,\,(\xi\in\mathbb{R}^{d})$
is in $\mathfrak{M}_{\gamma}(\mathbb{R}^{d})$.
\end{lemma}

\begin{proof}
By \eqref{eqn: f geq m}, we have
\begin{align*}
\psi(\xi)=|\xi|^{\gamma}f(\xi)\geq m|\xi|^{\gamma},
\end{align*}
which implies \eqref{eqn: sym cond 1}. On the other hand, by direct computation,
we see that there exists a constant $C_{1}>0$ such that
for any multi index $\alpha\in\mathbb{N}_{0}^{d}$ with $|\alpha|\leq \lfloor\frac{d}{2}\rfloor+1$,
\begin{align}\label{eqn: partial norm xi}
|\partial_{\xi}^{\alpha}|\xi|^{\gamma}|&\leq C_{1} |\xi|^{\gamma-|\alpha|}
\end{align}
Therefore, by applying the Leibniz rule, \eqref{eqn: partial norm xi} and \eqref{eqn: partial f}, we obtain that
\begin{align*}
|\partial_{\xi}^{\alpha}\psi(\xi)|
&\leq \sum_{\beta\leq \alpha}\binom{\alpha}{\beta}
        |\partial_{\xi}^{\beta}|\xi|^{\gamma}||\partial_{\xi}^{\alpha-\beta}f(\xi)|\\
&\leq \sum_{\beta\leq \alpha}\binom{\alpha}{\beta}
        C_{1} |\xi|^{\gamma-|\beta|}C |\xi|^{-|\alpha-\beta|}\\
&=\left(\sum_{\beta\leq \alpha}\binom{\alpha}{\beta}C_{1}C\right) |\xi|^{\gamma-|\alpha|},
\end{align*}
which implies \eqref{eqn: sym cond 2}. Hence $\psi\in\mathfrak{M}_{\gamma}(\mathbb{R}^{d})$.
\end{proof}

\begin{example}
\upshape
Let $a>0,\,\,b\geq 0$ and $\gamma>0$ be given.
Put
\begin{align*}
\psi(\xi)=|\xi|^{\gamma}(a+be^{-|\xi|^{\gamma}}),\quad \xi\in\mathbb{R}^{d}.
\end{align*}
Then by direct computation, we see that the function $f(\xi)=a+be^{-|\xi|^{\gamma}} (\xi\in\mathbb{R}^{d})$
satisfies the conditions \eqref{eqn: f geq m} and \eqref{eqn: partial f}. Therefore, by Lemma \ref{lem: xi f in mathfrak m},
it holds that $\psi\in\mathfrak{M}_{\gamma}(\mathbb{R}^{d})$.
\end{example}

\begin{lemma}\label{lmm:FB}
Let $\gamma>0$ and let $\psi\in \mathfrak{M}_{\gamma}(\mathbb{R}^{d})$.
Then for any $s\in\mathbb{R}$,
there exists a constant $C>0$ such that for any $\xi\in\mathbb{R}^{d}$ with $\xi\neq 0$
and multi-index $\alpha\in\mathbb{N}_{0}^{d}$ with $|\alpha|\leq \lfloor\frac{d}{2}\rfloor+1$,
\begin{align}\label{eqn: partial alpha m1}
|\partial_{\xi}^{\alpha}\left((1+\psi(\xi))^{s}\right)|
\leq C|\xi|^{-|\alpha|}(1+\psi(\xi))^{s}.
\end{align}
\end{lemma}

\begin{proof}
If $s=0$, then the inequality is obvious.
Let $s=-t<0$ be fixed for some $t>0$.
It is proved by induction on $|\alpha|$. For $|\alpha|=1$, we have
\begin{align*}
\partial_{\xi_{j}}(1+\psi(\xi))^{-t}=-t(1+\psi(\xi))^{-t-1}\partial_{\xi_{j}}\psi(\xi),
\end{align*}
from which, \eqref{eqn: sym cond 1} and \eqref{eqn: sym cond 2} imply that
\begin{align*}
|\partial_{\xi_{j}}(1+\psi(\xi))^{-t}|
&=t|(1+\psi(\xi))^{-t-1}||\partial_{\xi_{j}}\psi(\xi)|\\
&\leq t\frac{1}{(1+\psi(\xi))^{t}}\frac{\mu |\xi|^{\gamma}}{1+\kappa |\xi|^{\gamma}}|\xi|^{-1}\\
&\leq t\frac{\mu}{\kappa}\frac{|\xi|^{-1}}{(1+\psi(\xi))^{t}}.
\end{align*}
Suppose that \eqref{eqn: partial alpha m1} holds for all $|\alpha|\leq k$.
Denote $e_{j}$ by $(0,\cdots,1,\cdots,0)$ with the 1 in $j$-th position.
If $|\alpha|=k+1$, then there exist multi-indices $\beta$ and $\delta$ such that $|\beta|=k$, $|\delta|=1$
and $\alpha=\beta+\delta$. Suppose that $\delta=e_{j}$ for some $j$.
By using the induction hypothesis, \eqref{eqn: sym cond 1}  and \eqref{eqn: sym cond 2},
we obtain that
  \begin{align*}
  |\partial_{\xi}^{\alpha}(1+\psi(\xi))^{-t}|
  &=|\partial_{\xi}^{\beta}\partial_{\xi}^{\delta}(1+\psi(\xi))^{-t}|\nonumber\\
  &=|\partial_{\xi}^{\beta}(-t(1+\psi(\xi))^{-t-1}\partial_{\xi}^{\delta}\psi(\xi))|\nonumber\\
  &=\left|-t\sum_{\sigma\leq\beta}\binom{\beta}{\sigma}
  (\partial_{\xi}^{\beta-\sigma}(1+\psi(\xi))^{-t-1})(\partial_{\xi}^{\sigma}\partial_{\xi}^{\delta}\psi(\xi))\right|\nonumber\\
  &\leq t\sum_{\sigma\leq\beta}\binom{\beta}{\sigma}
  |\partial_{\xi}^{\beta-\sigma}(1+\psi(\xi))^{-t-1}||\partial_{\xi}^{\sigma}\partial_{\xi}^{\delta}\psi(\xi)|\nonumber\\
  &\leq t\sum_{\sigma\leq\beta}\binom{\beta}{\sigma}
  C\frac{|\xi|^{-(|\beta|-|\sigma|)}}{(1+\psi(\xi))^{t+1}}\mu|\xi|^{\gamma-(|\sigma|+1)}\nonumber\\
  &\leq tC\frac{\mu}{\kappa}\sum_{\sigma\leq\beta}\binom{\beta}{\sigma}
  \frac{|\xi|^{-(|\beta|+1)}}{(1+\psi(\xi))^{t}}\frac{\kappa|\xi|^{\gamma}}{1+\kappa |\xi|^{\gamma}}\nonumber\\
  &\le C'  \frac{|\xi|^{-|\alpha|}}{(1+\psi(\xi))^{t}},
  \end{align*}
which implies \eqref{eqn: partial alpha m1} for $s<0$.
The proof for the case $s>0$ is similar.
\end{proof}

\begin{proposition}\label{prop: multiplier}
Let $\gamma>0$ and let $\psi\in \mathfrak{M}_{\gamma}(\mathbb{R}^{d})$.
Let $\varphi:\mathbb{R}^{d}\rightarrow \mathbb{R}$ be a positive function
satisfying the condition \eqref{eqn: sym cond 2} with $\gamma$.
Then for any $s\ge t\ge0$, the functions
\begin{align*}
m_{1}(\xi)=(1+\psi(\xi))^{-s},\quad
m_{2}(\xi)=\frac{\varphi(\xi)^{t}}{(1+\psi(\xi))^{s}},\quad
m_{3}(\xi)=\frac{(1+\varphi(\xi))^{t}}{1+\psi(\xi)^{s}},\quad \xi\in\mathbb{R}^{d}
\end{align*}
are Fourier multipliers on $L_{K}^{p}$ for all $1<p<\infty$.
\end{proposition}

\begin{proof}
Let $1<p<\infty$ be given.
Firstly we prove that the function $m_{1}$ is a Fourier multiplier on $L_{K}^{p}$.
Since $(1+\psi(\xi))^{-s}\leq 1$, by applying \eqref{eqn: partial alpha m1},
we have
\begin{align*}
|\partial_{\xi}^{\alpha}m_{1}(\xi)|\leq C_{1}|\xi|^{-|\alpha|}
\end{align*}
for any multi index $\alpha\in\mathbb{N}_{0}^{d}$ with $|\alpha|\leq \lfloor\frac{d}{2}\rfloor+1$.
Therefore, by applying Corollary \ref{cor:Mihlin} (see also, (ii) in Remark \ref{rmk: multiplier}),
we see that $m_{1}$ is a Fourier multiplier on $L_{K}^{p}$.

We now prove that the function $m_{2}$ is a Fourier multiplier on $L_{K}^{p}$.
We claim that there exists a constant $C_{1}>0$ such that for any $\xi\in\mathbb{R}^{d}$ with $\xi\neq 0$ and multi-index $\alpha\in\mathbb{N}_{0}^{d}$ with $|\alpha|\leq \lfloor\frac{d}{2}\rfloor+1$,
\begin{align}
|\partial_{\xi}^{\alpha}m_{2}(\xi)|\leq C_{1}|\xi|^{-|\alpha|}.
\end{align}
By \eqref{eqn: sym cond 2} and direct computation, we have
\begin{align}\label{eqn: partial alpha psi s}
|\partial_{\xi}^{\alpha}\left(\varphi(\xi)^{t}\right)|\leq C'|\xi|^{\gamma t-|\alpha|}.
\end{align}
for some constant $C'>0$.
By applying the Leibniz rule, \eqref{eqn: partial alpha m1} and \eqref{eqn: sym cond 1}, we obtain that
\begin{align*}
  |\partial_{\xi}^{\alpha}m_{2}(\xi)|
  &=\left|\sum_{\beta\leq \alpha}\binom{\alpha}{\beta}\left(\partial_{\xi}^{\beta}\left(\varphi(\xi)^{t}\right)\right)
    \left(\partial_{\xi}^{\alpha-\beta}(1+\psi(\xi))^{-s}\right)\right|\\
  &\leq \sum_{\beta\leq \alpha}\binom{\alpha}{\beta}C'|\xi|^{\gamma t-|\beta|}C\frac{|\xi|^{-(|\alpha|-|\beta|)}}{(1+\psi(\xi))^{s}}\\
  &\leq C'C\sum_{\beta\leq \alpha}\binom{\alpha}{\beta}\frac{|\xi|^{\gamma t}}{(1+\kappa|\xi|^{\gamma})^{s}}|\xi|^{-|\alpha|}\\
  &=:C_{1}|\xi|^{-|\alpha|}.
  \end{align*}
Therefore, by applying Corollary \ref{cor:Mihlin},
we see that $m_{2}$ is a Fourier multiplier on $L_{K}^{p}$ for all $1<p<\infty$.

Finally we prove that the function $m_{3}$ is a Fourier multiplier on $L_{K}^{p}$.
We claim that there exists a constant $C_{2}>0$ such that for any $\xi\in\mathbb{R}^{d}$ with $\xi\neq 0$ and multi-index $\alpha\in\mathbb{N}_{0}^{d}$ with $|\alpha|\leq \lfloor\frac{d}{2}\rfloor+1$,
\begin{align}
|\partial_{\xi}^{\alpha}m_{3}(\xi)|\leq C_{2}|\xi|^{-|\alpha|}.
\end{align}
From \eqref{eqn: partial alpha psi s}, we see that the function
$\psi_{1}(\xi):=\psi(\xi)^{s}$ ($\xi\in\mathbb{R}^{d}$)
satisfies \eqref{eqn: sym cond 2}. Also, obviously, $\psi_{1}$ satisfies \eqref{eqn: sym cond 1}
with $\gamma_{\psi_{1}}=\gamma_{\psi} s$ and $\kappa_{\psi_{1}}=\kappa_{\psi}^{s}$.
Then by applying \eqref{eqn: partial alpha m1}, we have
\begin{align}\label{eqn: partial alpha m3 2}
|\partial_{\xi}^{\alpha}(1+\psi(\xi)^{s})^{-1}|
\leq C'(1+\psi(\xi)^{s})^{-1}|\xi|^{-|\alpha|}
\end{align}
for some constant $C'>0$. Therefore, by applying the Leibniz rule,
\eqref{eqn: partial alpha m1} and \eqref{eqn: partial alpha m3 2}, we obtain that
\begin{align*}
|\partial_{\xi}^{\alpha}m_{3}(\xi)|
&=\left|\sum_{\beta\leq \alpha}\binom{\alpha}{\beta}\left(\partial_{\xi}^{\beta}\left((1+\varphi(\xi))^{t}\right)\right)
  \left(\partial_{\xi}^{\alpha-\beta}(1+\psi(\xi)^{s})^{-1}\right)\right|\\
&\leq \sum_{\beta\leq \alpha}\binom{\alpha}{\beta}C\left(1+\varphi(\xi)\right)^{t}|\xi|^{-|\beta|}
        C'(1+\psi(\xi)^{s})^{-1}|\xi|^{-(|\alpha|-|\beta|)}\\
&\leq CC'\sum_{\beta\leq \alpha}\binom{\alpha}{\beta}\frac{\left(1+\varphi(\xi)\right)^{t}}{1+\psi(\xi)^{s}}|\xi|^{-|\alpha|}\\
&=:C_{2}|\xi|^{-|\alpha|}.
\end{align*}
Hence, by applying Corollary \ref{cor:Mihlin},
we see that $m_{3}$ is a Fourier multiplier on $L_{K}^{p}$.
\end{proof}

\section{Stochastic Banach Spaces}\label{sec:stochastic Banach}
In this section, we investigate some properties of the stochastic Banach spaces
with respect to a $Q$-Gaussian process with the covariance kernel $R$.
We first recall the basic definitions and results for the vector-valued tempered distributions
and the Bessel potential spaces.
For more details, we refer to \cite{Hytonen 2016} and \cite{Stein 1970}.

Let $(K,\bilin{\cdot}{\cdot}_{K})$ be a separable Hilbert space with an inner product.
We denote by $C_{\mathrm{c},K}^{\infty}=C_{\rm c}^{\infty}(\mathbb{R}^{d};K)$
the space of infinitely differentiable $K$-valued functions on $\mathbb{R}^{d}$ with compact support.
Let $\mathcal{D}_{K}:=\mathcal{D}(\mathbb{R}^{d};K)$ be the space of $K$-valued tempered distributions on $C_{\mathrm{c},K}^{\infty}$,
that is, $\mathcal{D}_{K}$ is the space of all continuous linear operators from $C_{\mathrm {c},\mathbb{R}}^{\infty}$ into $K$.
In the case of $K=\mathbb{R}$, we write
$C_{\rm c}^{\infty}= C_{\mathrm{c},\mathbb{R}}^{\infty}$ and $\mathcal{D}=\mathcal{D}_{\mathbb{R}}$.

Let $\psi:\mathbb{R}^{d}\rightarrow \mathbb{R}$ be a positive measurable function.
For $f\in L_{K}^{1}$, the pseudo differential operator $L_{\psi}$ with the symbol $\psi$ is defined by
\begin{align*}
L_{\psi}f(x)=\mathcal{F}^{-1}(\psi(\xi)\mathcal{F}f(\xi))(x)
\end{align*}
if the right hand side exists. For $\alpha\in\mathbb{R}$, we define
\begin{align*}
L_{\psi}^{\alpha}:=L_{\psi^{\alpha}}.
\end{align*}
For $p>1$ and $\alpha\in\mathbb{R}$, the $\psi$-Bessel potential space $H_{K,p}^{\psi,\alpha}=H_{p}^{\psi,\alpha}(\mathbb{R}^{d};K)$
is defined as the closure of $C_{\mathrm{c},K}^{\infty}$ with respect to the norm
\begin{align*}
\|u\|_{H_{K,p}^{\psi,\alpha}}
:=\|(1+L_{\psi})^{\frac{\alpha}{2}}u\|_{L_{K}^{p}}
:=\|\mathcal{F}^{-1}[(1+\psi(\xi))^{\frac{\alpha}{2}}\mathcal{F}(u)(\xi)]\|_{L_{K}^{p}}.
\end{align*}
In the case of $K=\mathbb{R}$, we write
$H_{p}^{\psi, \alpha}=H_{\mathbb{R}, p}^{\psi,\alpha}$. If $\psi(\xi)=|\xi|^2$ ($\xi\in\mathbb{R}^{d}$),
then it holds that $L_{\psi}=-\Delta$ and so $H_{K,p}^{\psi, \alpha}$ is the classical Bessel potential space.

\begin{proposition}\label{prop: psi bessel}
Let $\psi\in \mathfrak{M}_{\gamma}$ for some $\gamma>0$.
Let $\alpha\in\mathbb{R}$ and let $p>1$. Then it holds that
\begin{itemize}
  \item [\rm (i)] for any $\theta\in\mathbb{R}$, the operator
  $(1+L_{\psi})^{\frac{\alpha}{2}}$
  is an isometry from $H_{K,p}^{\psi,\theta}$ into $H_{K,p}^{\psi,\theta-\alpha}$,

  \item [\rm (ii)] for any $\epsilon\geq 0$,
  $H_{K,p}^{\psi,\alpha+\epsilon}$ is continuously embedded into $H_{K,p}^{\psi,\alpha}$,

\item [\rm (iii)] if $\alpha\geq 0$, then there exists a constant $C_{1},C_{2}>0$ such that for any $u\in H_{p}^{\psi,\alpha}(K)$,
  \begin{align}
  C_{1}(\|u\|_{L_{K}^{p}}+\|L_{\psi}^{\frac{\alpha}{2}}u\|_{L_{K}^{p}})
  \leq \|u\|_{H_{K,p}^{\psi,\alpha}}
  \leq C_{2}(\|u\|_{L_{K}^{p}}+\|L_{\psi}^{\frac{\alpha}{2}}u\|_{L_{K}^{p}}).\label{eqn: equiv. norm Bessel}
  \end{align}
\end{itemize}
\end{proposition}
\begin{proof}
(i)\enspace It is obvious by the definition of $H_{K,p}^{\psi,\alpha}$.

(ii)\enspace By (i), we may assume that $\alpha=0$.
By Proposition \ref{prop: multiplier},
the function $m(\xi)=(1+\psi(\xi))^{-\frac{\epsilon}{2}}$ is a Fourier multiplier on $L_{K}^{p}$
and so for any $u\in H_{K,p}^{\psi,\epsilon}$, we obtain that
\begin{align}
\|u\|_{L_{K}^{p}}
&=\|\mathcal{F}^{-1}\left((1+\psi(\xi))^{-\frac{\epsilon}{2}}
            (1+\psi(\xi))^{\frac{\epsilon}{2}}\mathcal{F}u(\xi)\right)\|_{L_{K}^{p}}\nonumber\\
&\leq C\|\mathcal{F}^{-1}\left((1+\psi(\xi))^{\frac{\epsilon}{2}}\mathcal{F}u(\xi)\right)\|_{L_{K}^{p}}\nonumber\\
&=C\|u\|_{H_{K,p}^{\psi,\epsilon}}\label{eqn: embed bessel}
\end{align}
for some constant $C>0$, which proves the desired result.

(iii)\enspace Firstly we prove the first inequality of \eqref{eqn: equiv. norm Bessel}.
By Proposition \ref{prop: multiplier},
the function $m(\xi)=\frac{\psi(\xi)^{\frac{\alpha}{2}}}{(1+\psi(\xi))^{\frac{\alpha}{2}}}$
is a Fourier multiplier on $L_{K}^{p}$ and so we obtain that
\begin{align}
\|L_{\psi}^{\frac{\alpha}{2}}u\|_{L_{K}^{p}}
&=\left\|\mathcal{F}^{-1}\left(\frac{\psi(\xi)^{\frac{\alpha}{2}}}{(1+\psi(\xi))^{\frac{\alpha}{2}}}
(1+\psi(\xi))^{\frac{\alpha}{2}}\mathcal{F}u(\xi)\right)\right\|_{L_{K}^{p}}\nonumber\\
&\leq C'\left\|\mathcal{F}^{-1}\left((1+\psi(\xi))^{\frac{\alpha}{2}}\mathcal{F}u(\xi)\right)\right\|_{L_{K}^{p}}\nonumber\\
&=C'\left\|u\right\|_{H_{K,p}^{\psi,\alpha}}\label{eqn: L alpha bdd Lp}
\end{align}
for some constant $C'>0$. Therefore, by applying \eqref{eqn: embed bessel} and \eqref{eqn: L alpha bdd Lp}, we have
\begin{align*}
\|u\|_{L_{K}^{p}}+\|L_{\psi}^{\frac{\alpha}{2}}u\|_{L_{K}^{p}}
\leq (C+C')\left\|u\right\|_{H_{K,p}^{\psi,\alpha}},
\end{align*}
which gives the first inequality of \eqref{eqn: equiv. norm Bessel}.

We now prove the second inequality of \eqref{eqn: equiv. norm Bessel}.
By Proposition \ref{prop: multiplier},
the function $m(\xi)=\frac{(1+\psi(\xi))^{\frac{\alpha}{2}}}{1+\psi(\xi)^{\frac{\alpha}{2}}}$
is a Fourier multiplier on $L_{K}^{p}$ and so we obtain that
\begin{align*}
\|u\|_{H_{K,p}^{\psi,\alpha}}
&=\left\|\mathcal{F}^{-1}\left((1+\psi(\xi))^{\frac{\alpha}{2}}\mathcal{F}u(\xi)\right)\right\|_{L_{K}^{p}}\\
&=\left\|\mathcal{F}^{-1}\left(\frac{(1+\psi(\xi))^{\frac{\alpha}{2}}}{1+\psi(\xi)^{\frac{\alpha}{2}}}
\left(1+\psi(\xi)^{\frac{\alpha}{2}}\right)\mathcal{F}u(\xi)\right)\right\|_{L_{K}^{p}}\\
&\leq \left\|\mathcal{F}^{-1}\left(\frac{(1+\psi(\xi))^{\frac{\alpha}{2}}}{1+\psi(\xi)^{\frac{\alpha}{2}}}
\mathcal{F}u(\xi)\right)\right\|_{L_{K}^{p}}\\
&\qquad+\left\|\mathcal{F}^{-1}\left(\frac{(1+\psi(\xi))^{\frac{\alpha}{2}}}{1+\psi(\xi)^{\frac{\alpha}{2}}}
\psi(\xi)^{\frac{\alpha}{2}}\mathcal{F}u(\xi)\right)\right\|_{L_{K}^{p}}\\
&\leq C''\left\|u\right\|_{L_{K}^{p}}
+C''\left\|\mathcal{F}^{-1}\left(\psi(\xi)^{\frac{\alpha}{2}}\mathcal{F}u(\xi)\right)\right\|_{L_{K}^{p}}\\
&=C''\left(\left\|u\right\|_{L_{K}^{p}}+\left\|L_{\psi}^{\frac{\alpha}{2}}u\right\|_{L_{K}^{p}}\right)
\end{align*}
for some constant $C''>0$, which implies the second inequality of \eqref{eqn: equiv. norm Bessel}.
\end{proof}

For $u\in H_{K,p}^{\psi, \alpha}$ and $\varphi\in C_{\mathrm{c},K}^{\infty}$, the duality is defined by
  \begin{align*}
(u,\varphi)
=\int_{\mathbb{R}^{d}}\bilin{[(1+L_{\psi})^{\frac{\alpha}{2}}u](x)}{[(1+L_{\psi})^{-\frac{\alpha}{2}}\varphi](x)}_{K} dx.
  \end{align*}
By applying the H\"{o}lder's inequality to $(u,\varphi)$,  we have
\begin{equation}\label{eqn:inequality for duality}
|(u,\varphi)|\leq N\|u\|_{H_{K,p}^{\psi,\alpha}},
\end{equation}
where $N=\|(1+L_{\psi})^{-\frac{\alpha}{2}}\varphi\|_{L_{K}^{\frac{p}{p-1}}}$.

Let $\mathcal{C}_{\boldsymbol{\beta}}(\mathcal{E}_{C_{\mathrm{c},K}^{\infty}})$
be the set of all smooth processes of the form
\begin{equation*}
g(t,\cdot)=\sum_{i=1}^{m}F_{i}1_{(t_{i-1},t_{i}]}(t)\phi_{i}(\cdot),\quad t\in [0,T],
\end{equation*}
where $F_{i}\in\mathcal{C}_{\boldsymbol{\beta}}$,
$0\leq t_{0}<t_{1}<\cdots<t_{m}\leq T$, and $\phi_{i}\in C_{\mathrm{c},K}^{\infty}$.

\begin{definition}
\upshape
Let $p>1$, $\alpha\in\mathbb{R}$, $T>0$ and let $X$ be a separable Hilbert space.
\begin{itemize}

  \item[(a)]
  For an element $\Phi\in\mathcal{C}_{\boldsymbol{\beta}}(\mathcal{E}_{C_{\mathrm{c},X}^{\infty}})$, we define a norm
  \begin{align*}
  \|\Phi\|_{\mathbb{D}_{\boldsymbol{\beta}}^{1,p}(|\mathcal{H}_{H_{X,p}^{\psi,\alpha}}|)}^{p}
  :=\mathbb{E}\|\Phi\|_{|\mathcal{H}_{H_{X,p}^{\psi,\alpha}}|}^{p}
        +\mathbb{E}\|D^{\boldsymbol{\beta}}\Phi\|_{|\mathcal{H}_{U_{0}}|\otimes|\mathcal{H}_{H_{X,p}^{\psi,\alpha}}|}^{p}.
  \end{align*}
  We denote by $\mathbb{D}_{\boldsymbol{\beta}}^{1,p}(|\mathcal{H}_{H_{X,p}^{\psi,\alpha}}|)$ the completion of $\mathcal{C}_{\boldsymbol{\beta}}(\mathcal{E}_{C_{\mathrm{c},X}^{\infty}})$ with respect to the norm $\|\cdot\|_{\mathbb{D}_{\boldsymbol{\beta}}^{1,p}(|\mathcal{H}_{H_{X,p}^{\psi,\alpha}}|)}$.

  \item[(b)] We denote by $\mathbb{L}_{\boldsymbol{r},\boldsymbol{\beta}}^{1,p}(H_{X,p}^{\psi,\alpha},T)$ the set of all elements $\Phi\in\mathbb{D}_{\boldsymbol{\beta}}^{1,p}(|\mathcal{H}_{H_{X,p}^{\psi,\alpha}}|)$ such that
\begin{align*}
\|\Phi\|_{\mathbb{L}_{\boldsymbol{r},\boldsymbol{\beta}}^{1,p}(H_{X,p}^{\psi,\alpha},T)}^{p}
&:=\mathbb{E}\!\!\int_{0}^{T}\|\Phi(s,\cdot)\|_{H_{X,p}^{\psi,\alpha}}^{p}ds\\
 &\qquad +\mathbb{E}\int_{0}^{T}\left(\!\!\int_{0}^{T}
                \|D_{\theta}^{\boldsymbol{\beta}}\Phi(s,\cdot)\|_{H_{X\otimes U_{0},p}^{\psi,\alpha}}^{\boldsymbol{r}}d\theta\right)^{\frac{p}{\boldsymbol{r}}}ds<\infty.
\end{align*}

  \item[(c)] We denote by $\widetilde{\mathbb{L}}_{\boldsymbol{r},\boldsymbol{\beta}}^{1,p}(H_{X,p}^{\psi,\alpha},T)$ the completion of $\mathcal{C}_{\boldsymbol{\beta}}(\mathcal{E}_{C_{\mathrm{c},X}^{\infty}})$ in $\mathbb{D}_{\boldsymbol{\beta}}^{1,p}(|\mathcal{H}_{H_{X,p}^{\psi,\alpha}}|)$ with respect to the norm $\|\cdot\|_{\mathbb{L}_{\boldsymbol{r},\boldsymbol{\beta}}^{1,p}(H_{X,p}^{\psi,\alpha},T)}$.
\end{itemize}
\end{definition}

\begin{lemma}\label{lmm:commutativity between D and 1-Delta}
Let $p>1$ and $\alpha\in\mathbb{R}$
and let $\Phi\in\mathbb{D}_{\boldsymbol{\beta}}^{1,p}(|\mathcal{H}_{H_{K,p}^{\psi,\alpha}}|)$. Then it holds that
\begin{align*}
  D^{\boldsymbol{\beta}}[(1+L_{\psi})^{\frac{\alpha}{2}}\Phi]
  =(1+L_{\psi})^{\frac{\alpha}{2}}[D^{\boldsymbol{\beta}}\Phi].
\end{align*}
\end{lemma}

\begin{proof}
Let $\Phi\in \mathcal{C}_{\boldsymbol{\beta}}(\mathcal{E}_{C_{\mathrm{c},K}^{\infty}})$ such that
\begin{align*}
\Phi(t,x)=\sum_{i=1}^{m}F_{i}1_{(t_{i-1},t_{i}]}(t)\phi_{i}(x),\quad t\in [0,T],\quad x\in\mathbb{R}^{d},
\end{align*}
where $F_{i}\in\mathcal{C}_{\boldsymbol{\beta}}$,
$0\leq t_{0}<t_{1}<\cdots<t_{m}\leq T$, and $\phi_{i}\in C_{\mathrm{c},K}^{\infty}$.
Then we obtain that
\begin{align*}
D^{\boldsymbol{\beta}}[(1+L_{\psi})^{\frac{\alpha}{2}}\Phi]
&=\sum_{i=1}^{m}\left(D^{\boldsymbol{\beta}}F_{i}\right)1_{(t_{i-1},t_{i}]}(1+L_{\psi})^{\frac{\alpha}{2}}\phi_{i}\\
&=(1+L_{\psi})^{\frac{\alpha}{2}}[D^{\boldsymbol{\beta}}\Phi],
\end{align*}
from which, by applying the approximation method, the proof is completed.
\end{proof}

\begin{proposition}\label{prop:L-p estimate of HSI}
For any $p\geq 2$,
there exists a constant $c>0$ depending on $p$ such that for any $\alpha\in\mathbb{R}$ and
$\Phi\in\mathbb{D}_{\boldsymbol{\beta}}^{1,p}(|\mathcal{H}_{H_{p, K\otimes U_{0}}^{\psi,\alpha}}|)$,
\begin{align}\label{eqn:L-p estimate of HSI}
\mathbb{E}\|\delta^{\boldsymbol{\beta}}(\Phi)\|_{H_{K,p}^{\psi,\alpha}}^{p}
&\leq c\left(\mathbb{E}\|\Phi\|_{|\mathcal{H}_{H_{p, K\otimes U_{0}}^{\psi,\alpha}}|}^{p}
  +\mathbb{E}\|D^{\boldsymbol{\beta}}\Phi\|_{|\mathcal{H}_{U_{0}}|\otimes|\mathcal{H}_{H_{p, K\otimes U_{0}}^{\psi,\alpha}}|}^{p}\right)\\
&=c\|\Phi\|_{\mathbb{D}_{\boldsymbol{\beta}}^{1,p}(|\mathcal{H}_{H_{p, K\otimes U_{0}}^{\psi,\alpha}}|)}^{p}.
 \nonumber
\end{align}
\end{proposition}

\begin{proof}
Note that by the definition of $(1+L_{\psi})^{\frac{\alpha}{2}}$ and stochastic Fubini theorem, we have
\[
(1+L_{\psi})^{\frac{\alpha}{2}}\delta^{\boldsymbol{\beta}}(\Phi)
=\delta^{\boldsymbol{\beta}}((1+L_{\psi})^{\frac{\alpha}{2}}\Phi)
\]
(also, see Lemma \ref{lmm:commutativity between D and 1-Delta}).
By taking $\widetilde{\Phi}(\cdot,x):=(1+L_{\psi})^{\frac{\alpha}{2}}\Phi(\cdot,x)$
 and applying the inequalities \eqref{eqn:ineq for continuity of skorohod integral},
\eqref{ineq: H norm leq absolute H norm} and \eqref{ineq: Hx H norm leq absolute H x absolute H norm}, we obtain that
\begin{align}
\mathbb{E}\|\delta^{\boldsymbol{\beta}}(\Phi)\|_{H_{K,p}^{\psi,\alpha}}^{p}
&=\mathbb{E}\|\delta^{\boldsymbol{\beta}}((1+L_{\psi})^{\frac{\alpha}{2}}\Phi)\|_{L_{K}^{p}}^{p}
=\int_{\mathbb{R}^{d}}\mathbb{E}\left\|\delta^{\boldsymbol{\beta}}(\widetilde{\Phi}(\cdot,x))\right\|_{K}^{p}dx\nonumber\\
&\leq c\int_{\mathbb{R}^{d}}\mathbb{E}\left\|\widetilde{\Phi}(\cdot,x)\right\|_{ K\otimes \mathcal{H}_{U_{0}}}^{p}dx
   +c\int_{\mathbb{R}^{d}}\mathbb{E}\left\|D^{\boldsymbol{\beta}}(\widetilde{\Phi}(\cdot,x))
      \right\|_{K\otimes \mathcal{H}_{U_{0}}\otimes \mathcal{H}_{U_{0}}}^{p}dx\nonumber\\
&= c\mathbb{E}[I_1]+c_p\mathbb{E}[I_2],\label{eqn:L-p estimate of HSI in proof}
\end{align}
where
\begin{align*}
I_1:=\int_{\mathbb{R}^{d}}\left\|\widetilde{\Phi}(\cdot,x)\right\|_{|\mathcal{H}_{ K\otimes U_{0}}|}^{p}dx,\quad
I_2:=\int_{\mathbb{R}^{d}}\left\|D^{\boldsymbol{\beta}}(\widetilde{\Phi}(\cdot,x))
  \right\|_{|\mathcal{H}_{U_{0}}|\otimes |\mathcal{H}_{K\otimes U_{0}}|}^{p}dx.
\end{align*}
For a notational convenience, we write
$d\boldsymbol{R}=\frac{\partial^{2}R}{\partial s\partial t}(s,t)
\frac{\partial^{2}R}{\partial \theta\partial \eta}(\theta,\eta)dsdt d\theta d\eta$.
Then by applying the Minkowski's inequality and Cauchy-Schwarz inequality, we obtain that
\begin{align}
I_2
&=\int_{\mathbb{R}^{d}}\left\|D^{\boldsymbol{\beta}}(\widetilde{\Phi}(\cdot,x))
  \right\|_{|\mathcal{H}_{U_{0}}|\otimes |\mathcal{H}_{K\otimes U_{0}}|}^{p}dx\nonumber\\
&=\int_{\mathbb{R}^{d}}\left(\int_{[0,T]^{4}}
 \|D_{\theta}^{\boldsymbol{\beta}}(\widetilde{\Phi}(s,x))\|_{K\otimes U_{0}\otimes U_{0}}
 \|D_{\eta}^{\boldsymbol{\beta}}(\widetilde{\Phi}(t,x))\|_{K\otimes U_{0}\otimes U_{0}}d\boldsymbol{R}\right)^{\frac{p}{2}}dx\nonumber\\
&\le \left(\int_{[0,T]^{4}}
 \left(\int_{\mathbb{R}^{d}}\|D_{\theta}^{\boldsymbol{\beta}}(\widetilde{\Phi}(s,x))\|_{K\otimes U_{0}\otimes U_{0}}^{\frac{p}{2}}
 \|D_{\eta}^{\boldsymbol{\beta}}(\widetilde{\Phi}(t,x))\|_{K\otimes U_{0}\otimes U_{0}}^{\frac{p}{2}}dx\right)^{\frac{2}{p}}
   d\boldsymbol{R}\right)^{\frac{p}{2}}\nonumber\\
&\leq \left(\int_{[0,T]^{4}}\left(\int_{\mathbb{R}^{d}}
  \|D_{\theta}^{\boldsymbol{\beta}}(\widetilde{\Phi}(s,x))\|_{K\otimes U_{0}\otimes U_{0}}^{p}dx\right)^{\frac{1}{p}}
  \left(\int_{\mathbb{R}^{d}}
  \|D_{\theta}^{\boldsymbol{\beta}}(\widetilde{\Phi}(t,x))\|_{K\otimes U_{0}\otimes U_{0}}^{p}dx\right)^{\frac{1}{p}}
  d\boldsymbol{R}\right)^{\frac{p}{2}}\nonumber\\
&=\left(\int_{[0,T]^{4}}
\|D_{\theta}^{\boldsymbol{\beta}}\Phi(s,\cdot)\|_{U_{0}\otimes H_{p, K\otimes U_{0}}^{\psi,\alpha}}
\|D_{\eta}^{\boldsymbol{\beta}}\Phi(t,\cdot)\|_{U_{0}\otimes H_{p, K\otimes U_{0}}^{\psi,\alpha}}d\boldsymbol{R}\right)^{\frac{p}{2}}\nonumber\\
&=\|D^{\boldsymbol{\beta}}\Phi\|_{|\mathcal{H}_{U_{0}}|\otimes|\mathcal{H}_{H_{p, K\otimes U_{0}}^{\psi,\alpha}}|}^{p}.
   \label{eqn:estimate of L2}
\end{align}
By similar arguments, we obtain that
\begin{align}\label{eqn:estimate of L1}
I_1\le \|\Phi\|_{|\mathcal{H}_{H_{p}^{\psi,\alpha}(K\otimes  U_{0})}|}^{p}.
\end{align}
Therefore, by combining \eqref{eqn:L-p estimate of HSI in proof}, \eqref{eqn:estimate of L2} and \eqref{eqn:estimate of L1},
we prove \eqref{eqn:L-p estimate of HSI}.
\end{proof}

Let $p\geq 2$ and $\alpha\in\mathbb{R}$.
Then $\delta^{\boldsymbol{\beta}}$ is a bounded linear operator
from $\mathbb{D}_{\boldsymbol{\beta}}^{1,p}(|\mathcal{H}_{H_{p, K\otimes U_{0}}^{\psi,\alpha}}|)$
into $L^p(\Omega,H_{K,p}^{\psi,\alpha})$.

\section{Generalized Littlewood-Paley Type Inequalities}\label{sec: GLPI}
In this section, we prove a generalized Littlewood-Paley type inequality
for the Banach space $L^{\boldsymbol{r}}((\eta_{1},\eta_{2});H)$-valued functions
(Theorem \ref{cor: 2nd LP ineq}), where $H$ is a Hilbert space.
In fact, Theorem \ref{cor: 2nd LP ineq} will be applied
for the proof of the main result for SPDEs  in Section \ref{sec: main result}.

Let $K$ be a real separable Hilbert space and
let $\psi:[0,\infty)\times \mathbb{R}^{d}\rightarrow \mathbb{C}$ be a measurable function.
For $l\geq 0$ and $f\in C_{\mathrm{c},K}^{\infty}$, the pseudo differential operator $L_{\psi}(l)$
with the symbol $\psi$ of order $\gamma>0$ (in the sense of \eqref{eqn: condition for symbol psi 1})
is defined by
\begin{align*}
L_{\psi}(l)f(x):=\mathcal{F}^{-1}(\psi(l,\xi)\mathcal{F}f(\xi))(x).
\end{align*}
We assume that there exist positive constants $\kappa:=\kappa_{\psi},\mu:=\mu_{\psi},\gamma:=\gamma_{\psi}$
and a natural number $N:=N_{\psi}\in\mathbb{N}$ with $N\ge \lfloor\frac{d}{2}\rfloor +1$ such that
\begin{itemize}
  \item [\textbf{(S1)}] for almost all $t\geq 0$ and $\xi\in\mathbb{R}^{d}$,
  \begin{equation}\label{eqn: condition for symbol psi 1}
  {\rm Re}[\psi (t,\xi)]\leq -\kappa|\xi|^{\gamma},
  \end{equation}
where ${\rm Re}(z)$ is the real part of the complex number $z$,

  \item [\textbf{(S2)}] for any multi-index $\alpha\in \mathbb{N}_{0}^{d}$ with
  $|\alpha|\leq N$,
   and for almost all $t\geq 0$ and $\xi\in\mathbb{R}^{d}\setminus\widetilde{\boldsymbol{0}}$,
   where $\widetilde{\boldsymbol{0}}=\{(x_1,\cdots,x_d)|x_i=0\text{ for some }i=1,\cdots,d\}$,
      \begin{equation}\label{eqn: condition for symbol psi 2}
      |\partial_{\xi}^{\alpha}\psi(t,\xi)|\leq \mu|\xi|^{\gamma-|\alpha|},
      \end{equation}
\end{itemize}

We put $\mathfrak{S}$ the set of all measurable functions (symbols of pseudo-differential operators)
$\psi:[0,\infty)\times \mathbb{R}^{d}\rightarrow \mathbb{C}$ satisfying the conditions \textbf{(S1)} and \textbf{(S2)}.

\begin{remark}
\upshape
The conditions \textbf{(S1)} and \textbf{(S2)} have been considered in \cite{Ji-Kim 2025-1, Ji-Kim 2025}.
Also, the conditions \textbf{(S1)} and \textbf{(S2)} have been considered in \cite{I. Kim K.-H. Kim 2016}
with $N=\lfloor\frac{d}{2}\rfloor+1$. In fact, the condition \textbf{(S2)} is stronger than the condition \eqref{eqn: sym cond 2}.
\end{remark}

If $\psi\in \mathfrak{S}$, then
by the condition \textbf{(S1)}, we have $\exp\left(\int_{s}^{t}\psi(r,\cdot)dr\right)\in L^1$ for each $t>s$.
Put
\begin{align*}
p_{\psi}(t,s,x)=\mathcal{F}^{-1}\left(\exp\left(\int_{s}^{t}\psi(r,\cdot)dr\right)\right)(x),\quad t>s,\quad x\in\mathbb{R}^d.
\end{align*}
Then by Corollary 2.10 in \cite{Ji-Kim 2025}, we see that $p_{\psi}(t,s,\cdot)\in L^{1}$ for each $t>s$.
Hence, for each $f\in L_{K}^{p}$ and $t>s$,
by applying the Young's convolution inequality, we can define
\begin{equation}\label{eqn: mathcal T}
\mathcal{T}_{\psi}(t,s)f(x)=p_{\psi}(t,s,\cdot)*f(x)=\int_{\mathbb{R}^{d}}p_{\psi}(t,s,x-y)f(y)dy
\end{equation}
as $\mathcal{T}_{\psi}(t,s)f\in L_{K}^{p}$.
Also, by the definition,
we can easily see that for any $t\geq r\geq s$ and $f\in L_{K}^{p}$,
\begin{align*}
\mathcal{T}_{\psi}(t,r)\mathcal{T}_{\psi}(r,s)f(x)
&=\mathcal{T}_{\psi}(t,s)f(x).
\end{align*}
This means that the family $\{\mathcal{T}_{\psi}(t,s)\}_{t\geq s\geq 0}$ is an evolution system.
If $\psi(t,\xi)=\psi(\xi)$, then we write $\mathcal{T}_{\psi}(t-s)=\mathcal{T}_{\psi}(t,s)$.

The following theorem is a generalized Littlewood-Paley type inequality for the Banach space
$L^{r}(\mathbb{R};H)$-valued functions, where $r\geq 1$ and $H$ is a Hilbert space.
Here for a function $f\in C_{\rm c}^{\infty}((a,b)\times\mathbb{R}^{d};L^{r}(\mathbb{R};H))$
and fixed $s\in (a,b)$, $x\in \mathbb{R}^{d}$ and $\theta\in \mathbb{R}$,
it holds that $f(s,x)(\theta)\in H$. Therefore, for a notational convenience,
we write $f(s,x,\theta):=f(s,x)(\theta)$
for all $s\in (a,b)$, $x\in \mathbb{R}^{d}$ and $\theta\in \mathbb{R}$.

\begin{theorem}\label{cor: 2nd LP ineq}
Let $H$ be a separable Hilbert space.
Let $\varphi\in\mathfrak{M}_{\gamma}(\mathbb{R}^{d})$ for some $\gamma\equiv\gamma_{\varphi}>0$
satisfying the condition \textbf{(S2)} and let $\psi\in\mathfrak{S}$.
Assume that $N_{\varphi},N_{\psi}>d+2+\lfloor\gamma_{\varphi}\rfloor+\lfloor\gamma_{\psi}\rfloor$.
Let $r\geq 1$ and let $q\geq \max\{2,r\}$.
Then for any $p\geq q$, it holds that

\begin{itemize}
  \item [\rm{(i)}] if  $-\infty< a< b<\infty$, then
there exists a constant $C_{1}>0$ depending on $a,b,p,q,d$,
$\gamma_{\varphi},\gamma_{\psi},\kappa_{\varphi},\kappa_{\psi},\mu_{\varphi}$ and $\mu_{\psi}$
such that for any $f\in C_{\rm c}^{\infty}((a,b)\times\mathbb{R}^{d}; L^{r}(\mathbb{R};H))$,
\begin{align}\label{ineq: 2nd LP ineq}
&\int_{\mathbb{R}^{d}}\int_{a}^{b}\left[\int_{a}^{t}(t-s)^{\frac{q\gamma_{\varphi}}{\gamma_{\psi}}-1}
\left(\int_{\mathbb{R}}\|L_{\varphi}
\mathcal{T}_{\psi}(t,s)f(s,\cdot,\theta)(x)\|_{H}^{r}d\theta\right)^{\frac{q}{r}} ds\right]^{\frac{p}{q}}dtdx\nonumber\\
&\qquad\leq C_{1}\int_{a}^{b}\left[\int_{\mathbb{R}}\left(\int_{\mathbb{R}^{d}}
\|f(t,x,\theta)\|_{H}^{p}dx\right)^{\frac{r}{p}}d\theta\right]^{\frac{p}{r}}dt
\end{align}
for which the right hand side of \eqref{ineq: 2nd LP ineq} is finite,

  \item [\rm{(ii)}] if $r\leq 2$ and $q=2$, there exists a constant $C_{2}>0$ depending on
  $p,d,\gamma_{\varphi},\gamma_{\psi}$,$\kappa_{\varphi},\kappa_{\psi}$, $\mu_{\varphi}$ and $\mu_{\psi}$
such that for any $f\in C_{\rm c}^{\infty}(\mathbb{R}^{d+1}; L^{r}(\mathbb{R};H))$,
\begin{align}\label{ineq: 2nd LP ineq q=2}
&\int_{-\infty}^{\infty}\int_{\mathbb{R}^{d}}\left[\int_{-\infty}^{t}(t-s)^{\frac{2\gamma_{\varphi}}{\gamma_{\psi}}-1}
    \left(\int_{\mathbb{R}}\|L_{\varphi}
    \mathcal{T}_{\psi}(t,s)f(s,\cdot,\theta)(x)\|_{H}^{r}d\theta\right)^{\frac{2}{r}}
            ds\right]^{\frac{p}{2}}dxdt\nonumber\\
&\qquad\leq C_{2}\int_{-\infty}^{\infty}\left[\int_{\mathbb{R}}\left(\int_{\mathbb{R}^{d}}
\|f(t,x,\theta)\|_{H}^{p}dx\right)^{\frac{r}{p}}d\theta\right]^{\frac{p}{r}}dt
\end{align}
for which the right hand side of \eqref{ineq: 2nd LP ineq q=2} is finite.
\end{itemize}
\end{theorem}

A proof of Theorem \ref{cor: 2nd LP ineq} will be given in Appendix \ref{sec:A Proof of LPI-Banach space}.

\begin{remark}\label{rmk: LP ineq FI}
\upshape
For any $-\infty\leq \eta_{1}<\eta_{2}\leq \infty$,
by considering the function $f1_{(\eta_{1},\eta_{2})}(\theta)$ in the inequality \eqref{ineq: 2nd LP ineq},
we have
\begin{align}\label{ineq: 2nd LP ineq FI}
&\int_{\mathbb{R}^{d}}\int_{a}^{b}\left[\int_{a}^{t}(t-s)^{\frac{q\gamma_{\varphi}}{\gamma_{\psi}}-1}
\left(\int_{\eta_{1}}^{\eta_{2}}\|L_{\varphi}
\mathcal{T}_{\psi}(t,s)f(s,\cdot,\theta)(x)\|_{H}^{r}d\theta\right)^{\frac{q}{r}} ds\right]^{\frac{p}{q}}dtdx\nonumber\\
&\qquad\leq C_{1}\int_{a}^{b}\left[\int_{\eta_{1}}^{\eta_{2}}\left(\int_{\mathbb{R}^{d}}
\|f(t,x,\theta)\|_{H}^{p}dx\right)^{\frac{r}{p}}d\theta\right]^{\frac{p}{r}}dt,
\end{align}
which will be used in the proof of Lemma \ref{lem: estimate of solution}.
\end{remark}

\begin{corollary}[\cite{Ji-Kim 2025-1}, Theorem 5.3]\label{thm:LP ineq}
Let $H$ be a separable Hilbert space.
Let $\varphi\in\mathfrak{M}_{\gamma}(\mathbb{R}^{d})$ for some $\gamma\equiv\gamma_{\varphi}>0$ satisfying the condition \textbf{(S2)} and let $\psi\in\mathfrak{S}$.
Assume that $N_{\varphi},N_{\psi}>d+2+\lfloor\gamma_{\varphi}\rfloor+\lfloor\gamma_{\psi}\rfloor$.
Then for any $q\geq 2$ and $p\geq q$, it holds that

\begin{itemize}
  \item [\rm{(i)}] if  $-\infty< a< b<\infty$, then
there exists a constant $C_{1}>0$ depending on $a,b,p,q,d$, $\gamma_{\varphi},\gamma_{\psi}$,
$\kappa_{\varphi},\kappa_{\psi},\mu_{\varphi}$ and $\mu_{\psi}$ such that
\begin{align}\label{ineq:LP-ineq}
&\int_{\mathbb{R}^{d}}\int_{a}^{b}\left(\int_{a}^{t}(t-s)^{\frac{q\gamma_{\varphi}}{\gamma_{\psi}}-1}
\|L_{\varphi}
\mathcal{T}_{\psi}(t,s)f(s,\cdot)(x)\|_{H}^{q}ds\right)^{\frac{p}{q}}dxdt\nonumber\\
&\qquad\leq C_{1}\int_{a}^{b}\int_{\mathbb{R}^{d}}
\|f(t,x)\|_{H}^{p}dtdx
\end{align}
for all $f\in L^p((a,b)\times\mathbb{R}^{d}; H)$,
  \item [\rm{(ii)}] (if $q=2$,) there exists a constant $C_{2}>0$ depending on
  $p,d,\gamma_{\varphi},\gamma_{\psi},\kappa_{\varphi},\kappa_{\psi}$, $\mu_{\varphi}$ and $\mu_{\psi}$
such that
\begin{align}\label{ineq:LP-ineq q=2}
&\int_{-\infty}^{\infty}\int_{\mathbb{R}^{d}}\left(\int_{-\infty}^{t}(t-s)^{\frac{2\gamma_{\varphi}}{\gamma_{\psi}}-1}
\|L_{\varphi}
\mathcal{T}_{\psi}(t,s)f(s,\cdot)(x)\|_{H}^{2}ds\right)^{\frac{p}{2}}dtdx\nonumber\\
&\qquad\leq C_{2}\int_{-\infty}^{\infty}\int_{\mathbb{R}^{d}}
\|f(t,x)\|_{H}^{p}dxdt
\end{align}
for all $f\in L^p(\mathbb{R}^{d+1}; H)$.
\end{itemize}
\end{corollary}

\begin{proof}
By Remark \ref{rmk: LP ineq FI}, it is obvious.
\end{proof}

\begin{remark}
\upshape
In fact, the results obtained in Corollary \ref{thm:LP ineq} are generalizations of results of Theorem 5.3 in \cite{Ji-Kim 2025-1}.
\end{remark}
\section{SPDEs Driven by Gaussian Processes in Hilbert Space}\label{sec: main result}

Let $\varphi\in\mathfrak{M}_{\gamma}(\mathbb{R}^{d})$ for some $\gamma>0$.
For each $\alpha\in\mathbb{R}$, $T>0$ and $p,r\ge1$, we define
\begin{align}
\mathbb{H}_{K,p,T}^{\varphi, \alpha}&=L^{p}(\Omega\times [0,T],\mathcal{F}\otimes \mathcal{B}([0,T]),H_{K,p}^{\varphi,\alpha})\label{eqn:H_pnt}
\end{align}

Motivated by Definition 3.1 in \cite{Krylov 1999},
we have the following definition.

\begin{definition}
\upshape
Let $q\geq \max\{2,\boldsymbol{r}\}$, where $\boldsymbol{r}$ is the constant given in Assumption \textbf{(R2)}.
Let $ T>0,\, \gamma>0$ and let $\alpha\in\mathbb{R}$.
For $p\geq q$ and $\mathcal{D}_{K}$-valued stochastic process $u$,
we say that $u\in\mathcal{H}_{K,p,\boldsymbol{r},q,T}^{\varphi,\alpha+\frac{2\gamma}{\gamma_{\varphi}}}$ if
\begin{itemize}
  \item [\rm{(i)}] $u(0,\cdot)\in L^{p}\left(\Omega,\mathcal{F},H_{K,p}^{\varphi,\alpha+\frac{2\gamma}{\gamma_{\varphi}}(1-\frac{1}{p})}\right)$,\

  \item [\rm{(ii)}] $u\in \mathbb{H}_{K,p,T}^{\varphi,\alpha+\frac{2\gamma}{\gamma_{\varphi}}}$,

  \item [\rm{(iii)}] there exist $f\in \mathbb{H}_{K,p}^{\varphi,\alpha}$,
  and $g\in \widetilde{\mathbb{L}}_{\boldsymbol{r},\boldsymbol{\beta}}^{1,p}
  (H_{K\otimes U_{0},p}^{\varphi,\alpha+\frac{2\gamma}{q'\gamma_{\varphi}}},T)$,
  where $q'$ is the conjugate number of $q$, such that
\begin{align*}
du(t,x)=f(t,x)dt+g(t,x)\delta\boldsymbol{\beta}_{t}
\end{align*}
in the sense of distributions, i.e., for any $\phi\in C_{\mathrm{c},K}^{\infty}$ and $t\in [0,T]$,
\begin{equation}\label{the weak solution}
(u(t,\cdot),\phi)=(u(0,\cdot),\phi)+\int_{0}^{t}(f(s,\cdot),\phi)ds
+\left(\int_{0}^{t}g(s,\cdot)\delta\boldsymbol{\beta}_{s},\phi\right)
\end{equation}
holds a.s.
\end{itemize}
In this case, we write
\begin{align*}
\mathbb{D}u=f,\quad \mathbb{S}u=g
\end{align*}
and define the norm
\begin{align}\label{eqn: sol. sp norm}
\|u\|_{\mathcal{H}_{K,p,\boldsymbol{r},q,T}^{\varphi,\alpha+\frac{2\gamma}{\gamma_{\varphi}}}}
=&\|u\|_{\mathbb{H}_{K,p,T}^{\varphi,\alpha+\frac{2\gamma}{\gamma_{\varphi}}}}
        +\|\mathbb{D}u\|_{\mathbb{H}_{K,p,T}^{\varphi,\alpha}}
+\|\mathbb{S}u\|_{\mathbb{L}_{\boldsymbol{r},\boldsymbol{\beta}}^{1,p}
(H_{K\otimes U_{0},p}^{\varphi,\alpha+\frac{2\gamma}{q'\gamma_{\varphi}}},T)}\nonumber\\
&\quad
+\left(\mathbb{E}\|u(0,\cdot)\|_{H_{K,p}^{\varphi,\alpha+\frac{2\gamma}{\gamma_{\varphi}}(1-\frac{1}{p})}}^{p}\right)^{\frac{1}{p}}.
\end{align}
\end{definition}
From \eqref{eqn: sol. sp norm},
it is obvious that for any $u\in\mathcal{H}_{K,p,\boldsymbol{r},q,T}^{\varphi,\alpha+\frac{2\gamma}{\gamma_{\varphi}}}$,
\begin{align}\label{eqn:ineq sol.space 2}
\|u\|_{\mathbb{H}_{K,p,T}^{\varphi,\alpha+\frac{2\gamma}{\gamma_{\varphi}}}}
\leq \|u\|_{\mathcal{H}_{K,p,\boldsymbol{r},q,T}^{\varphi,\alpha+\frac{2\gamma}{\gamma_{\varphi}}}}.
\end{align}
\begin{remark}
\upshape
If $q=\gamma=2$ and $\varphi(\xi)=|\xi|^2 \,\,(\xi\in\mathbb{R}^{d})$, then the norm
$\|u\|_{\mathbb{H}_{K,p,T}^{\varphi,\alpha+2}}$ is equivalent to
$\|u_{xx}\|_{\mathbb{H}_{K,p,T}^{\varphi,\alpha}}$, where $u_{xx}$ is the Hessian matrix
$\left[\frac{\partial^{2}u}{\partial x_{i}\partial x_{j}}\right]_{i,j=1}^{d}$.
Therefore, in this case, the norm defined in \eqref{eqn: sol. sp norm} of the solution space
is equivalent to the norm of the solution space defined in Definition 3.1 of \cite{Krylov 1999}.
\end{remark}
\begin{theorem}
Let $\alpha\in\mathbb{R}$ and let $T>0$. Then it holds that
\begin{itemize}
\item[\rm (i)] there exists a constant $C>0$ depending on $ d, p$ and $T$
such that for any $u\in\mathcal{H}_{K,p,\boldsymbol{r},q,T}^{\varphi,\alpha+\frac{2\gamma}{\gamma_{\varphi}}}$,
\begin{align}
\mathbb{E}\sup_{t\leq T}\|u(t,\cdot)\|_{H_{K,p}^{\varphi,\alpha}}^{p}
&\leq C\|u\|_{\mathcal{H}_{K,p,\boldsymbol{r},q,T}^{\varphi,\alpha+\frac{2\gamma}{\gamma_{\varphi}}}},
\label{eqn:ineq sol.space 1}
\end{align}
\item [\rm (ii)] $(\mathcal{H}_{K,p,\boldsymbol{r},q,T}^{\varphi,\alpha+\frac{2\gamma}{\gamma_{\varphi}}},
    \|\cdot\|_{\mathcal{H}_{K,p,\boldsymbol{r},q,T}^{\varphi,\alpha+\frac{2\gamma}{\gamma_{\varphi}}}})$
 is a Banach space.
\end{itemize}
\end{theorem}

\begin{proof}
(i)\enspace By (i) in Proposition \ref{prop: psi bessel}, we may assume that $\alpha=0$.
Let $u\in\mathcal{H}_{K,p,\boldsymbol{r},q,T}^{\varphi,\frac{2\gamma}{\gamma_{\varphi}}}$ be given.
Let $\zeta\in C_{\rm c}^{\infty}$ be nonnegative with $\int_{\mathbb{R}^{d}}\zeta(x)dx=1$.
For any $\epsilon>0$ and any function $h\in\mathcal{D}_{K}$, let
\begin{align*}
\zeta_{\epsilon}(x)=\epsilon^{-d}\zeta\left(\frac{x}{\epsilon}\right),\quad h^{(\epsilon)}(x)=h*\zeta_{\epsilon}(x),\quad x\in \mathbb{R}^{d}.
\end{align*}
Note that $\|h^{(\epsilon)}\|_{L_{K}^{p}}\leq \|h\|_{L_{K}^{p}}$ for any $h\in L_{K}^{p}$.
By the definition of $\mathcal{H}_{K,p,\boldsymbol{r},q,T}^{\varphi,\frac{2\gamma}{\gamma_{\varphi}}}$,
there exist $f\in\mathbb{H}_{K,p,T}^{\varphi,0}$
and $g\in\widetilde{\mathbb{L}}_{\boldsymbol{r},\boldsymbol{\beta}}^{1,p}
(H_{K\otimes U_{0},p}^{\varphi,\frac{2\gamma}{q'\gamma_{\varphi}}},T)$
such that for any $\phi\in C_{\mathrm{c},K}^{\infty}$ and $t\in [0,T]$, the equality
\begin{equation}\label{eqn: u satisfies duality form}
(u(t,\cdot),\phi)
=(u(0,\cdot),\phi)+\int_{0}^{t}(f(s,\cdot),\phi)ds
        +\left(\int_{0}^{t}g(s,\cdot)\delta\boldsymbol{\beta}_{s},\phi\right)
\end{equation}
holds a.s.
Let $k_{0}\in K$. Then $\zeta_{\epsilon}(x-\cdot)\otimes k_{0}\in C_{\mathrm{c},K}^{\infty}$
and we obtain that for any $\eta\in\mathbb{R}$ and $v(t,\cdot)\in H_{K,p}^{\varphi,\eta}$,
\begin{align*}
(v(t,\cdot),\zeta_{\epsilon}(x-\cdot)\otimes k_{0})
&=\int_{\mathbb{R}^{d}}\bilin{(1-L_{\varphi})^{\frac{\eta}{2}}v(t,y)}{(1-L_{\varphi})^{-\frac{\eta}{2}}\zeta_{\epsilon}(x-y) k_{0}}_{K}dy\\
&=\bilin{\int_{\mathbb{R}^{d}}v(t,y)\zeta_{\epsilon}(x-y) dy}{k_{0}}_{K}\\
&=\bilin{v^{(\epsilon)}(t,x)}{k_{0}}_{K}.
\end{align*}
Especially, by the stochastic Fubini theorem, we obtain that
\begin{align*}
\left(\int_{0}^{t}g(s,\cdot)\delta\boldsymbol{\beta}_{s},\zeta_{\epsilon}(x-\cdot)\otimes k_{0}\right)
&=\bilin{\left(\int_{0}^{t}g(s,\cdot)\delta\boldsymbol{\beta}_{s}\right)^{(\epsilon)}(x)}{k_{0}}_{K}\\
&=\bilin{\int_{0}^{t}g^{(\epsilon)}(s,x)\delta\boldsymbol{\beta}_{s}}{k_{0}}_{K}.
\end{align*}
Substituting $\zeta_{\epsilon}(x-\cdot)\otimes k_{0}$ instead of $\phi$ in \eqref{eqn: u satisfies duality form}, we have
\begin{align*}
\bilin{u^{(\epsilon)}(t,x)}{k_{0}}_{K}
&=\bilin{u^{(\epsilon)}(0,x)}{k_{0}}_{K}+\int_{0}^{t}\bilin{f^{(\epsilon)}(t,x)}{k_{0}}_{K}ds\\
&\qquad+\bilin{\int_{0}^{t}g^{(\epsilon)}(s,x)\delta\boldsymbol{\beta}_{s}}{k_{0}}_{K}.
\end{align*}
Since $k_{0}\in K$ is arbitrary, we have
\begin{equation*}
u^{(\epsilon)}(t,x)
=u^{(\epsilon)}(0,x)+\int_{0}^{t}f^{(\epsilon)}(s,x)ds
+\int_{0}^{t}g^{(\epsilon)}(s,x)\delta\boldsymbol{\beta}_{s},
\end{equation*}
which, by applying Jensen's inequality, implies that
\begin{align}
\|u^{(\epsilon)}(t,\cdot)\|_{L_{K}^{p}}^{p}
\leq & 3^{p-1}\left(\|u^{(\epsilon)}(0,\cdot)\|_{L_{K}^{p}}^{p}
+\left\|\int_{0}^{t}f^{(\epsilon)}(s,\cdot)ds\right\|_{L_{K}^{p}}^{p}\right.\nonumber\\
&\left.\qquad+\left\|\int_{0}^{t}g^{(\epsilon)}(s,\cdot)\delta\boldsymbol{\beta}_{s}\right\|_{L_{K}^{p}}^{p}\right).
\label{eqn:ineq u epsilon}
\end{align}
For the first term of the right hand side of \eqref{eqn:ineq u epsilon},
since $p\ge q\geq \max\{2,\boldsymbol{r}\}\ge 1$,
by applying (ii) in Proposition \ref{prop: psi bessel}, we have
\begin{align}
\mathbb{E}\|u^{(\epsilon)}(0,\cdot)\|_{L_{K}^{p}}^{p}
\leq \mathbb{E}\|u(0,\cdot)\|_{L_{K}^{p}}^{p}
\leq \mathbb{E}\|u(0,\cdot)\|_{H_{K,p}^{\varphi,\frac{2\gamma}{\gamma_{\varphi}}(1-\frac{1}{p})}}^{p}
\leq \|u\|_{\mathcal{H}_{K,p,\boldsymbol{r},q,T}^{\varphi,\frac{2\gamma}{\gamma_{\varphi}}}}^{p}.
\label{eqn:ineq 1st term}
\end{align}
For the second term the right hand side of \eqref{eqn:ineq u epsilon},
by applying the Jensen's inequality and the Fubini theorem, we obtain that
\begin{align*}
\left\|\int_{0}^{t}f^{(\epsilon)}(s,\cdot)ds\right\|_{L_{K}^{p}}^{p}
&=\int_{\mathbb{R}^{d}}\left\|\int_{0}^{t}f^{(\epsilon)}(s,x)ds\right\|_{K}^{p}dx\\
&\leq t^{p-1} \int_{\mathbb{R}^{d}} \int_{0}^{t}\|f^{(\epsilon)}(s,x)\|_{K}^{p}dsdx\\
&= t^{p-1} \int_{0}^{t} \int_{\mathbb{R}^{d}}\|f^{(\epsilon)}(s,x)\|_{K}^{p}dxds\\
&= t^{p-1} \int_{0}^{t} \|f^{(\epsilon)}(s,\cdot)\|_{L_{K}^{p}}^{p}ds\\
&\leq t^{p-1} \int_{0}^{t} \|f(s,\cdot)\|_{L_{K}^{p}}^{p}ds.
\end{align*}
Thus we obtain that
\begin{align}
\mathbb{E}\sup_{t\leq T}\left\|\int_{0}^{t}f^{(\epsilon)}(s,\cdot)ds\right\|_{L_{K}^{p}}^{p}
&\leq T^{p-1}\mathbb{E}\int_{0}^{T} \|f(s,\cdot)\|_{L_{K}^{p}}^{p}ds\nonumber\\
&=T^{p-1}\|f\|_{\mathbb{H}_{K,p,T}^{\varphi,0}}^{p}\nonumber\\
&\leq T^{p-1}\|u\|_{\mathcal{H}_{K,p,\boldsymbol{r},q,T}^{\varphi,\frac{2\gamma}{\gamma_{\varphi}}}}^{p}.
\label{eqn:ineq 2nd term}
\end{align}
We now estimate the third term the right hand side of \eqref{eqn:ineq u epsilon}.
By \eqref{eqn: cor maximal inequality for single gaussian process},
there exists a constant $C^{(1)}>0$ depending on $p,\boldsymbol{r}$, and $T$ such that for any $x\in\mathbb{R}^{d}$,
\begin{align*}
\mathbb{E}\sup_{t\leq T}\left\|\int_{0}^{t}g^{(\epsilon)}(s,x)\delta\boldsymbol{\beta}_{s}\right\|_{K}^p
&\leq C^{(1)}\left\{\mathbb{E}\int_{0}^{T}\|g^{(\epsilon)}(s,x)\|_{K\otimes U_{0}}^p ds\right.\\
&\qquad+\left.\mathbb{E}\int_{0}^{T}\left(\int_{0}^{T}\|D_{\theta}^{\boldsymbol{\beta}}
g^{(\epsilon)}(s,x)\|_{K\otimes U_{0}\otimes U_{0}}^{\boldsymbol{r}}d\theta\right)^{\frac{p}{\boldsymbol{r}}}ds\right\},
\end{align*}
from which, by applying Minkowski's inequality
and the fact that $\|h^{(\epsilon)}\|_{L_{K}^{p}}\leq \|h\|_{L_{K}^{p}}$
for any $h\in L_{K}^{p}$, we obtain that
\begin{align}
&\mathbb{E}\sup_{t\leq T}\left\|\int_{0}^{t}
    g^{(\epsilon)}(s,\cdot)\delta\boldsymbol{\beta}_{s}\right\|_{L_{K}^{p}}^{p}\nonumber\\
&=\mathbb{E}\sup_{t\leq T}\int_{\mathbb{R}^{d}}\left\|\int_{0}^{t}
    g^{(\epsilon)}(s,x)\delta\boldsymbol{\beta}_{s}\right\|_{K}^{p}dx\nonumber\\
&\leq \mathbb{E}\int_{\mathbb{R}^{d}}\sup_{t\leq T}\left\|\int_{0}^{t}
    g^{(\epsilon)}(s,x)\delta\boldsymbol{\beta}_{s}\right\|_{K}^{p}dx\nonumber\\
&\leq C^{(1)}\left\{\mathbb{E}\int_{0}^{T}\int_{\mathbb{R}^{d}}
        \|g^{(\epsilon)}(s,x)\|_{K\otimes U_{0}}^{p} dxds\right.\nonumber\\
&\left.\quad +\mathbb{E}\int_{0}^{T}\int_{\mathbb{R}^{d}}\left(\int_{0}^{T}\|D_{\theta}^{\boldsymbol{\beta}}
        g^{(\epsilon)}(s,x)\|_{K\otimes U_{0}\otimes U_{0}}^{\boldsymbol{r}}d\theta\right)^{\frac{p}{\boldsymbol{r}}}dxds\right\}\nonumber\\
&\leq C^{(1)}\left\{\mathbb{E}\int_{0}^{T}\|g^{(\epsilon)}(s,\cdot)\|_{L_{K\otimes U_{0}}^{p}}^{p}ds
    +\mathbb{E}\int_{0}^{T}\left(\int_{0}^{T}\|D_{\theta}^{\boldsymbol{\beta}}
        g^{(\epsilon)}(s,\cdot)\|_{L_{K\otimes U_{0}\otimes U_{0}}^{p}}^{\boldsymbol{r}}d\theta\right)^{\frac{p}{\boldsymbol{r}}}ds\right\}\nonumber\\
&\leq C^{(1)}\left\{\mathbb{E}\int_{0}^{T}\|g(s,\cdot)\|_{L_{K\otimes U_{0}}^{p}}^{p}ds
    +\mathbb{E}\int_{0}^{T}\left(\int_{0}^{T}\|D_{\theta}^{\boldsymbol{\beta}}
        g(s,\cdot)\|_{L_{K\otimes U_{0}\otimes U_{0}}^{p}}^{\boldsymbol{r}}d\theta\right)^{\frac{p}{\boldsymbol{r}}}ds\right\}\nonumber\\
&=C^{(1)}\|g\|_{\mathbb{L}_{\boldsymbol{r},\boldsymbol{\beta}}^{1,p}(L_{K\otimes U_{0}}^{p},T)}^{p}.
\label{eqn:ineq Skorohod int Lp}
\end{align}
Since
\begin{align*}
\|g\|_{\mathbb{L}_{\boldsymbol{r},\boldsymbol{\beta}}^{1,p}(L_{K\otimes U_{0}}^{p},T)}
\leq\|g\|_{\mathbb{L}_{\boldsymbol{r},\boldsymbol{\beta}}^{1,p}(H_{K\otimes U_{0},p}^{\varphi,\frac{2\gamma}{q'\gamma_{\varphi}}},T)}
\leq \|u\|_{\mathcal{H}_{K,p,\boldsymbol{r},q,T}^{\varphi,\frac{2\gamma}{\gamma_{\varphi}}}},
\end{align*}
we have
\begin{align}
\mathbb{E}\sup_{t\leq T}\left\|\int_{0}^{t}g^{(\epsilon)}(s,\cdot)\delta\boldsymbol{\beta}_{s}\right\|_{L_{K}^{p}}^{p}
\leq C^{(1)}\|u\|_{\mathcal{H}_{K,p,\boldsymbol{r},q,T}^{\varphi,\frac{2\gamma}{\gamma_{\varphi}}}}^{p}.
\label{eqn:ineq 3rd term}
\end{align}
Therefore, by combining \eqref{eqn:ineq u epsilon}, \eqref{eqn:ineq 1st term},
\eqref{eqn:ineq 2nd term} and \eqref{eqn:ineq 3rd term}, we have the inequality given in \eqref{eqn:ineq sol.space 1}.

(ii)\enspace Let $\{u_{j}\}_{j\geq 1}$ be a Cauchy sequence
in $\mathcal{H}_{K,p,\boldsymbol{r},q,T}^{\varphi,\alpha+\frac{2\gamma}{\gamma_{\varphi}}}$.
Then by \eqref{eqn:ineq sol.space 2}, $\{u_{j}\}_{j\geq 1}$ is also a Cauchy sequence
in $\mathbb{H}_{K,p,T}^{\varphi,\alpha+\frac{2\gamma}{\gamma_{\varphi}}}$ and so
$\{u_{j}\}_{j\geq 1}$ converges in $\mathbb{H}_{K,p,T}^{\varphi,\alpha+\frac{2\gamma}{\gamma_{\varphi}}}$.
Let $u\in\mathbb{H}_{K,p,T}^{\varphi,\alpha+\frac{2\gamma}{\gamma_{\varphi}}}$ be the limit of $\{u_{j}\}_{j\geq 1}$.
To show that $u\in \mathcal{H}_{K,p,\boldsymbol{r},q,T}^{\varphi,\alpha+\frac{2\gamma}{\gamma_{\varphi}}}$,
we claim that there exist $f\in \mathbb{H}_{K,p,T}^{\varphi,\alpha}$
and $g\in\widetilde{\mathbb{L}}_{\boldsymbol{r},\boldsymbol{\beta}}^{1,p}
(H_{K\otimes U_{0}, p}^{\varphi,\alpha+\frac{2\gamma}{q'\gamma_{\varphi}}},T)$ such that for any $\phi\in C_{\mathrm{c},K}^{\infty}$ and $t\in [0,T]$,
\begin{equation}\label{third condition of def of solution space for uj}
(u(t,\cdot),\phi)=(u(0,\cdot),\phi)+\int_{0}^{t}(f(s,\cdot),\phi)ds
+\left(\int_{0}^{t}g(s,\cdot)\delta\boldsymbol{\beta}_{s},\phi\right)
\end{equation}
a.s.
Since $u_{j}\in \mathcal{H}_{K,p,\boldsymbol{r},q,T}^{\varphi,\alpha+\frac{2\gamma}{\gamma_{\varphi}}}$ for each $j$,
by the definition, there exist $f_{j}\in\mathbb{H}_{K,p,T}^{\varphi,\alpha}$ and
$g_{j}\in \widetilde{\mathbb{L}}_{\boldsymbol{r},\boldsymbol{\beta}}^{1,p}
(H_{K\otimes U_{0},p}^{\varphi,\alpha+\frac{2\gamma}{q'\gamma_{\varphi}}},T)$
such that for any $\phi\in C_{\mathrm{c},K}^{\infty}$ and $t\in [0,T]$, \eqref{third condition of def of solution space for uj} holds a.s.

Since $\{u_{j}\}_{j\geq 1}$ is Cauchy in $\mathcal{H}_{K,p,\boldsymbol{r},q,T}^{\varphi,\alpha+\frac{2\gamma}{\gamma_{\varphi}}}$,
$\{u_{j}(0,\cdot)\}_{j\geq 1}$, $\{f_{j}\}_{j\geq 1}$,
and $\{g_{j}\}_{j\geq 1}$ are Cauchy in $L^{p}(\Omega,\mathcal{F}, H_{K,p}^{\varphi,\alpha+\frac{2\gamma}{\gamma_{\varphi}}(1-\frac{1}{p})})$,
$\mathbb{H}_{K,p,T}^{\varphi,\alpha}$ and
$\widetilde{\mathbb{L}}_{\boldsymbol{r},\boldsymbol{\beta}}^{1,p}
(H_{K\otimes U_{0},p}^{\varphi,\alpha+\frac{2\gamma}{q'\gamma_{\varphi}}},T)$, respectively.
We define $u(0,\cdot)\in L^{p}(\Omega,\mathcal{F}, H_{K,p}^{\varphi,\alpha+\frac{2\gamma}{\gamma_{\varphi}}(1-\frac{1}{p})})$,
$f\in\mathbb{H}_{K,p,T}^{\varphi,\alpha}$ and
$g\in\widetilde{\mathbb{L}}_{\boldsymbol{r},\boldsymbol{\beta}}^{1,p}
(H_{K\otimes U_{0},p}^{\varphi,\alpha+\frac{2\gamma}{q'\gamma_{\varphi}}},T)$
by the limits of $u_{j}(0,\cdot)$, $f_{j}$ and $g_{j}$, respectively.

Since $\|u_{j}-u\|_{\mathbb{H}_{K,p,T}^{\varphi,\alpha+\frac{2\gamma}{\gamma_{\varphi}}}}\rightarrow 0$,
there exists a subsequence $\{u_{j_{m}}\}_{m\geq 1}$ of $\{u_{j}\}_{j\geq 1}$ such that
\begin{align}\label{eqn:PWC}
\|u_{j_{m}}(\omega,t)-u(\omega,t)\|_{H_{K,p}^{\varphi,\alpha+\frac{2\gamma}{\gamma_{\varphi}}}}
\rightarrow 0\quad
 \text{for almost all } (\omega, t)\in \Omega\times [0,T].
\end{align}
Let $\Omega_{0}\subset\Omega$ and $A\subset [0,T]$ be null sets such that
\eqref{eqn:PWC} holds for $\omega\in\Omega\setminus\Omega_{0}$ and $t\in [0,T]\setminus A$.

Fix $\omega\in\Omega\setminus\Omega_{0}$ and $t\in [0,T]\setminus A$,
and consider the equality \eqref{third condition of def of solution space for uj} for $u_{j_{m}}$.
For the left hand side of \eqref{third condition of def of solution space for uj},
by applying the duality, we have
\begin{align*}
|(u_{j_{m}}(\omega,t),\phi)-(u(\omega,t),\phi)|
\leq C^{(2)} \|u_{j_{m}}(\omega,t)-u(\omega,t)\|_{H_{K,p}^{\varphi,\alpha+\frac{2\gamma}{\gamma_{\varphi}}}}\rightarrow 0.
\end{align*}
Considering the right hand side of \eqref{third condition of def of solution space for uj},
the first two terms converge to $(u(0,\cdot),\phi)$ and $\int_{0}^{t}(f(s,\cdot),\phi)ds$, respectively.
For the third term in the right hand side of \eqref{third condition of def of solution space for uj},
by applying the duality and \eqref{eqn:ineq Skorohod int Lp}, we obtain that
\begin{align*}
&\mathbb{E}\left|\left(\int_{0}^{t}
(g_{j_{m}}(s,\cdot)-g(s,\cdot))\delta\boldsymbol{\beta}_{s},\phi\right)\right|^{p}\\
&\leq C^{(3)}\left\{\mathbb{E}\int_{0}^{T}
\|g_{j_{m}}(s,x)-g(s,x)\|_{H_{K\otimes U_{0},p}^{\varphi,\alpha}}^{p}ds\right.\\
&\qquad\left.+\mathbb{E}\int_{0}^{T}\left(\int_{0}^{T}\|D_{\theta}^{\boldsymbol{\beta}}
(g_{j_{m}}(s,x)-g(s,x))\|_{H_{K\otimes U_{0}\otimes U_{0},p}^{\varphi,\alpha}}^{\boldsymbol{r}}d\theta\right)^{\frac{p}{\boldsymbol{r}}}ds\right\}\\
&=C^{(3)}\|g_{j_{m}}-g\|_{\mathbb{L}_{\boldsymbol{r},\boldsymbol{\beta}}^{1,p}(H_{K\otimes U_{0},p}^{\varphi,\alpha},T)}^{p}\\
&\leq C^{(3)}C^{(4)}\|g_{j_{m}}-g\|_{\mathbb{L}_{\boldsymbol{r},\boldsymbol{\beta}}^{1,p}
(H_{K\otimes U_{0},p}^{\varphi,\alpha+\frac{2\gamma}{q'\gamma_{\varphi}}},T)}^{p}\\
&\rightarrow 0 \quad \text{as}\quad m\rightarrow \infty.
\end{align*}
Hence there exists a subsequence $\{g_{j_{m_{l}}}\}_{l\geq 1}$ of $\{g_{j_{m}}\}_{m\geq 1}$ such that
\begin{align*}
\left(\int_{0}^{t}(g_{j_{m_{l}}}(s,\cdot)-g(s,\cdot))\delta\boldsymbol{\beta}_{s},\phi\right)\rightarrow 0\quad a.s.
\end{align*}
Therefore, for any $\phi\in C_{\mathrm{c},K}^{\infty}$ and $t\in [0,T]\setminus A$,
\eqref{third condition of def of solution space for uj} holds a.s.

On the other hand, for a fixed $\phi\in C_{\mathrm{c},K}^{\infty}$ and all $l\geq 1$,
$(u_{j_{m_{l}}}(t,\cdot),\phi)$ is continuous in $t$ a.s. (see, \eqref{the weak solution}),
and then by applying \eqref{eqn:ineq sol.space 1},
we can easily see that
$(u_{j_{m_{l}}}(t,\cdot),\phi)$ converges to $(u(t,\cdot),\phi)$ uniformly in $t$, in probability.
By applying the Borel-Cantelli lemma, we can see that the process $(u(t,\cdot),\phi)$ is continuous a.s.
Hence, for any $t\leq T$, the equality \eqref{third condition of def of solution space for uj} holds a.s.,
and so $u\in\mathcal{H}_{K,p,\boldsymbol{r},q,T}^{\varphi,\alpha+\frac{2\gamma}{\gamma_{\varphi}}}$.
\end{proof}

The following lemma is a Hilbert space version of Theorem 4.3 in \cite{Ji-Kim 2025}.

\begin{lemma} \label{lem: initial condition esti}
Let $\varphi\in\mathfrak{M}_{\gamma}(\mathbb{R}^{d})$ for some $\gamma\equiv\gamma_{\varphi}>0$
satisfying the condition \textbf{(S2)} and let $\psi\in\mathfrak{S}$.
Let $0< T< \infty$ and let $p\geq 2$. Then it holds that
there exists a constant $C>0$ depending on $T,d,p,\gamma_{\varphi},\gamma_{\psi},\mu_{\varphi},\mu_{\psi}$,
$\kappa_{\varphi}$ and $\kappa_{\psi}$
such that for any $s\geq 0$ and $f\in L_{K}^{p}$,
\begin{equation}\label{eqn: LP 1}
\int_{s}^{s+T}\int_{\mathbb{R}^{d}}(t-s)^{\frac{p\gamma_{\varphi}}{\gamma_{\psi}}-1}
       \|L_{\varphi}\mathcal{T}_{\psi}(t,s)f(x)\|_{K}^{p}dxdt
\leq C\int_{\mathbb{R}^{d}}\|f(x)\|_{K}^{p}dx.
\end{equation}
\end{lemma}

\begin{proof}
This lemma can be proved by the similar arguments used in the proof of Theorem 4.3 in \cite{Ji-Kim 2025}
by taking Theorem A.2 in \cite{Ji-Kim 2025-1} into account.
Therefore, the proof is omitted.
\end{proof}

\begin{lemma}\label{lem:L psi potential}
Let $\varphi\in\mathfrak{M}_{\gamma}(\mathbb{R}^{d})$ for some $\gamma\equiv\gamma_{\varphi}>0$
satisfying the condition \textbf{(S2)} and let $\psi\in\mathfrak{S}$.
Assume that $\gamma_{\varphi}=\gamma_{\psi}$ and $N_{\varphi},N_{\psi}>d+2+2\lfloor\gamma\rfloor$.
Then for any $1<p<\infty$, there exists a constant $C>0$ depending on $d,p,\gamma,\mu_{\varphi}$ and $\kappa_{\psi}$
such that for any $f\in L^{p}(\mathbb{R}^{d+1};K)$,
\begin{align*}
\int_{-\infty}^{\infty}\int_{\mathbb{R}^{d}}
\left\|\int_{-\infty}^{t}L_{\varphi}\mathcal{T}_{\psi}(t,s)f(s,x)ds\right\|_{K}^{p}dxdt
\leq C\int_{-\infty}^{\infty}\int_{\mathbb{R}^{d}}\|f(s,x)\|_{K}^{p}dxds.
\end{align*}
\end{lemma}

A proof of Lemma \ref{lem:L psi potential} will be given in Appendix \ref{sec: potential}.

\begin{lemma}\label{lem: estimate of solution}
Let $\varphi\in\mathfrak{M}_{\gamma}(\mathbb{R}^{d})$ for some $\gamma\equiv\gamma_{\varphi}>0$
satisfying the condition \textbf{(S2)} and let $\psi\in\mathfrak{S}$.
Let $\boldsymbol{r}$ be the constant appeared in \textbf{(R2)} and let $q\geq \max\{2,\boldsymbol{r}\}$.
Let $0<T<\infty$.
Then  for any $p\geq q$, there exists a constant $C>0$
depending on $p,d,\gamma_{\psi},\gamma_{\varphi},\kappa_{\psi},\kappa_{\varphi},\mu_{\psi},\mu_{\varphi},\boldsymbol{r}$
and $T$ such that for any $m\in\mathbb{R}$ and
$g\in {\mathbb{L}_{\boldsymbol{r},\boldsymbol{\beta}}^{1,p}(H_{K\otimes U_{0},p}^{\varphi,m},T)}$,
\begin{align}\label{eqn: esti stoc term}
\mathbb{E}\int_{0}^{T}\left\|\int_{0}^{t}\mathcal{T}_{\psi}(t,s)g(s,\cdot)\delta\boldsymbol{\beta}_{s}
                                        \right\|_{H_{K,p}^{\varphi,m+\frac{2\gamma_{\psi}}{q\gamma_{\varphi}}}}^{p}dt
\leq C\|g\|_{\mathbb{L}_{\boldsymbol{r},\boldsymbol{\beta}}^{1,p}(H_{K\otimes U_{0},p}^{\varphi,m},T)}^{p}.
\end{align}
\end{lemma}

\begin{proof}
Let $u(t,x)=\int_{0}^{t}\mathcal{T}_{\psi}(t,s)g(s,x)\delta\boldsymbol{\beta}_{s}$
and let $\alpha=\frac{2\gamma_{\psi}}{q\gamma_{\varphi}}$.
By (i) in Proposition \ref{prop: psi bessel}, we may assume that $m=0$. Then by (iii) in Proposition \ref{prop: psi bessel}, we have
\begin{align*}
\|u(t,\cdot)\|_{H_{K,p}^{\varphi,\alpha}}
&\leq C_{1}\left(\|u(t,\cdot)\|_{L_{K}^{p}}+\|L_{\varphi}^{\frac{\alpha}{2}}u(t,\cdot)\|_{L_{K}^{p}}\right).
\end{align*}
for some constant $C_{1}>0$. Then we have
\begin{align}\label{eqn: sto-esti 1}
\mathbb{E}\int_{0}^{T}\left\|u(t,\cdot)\right\|_{H_{K,p}^{\varphi,\alpha}}^{p}dt
\leq C_{1}\left(I_{1}+I_{2}\right),
\end{align}
where
\begin{align*}
I_{1}=\mathbb{E}\int_{0}^{T}\int_{\mathbb{R}^{d}}\|u(t,x)\|_{K}^{p}dxdt,\quad
I_{2}=\mathbb{E}\int_{0}^{T}\int_{\mathbb{R}^{d}}\|L_{\varphi}^{\frac{\alpha}{2}}u(t,x)\|_{K}^{p}dxdt.
\end{align*}
Firstly we estimate $I_{1}$. Note that by the approximation method, we have
\begin{align}\label{eqn: DT TD}
D_{\theta}^{\boldsymbol{\beta}}\mathcal{T}_{\psi}(t,s)g(s,\cdot)(x)
=\mathcal{T}_{\psi}(t,s)D_{\theta}^{\boldsymbol{\beta}}g(s,\cdot)(x).
\end{align}
Indeed, if
\begin{align*}
g(t,x)=\sum_{i=1}^{l}F_{i}1_{(t_{i-1},t_{i}]}(t)\phi_{i}(x),\qquad t\in [0,T], \quad x\in\mathbb{R}^{d},
\end{align*}
where $F_{i}\in \mathcal{C}_{\boldsymbol{\beta}}$, $0\leq t_{0}<t_{1}<\cdots <t_{l}\leq T$,
and $\phi_{i}\in C_{\mathrm{c},K\otimes U_{0}}^{\infty}$, then we obtain that
\begin{align*}
D_{\theta}^{\boldsymbol{\beta}}\mathcal{T}_{\psi}(t,s)g(s,\cdot)(x)
&=\sum_{i=1}^{l}\left(D_{\theta}^{\boldsymbol{\beta}}F_{i}\right)1_{(t_{i-1},t_{i}]}(t)\mathcal{T}_{\psi}(t,s)\phi_{i}(x)\\
&=\mathcal{T}_{\psi}(t,s)\left(\sum_{i=1}^{l}\left(D_{\theta}^{\boldsymbol{\beta}}F_{i}\right)
    1_{(t_{i-1},t_{i}]}(t)\phi_{i}(\cdot)\right)(x)\\
&=\mathcal{T}_{\psi}(t,s)D_{\theta}^{\boldsymbol{\beta}}g(s,\cdot)(x).
\end{align*}
By using \eqref{eqn: cor maximal inequality for single gaussian process}, \eqref{eqn: DT TD}, Minkowski's inequality,
and the fact that $\|\mathcal{T}_{\psi}(t,s)g(s,\cdot)\|_{L_{K}^{p}}\leq C_{2}\|g(s,\cdot)\|_{L_{K}^{p}}$,
we obtain that
\begin{align}\label{eqn: sto-esti 2}
I_{1}
&\leq C_{3}\left\{\mathbb{E}\int_{\mathbb{R}^{d}}\int_{0}^{T}\int_{0}^{t}
\|\mathcal{T}_{\psi}(t,s)g(s,\cdot)(x)\|_{K\otimes U_{0}}^{p}dsdtdx\right.\nonumber\\
&\left.\qquad+\mathbb{E}\int_{\mathbb{R}^{d}}\int_{0}^{T}\int_{0}^{t}\left(\int_{0}^{T}
\|D_{\theta}^{\boldsymbol{\beta}}\mathcal{T}_{\psi}(t,s)g(s,\cdot)(x)\|_{K\otimes U_{0}\otimes U_{0}}^{\boldsymbol{r}}d\theta\right)^{\frac{p}{\boldsymbol{r}}}dsdtdx\right\}\nonumber\\
&\leq C_{3}\left\{\mathbb{E}\int_{0}^{T}\int_{0}^{t}\int_{\mathbb{R}^{d}}
    \|\mathcal{T}_{\psi}(t,s)g(s,\cdot)(x)\|_{K\otimes U_{0}}^{p}dxdsdt\right.\nonumber\\
&\left.\qquad+\mathbb{E}\int_{0}^{T}\int_{0}^{t}\left(\int_{0}^{T}
\left(\int_{\mathbb{R}^{d}}\|\mathcal{T}_{\psi}(t,s)D_{\theta}^{\boldsymbol{\beta}}g(s,\cdot)(x)\|_{K\otimes U_{0}\otimes U_{0}}^{p}dx\right)^{\frac{\boldsymbol{r}}{p}}d\theta\right)^{\frac{p}{\boldsymbol{r}}}dsdt\right\}\nonumber\\
&\leq C_{3}C_{2}^{p}\left\{\mathbb{E}\int_{0}^{T}\int_{0}^{t}
\int_{\mathbb{R}^{d}}\|g(s,x)\|_{K\otimes U_{0}}^{p}dxdsdt\right.\nonumber\\
&\left.\qquad+\mathbb{E}\int_{0}^{T}\int_{0}^{t}\left(\int_{0}^{T}\left(\int_{\mathbb{R}^{d}}
\|D_{\theta}^{\boldsymbol{\beta}}g(s,x)\|_{K\otimes U_{0}\otimes U_{0}}^{p}dx\right)^{\frac{\boldsymbol{r}}{p}}d\theta\right)^{\frac{p}{\boldsymbol{r}}}dsdt\right\}\nonumber\\
&\leq C_{3}C_{2}^{p}T\|g\|_{\mathbb{L}_{\boldsymbol{r},\boldsymbol{\beta}}^{1,p}(L_{K\otimes U_{0}}^{p},T)}^{p}.
\end{align}
We now estimate $I_{2}$. Note that If we put $\widetilde{\varphi}(\xi)=\varphi(\xi)^{\frac{\alpha}{2}}$,
then $\widetilde{\varphi}\in\mathfrak{M}_{\gamma_{\widetilde{\varphi}}}(\mathbb{R}^{d})$ with the order $\gamma_{\widetilde{\varphi}}=\gamma_{\psi}/q$.
By the definition of $L_{\widetilde{\varphi}}$ and the stochastic Fubini theorem,
it holds that
\begin{align*}
L_{\widetilde{\varphi}}u(t,x)
=L_{\widetilde{\varphi}}\int_{0}^{t}\mathcal{T}_{\psi}(t,s)g(s,x)\delta\boldsymbol{\beta}_{s}
=\int_{0}^{t}L_{\widetilde{\varphi}}\mathcal{T}_{\psi}(t,s)g(s,x)\delta\boldsymbol{\beta}_{s}.
\end{align*}
Also, by using the same argument used in the proof of Lemma \ref{lmm:commutativity between D and 1-Delta} and \eqref{eqn: DT TD},
we see that
\begin{align*}
D_{\theta}^{\boldsymbol{\beta}}L_{\widetilde{\varphi}}\mathcal{T}_{\psi}(t,s)g(s,\cdot)(x)
=L_{\widetilde{\varphi}}\mathcal{T}_{\psi}(t,s)D_{\theta}^{\boldsymbol{\beta}}g(s,\cdot)(x).
\end{align*}
Therefore, by applying \eqref{eqn : cor of max.ineq for single gaussian process+1},
\eqref{ineq:LP-ineq} and \eqref{ineq: 2nd LP ineq FI} with $r=\boldsymbol{r}$, we obtain that
\begin{align}\label{eqn: sto-esti 3}
I_{2}
&=\mathbb{E}\int_{0}^{T}\int_{\mathbb{R}^{d}}\|L_{\varphi}^{\frac{\alpha}{2}}u(t,x)\|_{K}^{p}dxdt
=\mathbb{E}\int_{0}^{T}\int_{\mathbb{R}^{d}}\|L_{\widetilde{\varphi}}u(t,x)\|_{K}^{p}dxdt\nonumber\\
&\leq C\left\{\mathbb{E}\int_{\mathbb{R}^{d}}\int_{0}^{T}\left(\int_{0}^{t}
\|L_{\widetilde{\varphi}}\mathcal{T}_{\psi}(t,s)g(s,\cdot)(x)\|_{K\otimes U_{0}}^{q}ds\right)^{\frac{p}{q}}dtdx\right.\nonumber\\
&\left.\qquad+\mathbb{E}\int_{\mathbb{R}^{d}}\int_{0}^{T}\left(\int_{0}^{t}\left(\int_{0}^{T}
\|D_{\theta}^{\boldsymbol{\beta}}L_{\widetilde{\varphi}}
        \mathcal{T}_{\psi}(t,s)g(s,\cdot)(x)\|_{K\otimes U_{0}\otimes U_{0}}^{\boldsymbol{r}}
        d\theta\right)^{\frac{q}{\boldsymbol{r}}}ds\right)^{\frac{p}{q}}dtdx\right\}\nonumber\\
&\leq C'\left\{\mathbb{E}
\int_{\mathbb{R}^{d}}\int_{0}^{T}\|g(t,x)\|_{K\otimes U_{0}}^{p}dtdx\right.\nonumber\\
&\left.\qquad+\mathbb{E}\int_{0}^{T}\left(\int_{0}^{T}\left(\int_{\mathbb{R}^{d}}
    \|D_{\theta}^{\boldsymbol{\beta}}g(t,x)\|_{K\otimes U_{0}\otimes U_{0}}^{p}dx\right)^{\frac{\boldsymbol{r}}{p}}d\theta\right)^{\frac{p}{\boldsymbol{r}}}dt\right\}\nonumber\\
&=C'\|g\|_{\mathbb{L}_{\boldsymbol{r},\boldsymbol{\beta}}^{1,p}(L_{K\otimes U_{0}}^{p},T)}^{p}
\end{align}
for some positive constants $C$ and $C'$.
Hence, by combining \eqref{eqn: sto-esti 1}, \eqref{eqn: sto-esti 2} and \eqref{eqn: sto-esti 3},
we have the desired result.
\end{proof}

\begin{remark}
\upshape
Let $l_{2}$ be the space of all square summable sequences of real numbers.
In the case of that $\boldsymbol{\beta}$ is a $l_{2}$-valued fBm with the Hurst index $H\in (1/2,1)$,
the author in \cite{Balan 2011} proved \eqref{eqn: esti stoc term}
with $\psi=\varphi=|\xi|^{2}$, $q=2$, $\boldsymbol{r}=1/H$, $K=\mathbb{R}$ and $U_{0}=l_{2}$
(see, the proof of Theorem 5.6 in \cite{Balan 2011}).
\end{remark}

The following is the main theorem in this section.

\begin{theorem}\label{thm: unique existence of sol of Cauchy problem}
Let $\varphi\in\mathfrak{M}_{\gamma}(\mathbb{R}^{d})$ for some $\gamma\equiv\gamma_{\varphi}>0$
satisfying the condition \textbf{(S2)} and let $\psi\in\mathfrak{S}$.
Assume that $N_{\varphi},N_{\psi}>d+2+\lfloor\gamma_{\varphi}\rfloor+\lfloor\gamma_{\psi}\rfloor$.
Let $\boldsymbol{r}$ be the constant appeared in \textbf{(R2)}.
Let $q\geq \max\{2,\boldsymbol{r}\}$ and let $q'$ be the conjugate number of $q$.
Let $p\geq q$ and let $s\in\mathbb{R}$. Suppose that
\begin{align*}
f\in\mathbb{H}_{K,p,T}^{\varphi,s},\,\, g\in\widetilde{\mathbb{L}}_{\boldsymbol{r},\boldsymbol{\beta}}^{1,p}
(H_{K\otimes U_{0},p}^{\varphi,s+\frac{2\gamma_{\psi}}{q'\gamma_{\varphi}}},T),\,\,
u_{0}\in L^{p}(\Omega,\mathcal{F},H_{K,p}^{\varphi,s+\frac{2\gamma_{\psi}}{\gamma_{\varphi}}(1-\frac{1}{p})}).
\end{align*}
Then the SPDE given in \eqref{eqn :SPDE with pseudo diff op}
with initial condition $u(0,\cdot)=u_{0}$ has a unique solution
$u\in\mathcal{H}_{K,p,\boldsymbol{r},q,T}^{\varphi, s+\frac{2\gamma_{\psi}}{\gamma_{\varphi}}}$.
For the solution $u$, there exists a constant $C>0$ depending on $p,d,\gamma_{\psi},\gamma_{\varphi},\kappa_{\psi},\kappa_{\varphi},\mu_{\psi},\mu_{\varphi},\boldsymbol{r}$ and $T$
such that
\begin{align}\label{eqn: inequality for the norm of solution}
\|u\|_{\mathcal{H}_{K,p,\boldsymbol{r},q,T}^{\varphi, s+\frac{2\gamma_{\psi}}{\gamma_{\varphi}}}}
\leq C\left\{(\mathbb{E}\|u_{0}\|_{H_{K,p}^{\varphi,s+\frac{2\gamma_{\psi}}{\gamma_{\varphi}}(1-\frac{1}{p})}}^{p})^{\frac{1}{p}}
+\|f\|_{\mathbb{H}_{K,p,T}^{\varphi,s}}
+\|g\|_{\mathbb{L}_{\boldsymbol{r},\boldsymbol{\beta}}^{1,p}
    (H_{K\otimes U_{0},p}^{\varphi,s+\frac{2\gamma_{\psi}}{q'\gamma_{\varphi}}},T)}\right\}.
\end{align}
\end{theorem}

\begin{proof}
By (i) in Proposition \ref{prop: psi bessel}, we may assume that $s=0$.

\textbf{Step 1.}
We consider the following deterministic equation
\begin{align}\label{eqn: deter eq}
\frac{du}{dt}(t,x)=L_{\psi}(t)u(t,x),\quad u(0,x)=u_{0}(x).
\end{align}
Then
\begin{align*}
u(t,x)=\mathcal{T}_{\psi}(t,0)u_{0}(x)
\end{align*}
is a solution to \eqref{eqn: deter eq}. By the definition of the norm of $\mathcal{H}_{K,p,\boldsymbol{r},q,T}^{\varphi,\frac{2\gamma_{\psi}}{\gamma_{\varphi}}}$, we have
\begin{align}\label{eqn: sol. sp norm ini}
\|u\|_{\mathcal{H}_{K,p,\boldsymbol{r},q,T}^{\varphi,\frac{2\gamma_{\psi}}{\gamma_{\varphi}}}}
=&\|u\|_{\mathbb{H}_{K,p,T}^{\varphi,\frac{2\gamma_{\psi}}{\gamma_{\varphi}}}}
        +\|L_{\psi}u\|_{\mathbb{H}_{K,p,T}^{\varphi,0}}
+\left(\mathbb{E}\|u_{0}\|_{H_{K,p}^{\varphi,\frac{2\gamma_{\psi}}{\gamma_{\varphi}}(1-\frac{1}{p})}}^{p}\right)^{\frac{1}{p}}.
\end{align}
Firstly we estimate $\|u\|_{\mathbb{H}_{K,p,T}^{\varphi,\frac{2\gamma_{\psi}}{\gamma_{\varphi}}}}$.
Note that by applying Jensen's inequality and (iii) in Proposition \ref{prop: psi bessel}, we obtain that
\begin{align}
\|u\|_{\mathbb{H}_{K,p,T}^{\varphi,\frac{2\gamma_{\psi}}{\gamma_{\varphi}}}}^{p}
&=\mathbb{E}\int_{0}^{T}\|u(t,\cdot)\|_{H_{K,p}^{\varphi,\frac{2\gamma_{\psi}}{\gamma_{\varphi}}}}^{p}dt\nonumber\\
&\leq C_{1}\mathbb{E}\int_{0}^{T}\left(\|u\|_{L_{K}^{p}}
            +\|L_{\varphi}^{\frac{\gamma_{\psi}}{\gamma_{\varphi}}}u\|_{L_{K}^{p}}\right)^{p}dt\nonumber\\
&\leq C_{1}2^{p-1}\mathbb{E}\int_{0}^{T}\left(\|u\|_{L_{K}^{p}}^{p}
            +\|L_{\varphi}^{\frac{\gamma_{\psi}}{\gamma_{\varphi}}}u\|_{L_{K}^{p}}^{p}\right)dt\nonumber\\
&= C_{1}2^{p-1}\left(\|u\|_{\mathbb{H}_{K,p,T}^{\varphi,0}}^{p}
+\|L_{\varphi}^{\frac{\gamma_{\psi}}{\gamma_{\varphi}}}u\|_{\mathbb{H}_{K,p,T}^{\varphi,0}}^{p}\right).
\label{eqn: esti sol homoeq}
\end{align}
For $\|u\|_{\mathbb{H}_{K,p,T}^{\varphi,0}}^{p}$,
by applying the fact that $\|\mathcal{T}_{\psi}(t,0)h\|_{L_{K}^{p}}\leq C_{2}\|h\|_{L_{K}^{p}}$
and (ii) in Proposition \ref{prop: psi bessel},
we have
\begin{align}
\|u\|_{\mathbb{H}_{K,p,T}^{\varphi,0}}^{p}
&=\mathbb{E}\int_{0}^{T}\|\mathcal{T}_{\psi}(t,0)u_{0}\|_{L_{K}^{p}}^{p}dt\nonumber\\
&\leq C_{2}\mathbb{E}\int_{0}^{T}\|u_{0}\|_{L_{K}^{p}}^{p}dt\nonumber\\
&\leq C_{2}TC_{3}\mathbb{E}\|u_{0}\|_{H_{K,p}^{\varphi, \frac{2\gamma_{\psi}}{\gamma_{\varphi}}\left(1-\frac{1}{p}\right)}}^{p}.
\label{eqn: esti1 in homoeq}
\end{align}
On the other hand, for $\|L_{\varphi}^{\frac{\gamma_{\psi}}{\gamma_{\varphi}}}u\|_{\mathbb{H}_{K,p,T}^{\varphi,0}}^{p}$,
we note that if we put
\begin{align}\label{eqn: tilde varphi}
\widetilde{\varphi}(\xi)=\varphi(\xi)^{\frac{\gamma_{\psi}}{\gamma_{\varphi}}},
\end{align}
then by the definition of $\mathfrak{M}_{\gamma}(\mathbb{R}^{d})$,
we see that $\widetilde{\varphi}^{\frac{1}{p}}\in\mathfrak{M}_{\gamma}(\mathbb{R}^{d})$
with the order $\gamma_{\widetilde{\varphi}^{\frac{1}{p}}}=\frac{\gamma_{\psi}}{p\gamma_{\varphi}}$.
By applying Lemma \ref{lem: initial condition esti} with $\varphi=\widetilde{\varphi}^{\frac{1}{p}}$ and $s=0$,
and applying (iii) in Proposition \ref{prop: psi bessel},
we obtain that
\begin{align}
\|L_{\varphi}^{\frac{\gamma_{\psi}}{\gamma_{\varphi}}}u\|_{\mathbb{H}_{K,p,T}^{\varphi,0}}^{p}
&= \|L_{\widetilde{\varphi}}u\|_{\mathbb{H}_{K,p,T}^{\varphi, 0}}^{p}\nonumber\\
&=\mathbb{E}\int_{0}^{T}\int_{\mathbb{R}^{d}}\|L_{\widetilde{\varphi}}^{\frac{1}{p}}\mathcal{T}_{\psi}(t,0)
        L_{\widetilde{\varphi}}^{1-\frac{1}{p}}u_{0}(x)\|_{K}^{p}dxdt\nonumber\\
&\leq C_{4}\mathbb{E}\int_{\mathbb{R}^{d}}\|L_{\widetilde{\varphi}}^{1-\frac{1}{p}}u_{0}(x)\|_{K}^{p}dx
        \nonumber\\
&=C_{4}\mathbb{E} \|L_{\varphi}^{\frac{\gamma_{\psi}}{\gamma_{\varphi}}\left(1-\frac{1}{p}\right)}u_{0}\|_{L_{K}^{p}}^{p}\nonumber\\
&\leq C_{4}C_{5}\mathbb{E}\|u_{0}\|
_{H_{K,p}^{\varphi, \frac{2\gamma_{\psi}}{\gamma_{\varphi}}\left(1-\frac{1}{p}\right)}}^{p}.\label{eqn:deter Lv}
\end{align}
Therefore, by \eqref{eqn: esti sol homoeq}, \eqref{eqn: esti1 in homoeq} and \eqref{eqn:deter Lv},
we obtain that
\begin{align}
\|u\|_{\mathbb{H}_{K,p,T}^{\varphi,\frac{2\gamma_{\psi}}{\gamma_{\varphi}}}}^{p}
&\leq C_{1}2^{p-1}\left(\|u\|_{\mathbb{H}_{K,p,T}^{\varphi,0}}^{p}
+\|L_{\varphi}^{\frac{\gamma_{\psi}}{\gamma_{\varphi}}}u\|_{\mathbb{H}_{K,p,T}^{\varphi,0}}^{p}\right)\nonumber\\
&\leq C_{1}2^{p-1}\left(C_{2}TC_{3}+C_{4}C_{5}\right)
\mathbb{E}\|u_{0}\|_{H_{K,p}^{\varphi, \frac{2\gamma_{\psi}}{\gamma_{\varphi}}
                                    \left(1-\frac{1}{p}\right)}}^{p}.\label{eqn:esti2 in homoeq}
\end{align}

We now estimate $\|L_{\psi}u\|_{\mathbb{H}_{K,p,T}^{\varphi,0}}$.
Note that by the condition \textbf{(S2)} for $\psi(t,\xi)$ and Proposition \ref{prop: multiplier},
we see that for each $t>0$, the function $m(\xi)=\frac{\psi(t,\xi)}{(1+\varphi(\xi))^{\frac{\gamma_{\psi}}{\gamma_{\varphi}}}}$
is a Fourier multiplier on $L_{K}^{p}$.
Therefore, we obtain that there exists a constant $C_{6}>0$ such that
for any $v\in\mathbb{H}_{K,p,T}^{\varphi,\frac{2\gamma_{\psi}}{\gamma_{\varphi}}}$,
\begin{align}
\|L_{\psi}v\|_{\mathbb{H}_{K,p,T}^{\varphi,0}}^{p}
&=\mathbb{E}\int_{0}^{T}\|L_{\psi}(t)v(t,\cdot)\|_{L_{K}^{p}}^{p}dt\nonumber\\
&=\mathbb{E}\int_{0}^{T}\left\|\mathcal{F}^{-1}\left(\frac{\psi(t,\xi)}{(1+\varphi(\xi))^{\frac{\gamma_{\psi}}{\gamma_{\varphi}}}}
    (1+\varphi(\xi))^{\frac{\gamma_{\psi}}{\gamma_{\varphi}}}\mathcal{F}v(t,\xi)\right)\right\|_{L_{K}^{p}}dt\nonumber\\
 &\leq C_{6} \mathbb{E}\int_{0}^{T} \left\|\mathcal{F}^{-1}\left(
    (1+\varphi(\xi))^{\frac{\gamma_{\psi}}{\gamma_{\varphi}}}\mathcal{F}v(t,\xi)\right)\right\|_{L_{K}^{p}} dt\nonumber\\
 &=C_{6}\|v\|_{\mathbb{H}_{K,p,T}^{\varphi,\frac{2\gamma_{\psi}}{\gamma_{\varphi}}}}^{p}. \label{eqn:L psi bound}
\end{align}
Hence, by combining \eqref{eqn: sol. sp norm ini}, \eqref{eqn:esti2 in homoeq} and \eqref{eqn:L psi bound},
we obtain that
\begin{align}
\|u\|_{\mathcal{H}_{K,p,\boldsymbol{r},q,T}^{\varphi, \frac{2\gamma_{\psi}}{\gamma_{\varphi}}}}^{p}
&\leq 3^{p-1}\left(\|u\|_{\mathbb{H}_{K,p,T}^{\varphi,\frac{2\gamma_{\psi}}{\gamma_{\varphi}}}}^{p}
        +\|L_{\psi}u\|_{\mathbb{H}_{K,p,T}^{\varphi,0}}^{p}
+\mathbb{E}\|u_{0}\|_{H_{K,p}^{\varphi,\frac{2\gamma_{\psi}}{\gamma_{\varphi}}(1-\frac{1}{p})}}^{p}\right)\nonumber\\
&\leq 3^{p-1}C_{7}\mathbb{E}\|u_{0}\|
_{H_{K,p}^{\varphi, \frac{2\gamma_{\psi}}{\gamma_{\varphi}}\left(1-\frac{1}{p}\right)}}^{p},\label{eqn: norm sol ini}
\end{align}
which proves \eqref{eqn: inequality for the norm of solution} for the case $f=0$ and $g=0$.

\textbf{Step 2.} We consider the following deterministic equation
\begin{align}\label{eqn: deter eq 1}
\frac{du}{dt}(t,x)=L_{\psi}(t)u(t,x)+f(t,x),\qquad u(0,x)=0.
\end{align}
Then
\begin{align*}
u(t,x)=\int_{0}^{t}\mathcal{T}_{\psi}(t,s)f(s,x)ds
\end{align*}
is a solution to \eqref{eqn: deter eq 1}. By the definition of the norm of $\mathcal{H}_{K,p,\boldsymbol{r},q,T}^{\varphi,\frac{2\gamma_{\psi}}{\gamma_{\varphi}}}$, we have
\begin{align}\label{eqn: sol. sp norm 1}
\|u\|_{\mathcal{H}_{K,p,\boldsymbol{r},q,T}^{\varphi,\frac{2\gamma_{\psi}}{\gamma_{\varphi}}}}
=&\|u\|_{\mathbb{H}_{K,p,T}^{\varphi,\frac{2\gamma_{\psi}}{\gamma_{\varphi}}}}
        +\|L_{\psi}u+f\|_{\mathbb{H}_{K,p,T}^{\varphi,0}}.
\end{align}
Firstly we estimate $\|u\|_{\mathbb{H}_{K,p,T}^{\varphi,\frac{2\gamma_{\psi}}{\gamma_{\varphi}}}}$.
From \eqref{eqn: esti sol homoeq}, we have
\begin{align}
\|u\|_{\mathbb{H}_{K,p,T}^{\varphi,\frac{2\gamma_{\psi}}{\gamma_{\varphi}}}}^{p}
&\leq C_{1}2^{p-1}\left(\|u\|_{\mathbb{H}_{K,p,T}^{\varphi,0}}^{p}
+\|L_{\varphi}^{\frac{\gamma_{\psi}}{\gamma_{\varphi}}}u\|_{\mathbb{H}_{K,p,T}^{\varphi,0}}^{p}\right).
\label{eqn: esti sol inhomoeq}
\end{align}
On the other hand, by applying Jensen's inequality and the fact that $\|\mathcal{T}_{\psi}(t,s)h\|_{L_{K}^{p}}\leq C_{9}\|h\|_{L_{K}^{p}}$, we obtain that
\begin{align}
\|u\|_{\mathbb{H}_{K,p,T}^{\varphi,0}}^{p}
&=\mathbb{E}\int_{0}^{T}\left\|\int_{0}^{t}\mathcal{T}_{\psi}(t,s)f(s,\cdot)ds\right\|_{L_{K}^{p}}^{p}dt\nonumber\\
&\leq\mathbb{E}\int_{0}^{T}t^{p-1}\int_{0}^{t}\left\|\mathcal{T}_{\psi}(t,s)f(s,\cdot)\right\|_{L_{K}^{p}}^{p}dsdt\nonumber\\
&\leq C_{8}\mathbb{E}\int_{0}^{T}t^{p-1}\int_{0}^{t}\left\|f(s,\cdot)\right\|_{L_{K}^{p}}^{p}dsdt\nonumber\\
&\leq C_{8}T^{p}\|f\|_{\mathbb{H}_{K,p,T}^{\varphi,0}}^{p}.\label{eqn: deter esti}
\end{align}
Also, by applying Lemma \ref{lem:L psi potential}, we obtain that
\begin{align}
\|L_{\varphi}^{\frac{\gamma_{\psi}}{\gamma_{\varphi}}}u\|_{\mathbb{H}_{K,p,T}^{\varphi,0}}^{p}
&= \|L_{\widetilde{\varphi}}u\|_{\mathbb{H}_{K,p,T}^{\varphi, 0}}^{p}\nonumber\\
&=\mathbb{E}\int_{0}^{T}\int_{\mathbb{R}^{d}}\left\|\int_{0}^{t}L_{\widetilde{\varphi}}\mathcal{T}_{\psi}(t,s)f(s,x)ds
\right\|_{K}^{p}dxdt\nonumber\\
&\leq C_{9}\mathbb{E}\int_{0}^{T}\int_{\mathbb{R}^{d}}\|f(s,x)\|_{K}^{p}dxds
        \nonumber\\
&=C_{9}\|f\|_{\mathbb{H}_{K,p,T}^{\varphi, 0}}^{p}.\label{eqn:deter Lv 1}
\end{align}
Therefore, by combining \eqref{eqn: esti sol inhomoeq}, \eqref{eqn: deter esti} and \eqref{eqn:deter Lv 1},
we obtain that
\begin{align}
\|u\|_{\mathbb{H}_{K,p,T}^{\varphi,\frac{2\gamma_{\psi}}{\gamma_{\varphi}}}}^{p}
&\leq C_{1}2^{p-1}\left(\|u\|_{\mathbb{H}_{K,p,T}^{\varphi,0}}^{p}
+\|L_{\varphi}^{\frac{\gamma_{\psi}}{\gamma_{\varphi}}}u\|_{\mathbb{H}_{K,p,T}^{\varphi,0}}^{p}\right)\nonumber\\
&\leq C_{1}2^{p-1}\left(C_{8}+C_{9}\right)\|f\|_{\mathbb{H}_{K,p,T}^{\varphi,0}}^{p}.\label{eqn: deter esti 2}
\end{align}
We now estimate $\|L_{\psi}u+f\|_{\mathbb{H}_{K,p,T}^{\varphi,0}}$.
By applying triangle inequality, \eqref{eqn:L psi bound} with $v=u$ and \eqref{eqn: deter esti 2}, we obtain that
\begin{align}
\|L_{\psi}u+f\|_{\mathbb{H}_{K,p,T}^{\varphi,0}}
&\leq \|L_{\psi}u\|_{\mathbb{H}_{K,p,T}^{\varphi,0}}+\|f\|_{\mathbb{H}_{K,p,T}^{\varphi,0}}\nonumber\\
&\leq C_{6}\|u\|_{\mathbb{H}_{K,p,T}^{\varphi,\frac{2\gamma_{\psi}}{\gamma_{\varphi}}}}
    +\|f\|_{\mathbb{H}_{K,p,T}^{\varphi,0}}\nonumber\\
&\leq (C_{6}C_{1}2^{p-1}\left(C_{8}+C_{9}\right)+1)\|f\|_{\mathbb{H}_{K,p,T}^{\varphi,0}}.\label{eqn: deter esti 3}
\end{align}
Hence, by combining \eqref{eqn: sol. sp norm 1}, \eqref{eqn: deter esti 2} and \eqref{eqn: deter esti 3},
we obtain that
\begin{align}
\|u\|_{\mathcal{H}_{K,p,\boldsymbol{r},q,T}^{\varphi, \frac{2\gamma_{\psi}}{\gamma_{\varphi}}}}^{p}
&\leq 2^{p-1}\left(\|u\|_{\mathbb{H}_{K,p,T}^{\varphi,\frac{2\gamma_{\psi}}{\gamma_{\varphi}}}}^{p}
        +\|L_{\psi}u+f\|_{\mathbb{H}_{K,p,T}^{\varphi,0}}^{p}\right)\nonumber\\
&\leq 2^{p-1}\left((1+C_{6})C_{1}2^{p-1}\left(C_{8}+C_{9}\right)+1\right)
\|f\|_{\mathbb{H}_{K,p,T}^{\varphi,0}}^{p},\label{eqn: norm sol pot}
\end{align}
which proves \eqref{eqn: inequality for the norm of solution} for the case $g=0$ and $u_{0}=0$.

\textbf{Step 3.}
Suppose that
\begin{align*}
g(t,x)=\sum_{i=1}^{l}F_{i}1_{(t_{i-1},t_{i}]}(t)\phi_{i}(x),\qquad t\in [0,T], \quad x\in\mathbb{R}^{d},
\end{align*}
where $F_{i}\in \mathcal{C}_{\boldsymbol{\beta}}$, $0\leq t_{0}<t_{1}<\cdots <t_{l}\leq T$,
and $\phi_{i}\in C_{\mathrm{c},K\otimes U_{0}}^{\infty}$.
We consider the following equation
\begin{align}
du(t,x)=L_{\psi}(t)u(t,x)+g(t,x)\delta\boldsymbol{\beta}_{t},\quad
u(0,\cdot)=0.\label{eqn: sto eq.}
\end{align}
It is known that for each $t\in [0,T]$ and $x\in\mathbb{R}^{d}$, the solution of \eqref{eqn: sto eq.}
is given by
\begin{align*}
u(t,x)=\int_{0}^{t}\mathcal{T}_{\psi}(t,r)g(r,x)\delta\boldsymbol{\beta}_{r},
\end{align*}
see \cite{Krylov 1999, Balan 2011}.

By the definition of the norm of $\mathcal{H}_{K,p,\boldsymbol{r},q,T}^{\varphi,\frac{2\gamma_{\psi}}{\gamma_{\varphi}}}$,
we have
\begin{align}\label{eqn: sol. sp norm sto}
\|u\|_{\mathcal{H}_{K,p,\boldsymbol{r},q,T}^{\varphi,\frac{2\gamma_{\psi}}{\gamma_{\varphi}}}}
=&\|u\|_{\mathbb{H}_{K,p,T}^{\varphi,\frac{2\gamma_{\psi}}{\gamma_{\varphi}}}}
        +\|L_{\psi}u\|_{\mathbb{H}_{K,p,T}^{\varphi,0}}
+\|g\|_{\mathbb{L}_{\boldsymbol{r},\boldsymbol{\beta}}^{1,p}
(H_{K\otimes U_{0},p}^{\varphi,\frac{2\gamma_{\psi}}{q'\gamma_{\varphi}}},T)}
\end{align}
We estimate $\|u\|_{\mathbb{H}_{K,p,T}^{\varphi,\frac{2\gamma_{\psi}}{\gamma_{\varphi}}}}$
and $\|L_{\psi}u\|_{\mathbb{H}_{K,p,T}^{\varphi,0}}$.
By Lemma \ref{lem: estimate of solution} with $m=\frac{2\gamma_{\psi}}{q^{\prime}\gamma_{\varphi}}$,
we obtain that
\begin{align}
\|u\|_{\mathbb{H}_{K,p,T}^{\varphi,\frac{2\gamma_{\psi}}{\gamma_{\varphi}}}}^{p}
&= \mathbb{E}\int_{0}^{T}\left\|\int_{0}^{t}\mathcal{T}_{\psi}(t,r)g(r,\cdot)\delta\boldsymbol{\beta}_{s}
        \right\|_{H_{K,p}^{\varphi,\left(\frac{1}{q}+\frac{1}{q'}\right)\frac{2\gamma_{\psi}}{\gamma_{\varphi}}}}^{p}dt\nonumber\\
&\leq C_{10}\|g\|_{\mathbb{L}_{\boldsymbol{r},\boldsymbol{\beta}}^{1,p}
        (H_{K\otimes U_{0},p}^{\varphi,\frac{2\gamma_{\psi}}{q^{\prime}\gamma_{\varphi}}},T)}^{p}\label{eqn: sto esti 1}
\end{align}
for some constant $C_{10}>0$. For $\|L_{\psi}u\|_{\mathbb{H}_{K,p,T}^{\varphi,0}}$,
by applying \eqref{eqn:L psi bound} with $v=u$ and \eqref{eqn: sto esti 1}, we have
\begin{align}
\|L_{\psi}u\|_{\mathbb{H}_{K,p,T}^{\varphi,0}}^{p}
\leq  C_{6}\|u\|_{\mathbb{H}_{K,p,T}^{\varphi,\frac{2\gamma_{\psi}}{\gamma_{\varphi}}}}^{p}
\leq C_{6}C_{10}\|g\|_{\mathbb{L}_{\boldsymbol{r},\boldsymbol{\beta}}^{1,p}
        (H_{K\otimes U_{0},p}^{\varphi,\frac{2\gamma_{\psi}}{q'\gamma_{\varphi}}},T)}^{p}.\label{eqn: sto esti 2}
\end{align}
Therefore, by combining \eqref{eqn: sol. sp norm sto}, \eqref{eqn: sto esti 1} and \eqref{eqn: sto esti 2},
we obtain that
\begin{align}
\|u\|_{\mathcal{H}_{K,p,\boldsymbol{r},q,T}^{\varphi,\frac{2\gamma_{\psi}}{\gamma_{\varphi}}}}^{p}
&\leq 3^{p-1}\left(\|u\|_{\mathbb{H}_{K,p,T}^{\varphi,\frac{2\gamma_{\psi}}{\gamma_{\varphi}}}}^{p}
+\|L_{\psi}u\|_{\mathbb{H}_{K,p,T}^{\varphi,0}}^{p}
+\|g\|_{\mathbb{L}_{\boldsymbol{r},\boldsymbol{\beta}}^{1,p}
    (H_{K\otimes U_{0},p}^{\varphi,\frac{2\gamma_{\psi}}{q'\gamma_{\varphi}}},T)}^{p}\right)\nonumber\\
&\leq 3^{p-1}\left(C_{10}+C_{6}C_{10}+1\right)\|g\|_{\mathbb{L}_{\boldsymbol{r},\boldsymbol{\beta}}^{1,p}
    (H_{K\otimes U_{0},p}^{\varphi,\frac{2\gamma_{\psi}}{q'\gamma_{\varphi}}},T)}^{p},\label{eqn: norm sol sto}
\end{align}
which proves \eqref{eqn: inequality for the norm of solution} for the case $f=0$ and $u_{0}=0$.

\textbf{Step 4.} Finally, if $u_{1},\,u_{2},\,u_{3}$ are solutions of equations
\eqref{eqn: deter eq}, \eqref{eqn: deter eq 1}, \eqref{eqn: sto eq.}, respectively, then $u=u_{1}+u_{2}+u_{3}$ is a solution
of \eqref{eqn :SPDE with pseudo diff op}.
For this solution $u$, by applying \eqref{eqn: norm sol ini}, \eqref{eqn: norm sol pot} and \eqref{eqn: norm sol sto},
we obtain that
\begin{align*}
\|u\|_{\mathcal{H}_{K,p,\boldsymbol{r},q,T}^{\varphi,\frac{2\gamma_{\psi}}{\gamma_{\varphi}}}}^{p}
&\leq 3^{p-1}\left(\|u_{1}\|_{\mathcal{H}_{K,p,\boldsymbol{r},q,T}^{\varphi,\frac{2\gamma_{\psi}}{\gamma_{\varphi}}}}^{p}
        +\|u_{2}\|_{\mathcal{H}_{K,p,\boldsymbol{r},q,T}^{\varphi,\frac{2\gamma_{\psi}}{\gamma_{\varphi}}}}^{p}
        +\|u_{3}\|_{\mathcal{H}_{K,p,\boldsymbol{r},q,T}^{\varphi,\frac{2\gamma_{\psi}}{\gamma_{\varphi}}}}^{p}\right)\\
&\leq 3^{p-1}C\left(\mathbb{E}\|u_{0}\|_{H_{K,p}^{\varphi,\frac{2\gamma_{\psi}}{\gamma_{\varphi}}(1-\frac{1}{p})}}^{p}
        +\|f\|_{\mathbb{H}_{K,p,T}^{\varphi,0}}^{p}
        +\|g\|_{\mathbb{L}_{\boldsymbol{r},\boldsymbol{\beta}}^{1,p}(H_{K\otimes U_{0},p}^{\varphi,\frac{2\gamma_{\psi}}{q'\gamma_{\varphi}}},T)}^{p}\right).
\end{align*}
for some constant $C>0$. The proof is complete.
\end{proof}

\section{Examples of Covariance Kernels for Gaussian Processes}\label{sec: examples}
In this section, we discuss some examples of $Q$-Gaussian process $\boldsymbol{\beta}=\{\boldsymbol{\beta}_t\}_{0\le t\le T}$
for which the conditions \textbf{(R1)} and \textbf{(R2)} hold.

\begin{example}\label{example: Bm}
\upshape
Let $\boldsymbol{\beta}=\{\boldsymbol{\beta}_t\}_{0\le t\le T}$ be a $Q$-Wiener process
with the covariance kernel $R(t,s)$ such that
\begin{align*}
\mathbb{E}[\bilin{\boldsymbol{\beta}_{t}}{u}_{U}\bilin{\boldsymbol{\beta}_{s}}{v}_{U}]&=R(t,s)\bilin{Qu}{v}_{U},
\quad u,v\in U,\\
R(t,s)&=\min\{t,s\},\quad 0\le t,s\le T.
\end{align*}
Then we have
\begin{align*}
\frac{\partial^{2}R}{\partial t \partial s}(t,s)=\delta(t-s)\geq 0  \quad \text{for all }t,s\in [0,T]
\end{align*}
in the sense of distribution, where $\delta$ is the Dirac delta function at zero.
Then the integral kernel operator $K_{R}$ is given by
\begin{align*}
K_{R} f(t)=\int_{0}^{T}f(s)\delta(t-s)ds=f(t),\quad t\in [0,T],
\end{align*}
i.e., $K_{R}$ is the identity operator.
In this case, $\mathcal{H}=L^{2}([0,T])$ and for any $1\le p\le q\le \infty$, we have
\begin{align*}
\|K_{R}f\|_{L^{p}([0,T])}=\|f\|_{L^{p}([0,T])}\le C_{p,q}\|f\|_{L^{q}([0,T])}
\end{align*}
for some constant $C_{p,q}\ge0$.
By taking $\boldsymbol{s}=p\ge1$ and $\boldsymbol{r}=q\ge p$ with $1/p+1/q=1$ and $C_{R}=C_{p,q}$,
we see that the assumptions given as in \textbf{(R1)} and \textbf{(R2)} hold.
\end{example}

\begin{example}\label{example: fBm}
\upshape
Let $\boldsymbol{\beta}=\{\boldsymbol{\beta}_t\}_{0\le t\le T}$ be a $Q$-fractional Brownian motion ($Q$-fBm)
with the covariance kernel $R(t,s)$ such that
\begin{align*}
\mathbb{E}[\bilin{\boldsymbol{\beta}_{t}}{u}_{U}\bilin{\boldsymbol{\beta}_{s}}{v}_{U}]&=R(t,s)\bilin{Qu}{v}_{U},
\quad u,v\in U,\\
R(t,s)&=\frac{1}{2}(t^{2H}+s^{2H}-|t-s|^{2H}),\quad 0\le t,s\le T
\end{align*}
(see \cite{Duncan 2006,Grecksch 2009}), where $H$ is the Hurst parameter with $H\in (0,1)$.
Then we have
\begin{align*}
\frac{\partial^{2}R}{\partial t \partial s}(t,s)=H(2H-1)|t-s|^{2H-2}\geq 0  \quad \text{for all }t,s\in [0,T]
\end{align*}
in the sense of distribution.
Suppose that $H\in (\frac{1}{2},1)$.
Then by applying the Hardy-Littlewood inequality,
we see that the integral kernel operator $K_{R}$ defined by
\begin{align*}
K_{R} f(t)=H(2H-1)\int_{0}^{T}f(s)|t-s|^{2H-2}ds
\end{align*}
satisfies that
\begin{align*}
\|K_{R} f\|_{L^{\frac{1}{1-H}}([0,T])}\leq C\|f\|_{L^{\frac{1}{H}}([0,T])}
\end{align*}
for some constant $C=C_{H}>0$ (see Equation (12) in \cite{Nualart 2003}),
and then by taking $\boldsymbol{r}=\frac{1}{H},\boldsymbol{s}=\frac{1}{1-H}$ and $C_{R}=C$,
we see that the assumptions given as in \textbf{(R1)} and \textbf{(R2)} hold.
\end{example}

\begin{example}
\upshape
A $Q$-fBm $\boldsymbol{\beta}=\{\boldsymbol{\beta}_t\}_{0\le t\le T}$ with the Hurst parameter $H=1$
is defined by
\begin{align*}
\boldsymbol{\beta}_{t}=t \boldsymbol{X}\quad t\in [0,T],
\end{align*}
where $\boldsymbol{X}$ is a $U$-valued standard normal random variable (see Remark 1.2.3 in \cite{Mishura 2008}).
Then we have
\begin{align*}
R(t,s)=ts,\quad \frac{\partial^{2} R}{\partial t \partial s}(t,s)=1\geq 0
\end{align*}
for all $t,s\in [0,T]$, and then for the integral kernel operator $K_{R}$ defined by
\begin{align*}
K_{R}f(t)=\int_{0}^{T}f(s)\frac{\partial^{2} R}{\partial t \partial s}(t,s)ds
=\int_{0}^{T}f(s)ds,
\end{align*}
we have
\begin{align*}
|K_{R}f(t)|\leq \int_{0}^{T}|f(s)|ds=\|f\|_{L^{1}([0,T])}.
\end{align*}
Therefore, we have
\begin{align*}
\|K_{R}f\|_{L^{\infty}([0,T])}\leq \|f\|_{L^{1}([0,T])},
\end{align*}
from which by taking $\boldsymbol{r}=1$, $\boldsymbol{s}=\infty$ and $C_{R}=1$,
we see that the assumptions \textbf{(R1)} and \textbf{(R2)} hold.
\end{example}

\begin{example}
\upshape
Let $0<\delta<1$ and let $\boldsymbol{\beta}=\{\boldsymbol{\beta}_t\}_{0\le t\le T}$
be a $Q$-centered Gaussian process with the covariance kernel $R(t,s)$ such that
\begin{align*}
\mathbb{E}[\bilin{\boldsymbol{\beta}_{t}}{u}_{U}\bilin{\boldsymbol{\beta}_{s}}{v}_{U}]&=R(t,s)\bilin{Qu}{v}_{U},
\quad u,v\in U,\\
\frac{\partial^{2} R}{\partial t \partial s}(t,s)&=\frac{(2\sqrt{\pi})^{-1}}{\Gamma\left(\frac{\delta}{2}\right)}
   \int_{0}^{\infty}e^{-x}e^{-\frac{|t-s|^{2}}{4x}}x^{\frac{\delta-1}{2}}\frac{dx}{x},\quad 0\le t,s\le T,
\end{align*}
where $\Gamma$ is the gamma function.
Then for the integral kernel operator $K_{R}$ defined by
\begin{align*}
K_{R} f(t)=\int_{0}^{T}f(s)\frac{\partial^{2} R}{\partial t \partial s}(t,s)ds,
\end{align*}
we have
\begin{align*}
K_{R} f(t)=G_{\delta}*f(t)
\end{align*}
with the (one-dimensional) Bessel kernel $G_{\delta}$ given by
\begin{align*}
G_{\delta}(t)=\frac{(2\sqrt{\pi})^{-1}}{\Gamma\left(\frac{\delta}{2}\right)}
\int_{0}^{\infty}e^{-x}e^{-\frac{|t|^{2}}{4x}}x^{\frac{\delta-1}{2}}\frac{dx}{x}
\end{align*}
 (see \cite{Hu 2015,Balan 2008}).
By taking $\boldsymbol{r}=\frac{2}{\delta+1}$ and $\boldsymbol{s}=\frac{\boldsymbol{r}}{\boldsymbol{r}-1}$,
 we have
 \begin{align*}
 \frac{1}{\boldsymbol{r}}-\frac{1}{\boldsymbol{s}}=\delta, \quad 1<\boldsymbol{r}<2<\boldsymbol{s}<\infty.
\end{align*}
Therefore, by Corollary 6.1.6 in \cite{Grafakos 2009},
there exists a constant $C=C_{\delta}>0$ such that for any $f\in L^{\boldsymbol{r}}([0,T])$
\begin{align*}
\|K_{R} f\|_{L^{\boldsymbol{s}}([0,T])}\leq C\|f\|_{L^{\boldsymbol{r}}([0,T])},
\end{align*}
and hence the assumptions $\textbf{(R1)}$ and $\textbf{(R2)}$ hold.
\end{example}

\begin{example}
\upshape
Let $\delta>0$ and let $\boldsymbol{\beta}$ be a $Q$-centered Gaussian process with the covariance kernel $R(t,s)$ such that
\begin{align*}
\mathbb{E}[\bilin{\boldsymbol{\beta}_{t}}{u}_{U}\bilin{\boldsymbol{\beta}_{s}}{v}_{U}]&=R(t,s)\bilin{Qu}{v}_{U},
\quad u,v\in U,\\
\frac{\partial^{2} R}{\partial t \partial s}(t,s)&=\frac{1}{(4\pi \delta)^{\frac{1}{2}}}\exp\left(-\frac{|t-s|^{2}}{4\delta}\right)
,\quad 0\le t,s\le T.
\end{align*}
Then for the integral kernel operator $K_{R}$ defined by
\begin{align*}
K_{R} f=H_{\delta}*f
\end{align*}
with the (one-dimensional) heat kernel $H_{\delta}$ given by
\begin{align*}
H_{\delta}(t)=\frac{1}{(4\pi \delta)^{\frac{1}{2}}}\exp\left(-\frac{|t|^{2}}{4\delta}\right)
\end{align*}
(see \cite{Balan 2008}).
For $1\leq \boldsymbol{r},\boldsymbol{s},p\leq \infty$, let $1+\frac{1}{\boldsymbol{s}}=\frac{1}{p}+\frac{1}{\boldsymbol{r}}$ and $\frac{1}{\boldsymbol{r}}+\frac{1}{\boldsymbol{s}}=1$.
By Young's inequality we have
\begin{align*}
\|K_{R} f\|_{L^{\boldsymbol{s}}([0,T])}
&=\|H_{\delta}*f\|_{L^{\boldsymbol{s}}([0,T])}
\leq \|H_{\delta}\|_{L^{p}([0,T])}\|f\|_{L^{\boldsymbol{r}}([0,T])}.
\end{align*}
Since $\|H_{\delta}\|_{L^{p}([0,T])}<\infty$, by taking $C_{R}=\|H_{\delta}\|_{L^{p}([0,T])}$ we have
\begin{align*}
\|K_{R} f\|_{L^{\boldsymbol{s}}([0,T])}\leq C_{R}\|f\|_{L^{\boldsymbol{r}}([0,T])}.
\end{align*}
Therefore, the assumptions $\textbf{(R1)}$ and $\textbf{(R2)}$ hold.
\end{example}

\appendix
\section*{Appendix}\label{sec: Appendix}
\addcontentsline{toc}{section}{Appendix}
\renewcommand{\thesubsection}{\Alph{subsection}}
\counterwithin{theorem}{subsection}
\counterwithin{equation}{subsection}

\subsection{A Proof of Theorem \ref{cor: 2nd LP ineq}}\label{sec:A Proof of LPI-Banach space}

For a notational convenience, we put $V=L^{r}(\mathbb{R};H)$,
and for any $-\infty\le a<b\le \infty$ and $f\in C_{\rm c}^{\infty}((a,b)\times \mathbb{R}^{d};V)$,
we put
\begin{align*}
\mathcal{R}_{a,q}f(t,x)
&=\left[\int_{a}^{t}
(t-s)^{\frac{q\gamma_{\varphi}}{\gamma_{\psi}}-1}\|L_{\varphi}\mathcal{T}_{\psi}(t,s)f(s,\cdot,\cdot)(x)\|_{V}^{q} ds\right]^{\frac{1}{q}}\\
&=\left[\int_{a}^{t}(t-s)^{\frac{q\gamma_{\varphi}}{\gamma_{\psi}}-1}\left(\int_{\mathbb{R}}
\|L_{\varphi}
\mathcal{T}_{\psi}(t,s)f(s,\cdot,\theta)(x)\|_{H}^{r}d\theta\right)^{\frac{q}{r}} ds\right]^{\frac{1}{q}}
\end{align*}
To prove Theorem \ref{cor: 2nd LP ineq}, we apply the arguments used in the proofs of Theorems 4.1 and 5.3 in \cite{Ji-Kim 2025-1}.

\begin{lemma}\label{lem: 2nd LP ineq for p equal lambda}
Let $r\geq 1$ and let $q\geq \max\{2,r\}$ be given.
Then it holds that
\begin{itemize}
  \item [\rm{(i)}] if $-\infty<a\leq c_{1}\leq c_{2}\leq b<\infty$, then
  there exists a constant $C_{1}>0$ depending on $a,b,d,q, \gamma_{\varphi},\gamma_{\psi}$, $\mu_{\varphi},\mu_{\psi}$,
  $\kappa_{\varphi}$ and $\kappa_{\psi}$
such that for any $f\in C_{\rm c}^{\infty}((a,b)\times\mathbb{R}^{d}; V)$,
\begin{align}
&\int_{c_{1}}^{c_{2}}\int_{\mathbb{R}^{d}}|\mathcal{R}_{a,q}f(t,x)|^{q}dxdt\nonumber\\
&\qquad\leq C_{1}\int_{a}^{c_{2}}\left[\int_{\mathbb{R}}\left(\int_{\mathbb{R}^{d}}
\|f(s,x,\theta)\|_{H}^{q}dx\right)^{\frac{r}{q}}d\theta\right]^{\frac{q}{r}} ds
\label{eqn:LP p=q 1}
\end{align}
for which the right hand side of \eqref{eqn:LP p=q 1} is finite,

  \item [\rm{(ii)}] if $r\leq 2$ and $q=2$ and $-\infty\leq c_{1}\leq c_{2}\leq \infty$,
  then there exists a constant $C_{2}>0$ depending on $\mu_{\varphi},\kappa_{\psi}$, $\gamma_{\varphi}$ and $\gamma_{\psi}$
  such that for any $f\in C_{\rm c}^{\infty}(\mathbb{R}^{d+1}; V)$,
\begin{align}
&\int_{c_{1}}^{c_{2}}\int_{\mathbb{R}^{d}}|\mathcal{R}_{-\infty,2}f(t,x)|^2 dxdt\nonumber\\
&\qquad\leq C_{2}\int_{-\infty}^{c_{2}}\left[\int_{\mathbb{R}}\left(\int_{\mathbb{R}^{d}}
\|f(s,x,\theta)\|_{H}^{2}dx\right)^{\frac{r}{2}}d\theta\right]^{\frac{2}{r}} ds
\label{eqn: LP 2}
\end{align}
for which the right hand side of \eqref{eqn: LP 2} is finite.
\end{itemize}
\end{lemma}

\begin{proof}
(i)\enspace By the assumption, we have $q\geq r$ and so
by applying the Fubini theorem and Minkowski's inequality with the measures
$d\theta$ and $(t-s)^{\frac{q\gamma_{\varphi}}{\gamma_{\psi}}-1}dtdx$, we obtain that
\begin{align}
&\int_{\mathbb{R}^{d}}\int_{c_{1}}^{c_{2}}|\mathcal{R}_{a,q}f(t,x)|^{q}dtdx\nonumber\\
&\leq \int_{a}^{c_{2}}\int_{\mathbb{R}^{d}}\int_{a}^{t}(t-s)^{\frac{q\gamma_{\varphi}}{\gamma_{\psi}}-1}\left(\int_{\mathbb{R}}
\|L_{\varphi}
\mathcal{T}_{\psi}(t,s)f(s,\cdot,\theta)(x)\|_{H}^{r}d\theta\right)^{\frac{q}{r}} dsdxdt\nonumber\\
&=\int_{a}^{c_{2}}\int_{\mathbb{R}^{d}}\int_{s}^{c_{2}}
(t-s)^{\frac{q\gamma_{\varphi}}{\gamma_{\psi}}-1}\left(\int_{\mathbb{R}}
\|L_{\varphi}
\mathcal{T}_{\psi}(t,s)f(s,\cdot,\theta)(x)\|_{H}^{r}d\theta\right)^{\frac{q}{r}} dtdxds\nonumber\\
&\leq \int_{a}^{c_{2}}\left(\int_{\mathbb{R}}
\left(\int_{\mathbb{R}^{d}}\int_{s}^{c_{2}}(t-s)^{\frac{q\gamma_{\varphi}}{\gamma_{\psi}}-1}\|L_{\varphi}
\mathcal{T}_{\psi}(t,s)f(s,\cdot,\theta)(x)\|_{H}^{q}dtdx\right)^{\frac{r}{q}}d\theta\right)^{\frac{q}{r}} ds.
\label{eqn:ineq p=q}
\end{align}
By Lemma \ref{lem: initial condition esti},
it holds that there exists a constant $C>0$ such that
\begin{align}
\int_{\mathbb{R}^{d}}\int_{s}^{c_{2}}(t-s)^{\frac{q\gamma_{\varphi}}{\gamma_{\psi}}-1}
\|L_{\varphi}\mathcal{T}_{\psi}(t,s)f(s,\cdot,\theta)(x)\|_{H}^{q}dtdx
\leq C \|f(s,\cdot,\theta)\|_{L_{H}^{q}}^{q}.
\label{eqn:int esti}
\end{align}
Therefore, by combining \eqref{eqn:ineq p=q} and \eqref{eqn:int esti}, we obtain that
\begin{align*}
\int_{\mathbb{R}^{d}}\int_{c_{1}}^{c_{2}}|\mathcal{R}_{a,q}f(t,x)|^{q}dtdx
&\leq C\int_{a}^{c_{2}}\left(\int_{\mathbb{R}}
\|f(s,\cdot,\theta)\|_{L_{H}^{q}}^{r}d\theta\right)^{\frac{q}{r}} ds\\
&= C\int_{a}^{c_{2}}\left(\int_{\mathbb{R}}\left(\int_{\mathbb{R}^{d}}
\|f(s,x,\theta)\|_{H}^{q}dx\right)^{\frac{r}{q}}d\theta\right)^{\frac{q}{r}} ds,
\end{align*}
which proves the inequality given in \eqref{eqn:LP p=q 1}.

(ii)\enspace By the assumption $r\leq 2$ and $q=2$, we have $q\geq r$ and so by applying \eqref{eqn:ineq p=q},
we obtain that
\begin{align}
&\int_{\mathbb{R}^{d}}\int_{c_{1}}^{c_{2}}|\mathcal{R}_{-\infty,2}f(t,x)|^2 dtdx\nonumber\\
&\leq \int_{-\infty}^{c_{2}}\left(\int_{\mathbb{R}}
\left(\int_{\mathbb{R}^{d}}\int_{s}^{c_{2}}(t-s)^{\frac{2\gamma_{\varphi}}{\gamma_{\psi}}-1}
\|L_{\varphi}\mathcal{T}_{\psi}(t,s)f(s,\cdot,\theta)(x)\|_{H}^{2}
dtdx\right)^{\frac{r}{2}}d\theta\right)^{\frac{2}{r}} ds.\label{eqn: c1 c2 q=2}
\end{align}
Note that by applying the Plancherel theorem and the arguments used in the proof of Theorem 3.3 in \cite{Ji-Kim 2025-1},
we have
\begin{align}
&\int_{\mathbb{R}^{d}}\int_{s}^{c_{2}}(t-s)^{\frac{2\gamma_{\varphi}}{\gamma_{\psi}}-1}
   \|L_{\varphi}\mathcal{T}_{\psi}(t,s)f(s,\cdot,\theta)(x)\|_{H}^{2}dtdx\nonumber\\
&\qquad \le \mu_{\varphi}^{2}\Gamma\left(\frac{2\gamma_{\varphi}}{\gamma_{\psi}}\right)(2\kappa_{\psi})^{-\frac{2\gamma_{\varphi}}{\gamma_{\psi}}}
   \int_{\mathbb{R}^{d}}\left\|f(s,x,\theta)\right\|_{H}^{2} dx.\label{eqn: ineq q=2}
\end{align}
By combining \eqref{eqn: c1 c2 q=2} and \eqref{eqn: ineq q=2}, we obtain that
\begin{align*}
&\int_{\mathbb{R}^{d}}\int_{c_{1}}^{c_{2}}|\mathcal{R}_{-\infty,2}f(t,x)|^2 dtdx\nonumber\\
&\leq \mu_{\varphi}^{2}\Gamma\left(\frac{2\gamma_{\varphi}}{\gamma_{\psi}}\right)
(2\kappa_{\psi})^{-\frac{2\gamma_{\varphi}}{\gamma_{\psi}}}
\int_{-\infty}^{c_{2}}\left(\int_{\mathbb{R}}
\left(\int_{\mathbb{R}^{d}}\left\|f(s,x,\theta)\right\|_{H}^{2} dx\right)^{\frac{r}{2}}d\theta\right)^{\frac{2}{r}} ds,
\end{align*}
which proves the inequality given in \eqref{eqn: LP 2}.
\end{proof}

For $x\in\mathbb{R}^{d}$ and $r>0$, we denote $B_{r}(x)=\{y\in \mathbb{R}^{d}\,:\, |x-y|<r\}$ and $B_{r}=B_{r}(0)$.

\begin{lemma} \label{lem: g equal 0 outside of B3r1 2nd}
Let $r\geq 1$ and let $q\geq \max\{2,r\}$ be given.
Then for any $r_{1}>0$, it holds that

\begin{itemize}
  \item [\rm (i)] if $-\infty<a\leq -2r_{1}<0\leq b<\infty$, then there exists a constant $C_{1}>0$ depending on
  $a,d,q,\gamma_{\varphi},\gamma_{\psi},\mu_{\varphi},\mu_{\psi}$, $\kappa_{\varphi}$ and $\kappa_{\psi}$ such that
  for any $r_{2}>0$ and $f\in C_{\rm c}^{\infty}((a,b)\times\mathbb{R}^{d};V)$
  satisfying that $f(t,x,\cdot)=0$ for $x\notin B_{3r_{2}}$,
\begin{align}
&\int_{-2r_{1}}^{0}\int_{B_{r_{2}}}|\mathcal{R}_{a,q}f(s,y)|^{q}dyds\nonumber\\
&\qquad\leq C_{1}\int_{a}^{0}\left[\int_{\mathbb{R}}\left(\int_{B_{3r_{2}}}
\|f(s,y,\theta)\|_{H}^{q}dy\right)^{\frac{r}{q}}
d\theta\right]^{\frac{q}{r}}ds,\label{eqn:int -2r_{1} B3r}
\end{align}
for which the right hand side of \eqref{eqn:int -2r_{1} B3r} is finite,

  \item [\rm (ii)] if $r\leq 2$ and $q=2$,
there exists a constant $C_{2}>0$ depending on
$d,\gamma_{\varphi},\gamma_{\psi},\mu_{\varphi},\mu_{\psi}$, $\kappa_{\varphi}$ and $\kappa_{\psi}$
such that for any $r_{2}>0$ and $f\in  C_{\rm c}^{\infty}(\mathbb{R}^{d+1};V)$
satisfying that $f(t,x,\cdot)=0$ for $x\notin B_{3r_{2}}$,
\begin{align}
&\int_{-2r_{1}}^{0}\int_{B_{r_{2}}}|\mathcal{R}_{-\infty,2}f(l,s,y)|^{2}dyds\nonumber\\
&\qquad\leq C_{2}\int_{-\infty}^{0}\left[\int_{\mathbb{R}}\left(\int_{B_{3r_{2}}}
\|f(s,y,\theta)\|_{H}^{2}dy\right)^{\frac{r}{2}}
d\theta\right]^{\frac{2}{r}}ds\label{eqn:int -2r' B3r 3}
\end{align}
for which the right hand side of \eqref{eqn:int -2r' B3r 3}
is finite.
\end{itemize}
\end{lemma}

\begin{proof}
(i)\enspace By \eqref{eqn:LP p=q 1} with $c_{1}=-2r_{1}$ and $c_{2}=0$, we obtain that
\begin{align*}
\int_{-2r_{1}}^{0}\int_{B_{r_{2}}}|\mathcal{R}_{a,q}f(s,y)|^{q}dyds
&\leq \int_{-2r_{1}}^{0}\int_{\mathbb{R}^{d}}|\mathcal{R}_{a,q}f(s,y)|^{q}dyds\\
&\leq C_{1}\int_{a}^{0}\left[\int_{\mathbb{R}}\left(\int_{\mathbb{R}^{d}}
\|f(s,y,\theta)\|_{H}^{q}dy\right)^{\frac{r}{q}}
d\theta\right]^{\frac{q}{r}}ds\\
&=C_{1}\int_{a}^{0}\left[\int_{\mathbb{R}}\left(\int_{B_{3r_{2}}}
\|f(s,y,\theta)\|_{H}^{q}dy\right)^{\frac{r}{q}}
d\theta\right]^{\frac{q}{r}}ds,
\end{align*}
which proves \eqref{eqn:int -2r_{1} B3r}.

(ii)\enspace
By \eqref{eqn: LP 2} with $c_{1}=-2r_{1}$ and $c_{2}=0$, we obtain that
\begin{align*}
\int_{-2r_{1}}^{0}\int_{B_{r}}|\mathcal{R}_{-\infty,2}f(s,y)|^{2}dyds
&\leq \int_{-2r_{1}}^{0}\int_{\mathbb{R}^{d}}|\mathcal{R}_{-\infty,q}f(l,s,y)|^{2}dyds\\
&\leq C_{2}\int_{-\infty}^{0}\left[\int_{\mathbb{R}}\left(\int_{\mathbb{R}^{d}}
\|f(s,y,\theta)\|_{H}^{2}dy\right)^{\frac{r}{2}}
d\theta\right]^{\frac{2}{r}}ds\\
&=C_{2}\int_{-\infty}^{0}\left[\int_{\mathbb{R}}\left(\int_{B_{3r_{2}}}
\|f(s,y,\theta)\|_{H}^{2}dy\right)^{\frac{r}{2}}
d\theta\right]^{\frac{2}{r}}ds,
\end{align*}
which proves \eqref{eqn:int -2r' B3r 3}.
\end{proof}

Recall that the definitions of the maximal function and the sharp function.
For any $R\geq 0$ and a real-valued locally integrable function $h$ on $\mathbb{R}^{d}$,
the maximal function $\mathbb{M}_{x}^{R}h(x)$ is defined by
\begin{align*}
\mathbb{M}_{x}^{R}h(x)=\sup_{r>R}\frac{1}{|B_{r}(x)|}\int_{B_{r}(x)}|h(y)|dy,\quad
\mathbb{M}_{x}h(x)=\mathbb{M}_{x}^{0}h(x).
\end{align*}
For any $R_{1}\geq R_{2}\ge0$, it is obvious that
\begin{align*}
\mathbb{M}_{x}^{R_{1}}h(x)\leq \mathbb{M}_{x}^{R_{2}}h(x).
\end{align*}
If $h$ is a real-valued locally integrable function on $\mathbb{R}^{1+d}$,
we denote
\begin{align*}
\mathbb{M}_{x}^{R}h(t,x)=\mathbb{M}_{x}^{R}(h(t,\cdot))(x),\quad
\mathbb{M}_{t}^{R}h(t,x)=\mathbb{M}_{t}^{R}(h(\cdot,x))(t).
\end{align*}

For $r_{1},r_{2}>0$, we put
\begin{align*}
Q_{r_{1},r_{2}}:=(-2r_{1},0)\times B_{r_{2}}.
\end{align*}
\begin{lemma}\label{lem: bound of integral of mathcal G over Qr2r1 2nd}
Let $r\geq 1$, $q\geq \max\{2,r\}$.
Then there exists a constant $C>0$ depending on $d,\mu_{\varphi},\mu_{\psi},\gamma_{\varphi},\gamma_{\psi},\kappa_{\varphi},\kappa_{\psi}$ and $q$
such that for any $r_1,r_2>0$, $x\in B_{r_{2}}$ and
$f\in C_{\rm c}^{\infty}((a,b)\times\mathbb{R}^{d};V)$ (for $-\infty\leq a<b\leq\infty$)
with support in $(-10r_{1},10r_{1})\times \mathbb{R}^{d}\setminus B_{2r_{2}}$,
\begin{align}
&\int_{Q_{r_{1},r_{2}}}|\mathcal{R}_{a,q}f(s,y)|^{q}dsdy\nonumber\\
&\qquad\leq Cr_{2}^{d-\gamma_{\varphi}q}r_{1}^{\frac{q\gamma_{\varphi}}{\gamma_{\psi}}}
   \int_{-10r_{1}}^{0}
\left(\int_{\mathbb{R}}\left(\mathbb{M}_{x}^{3r_{1}}
\|f(r,x,\theta)\|_{H}^{q}\right)^{\frac{r}{q}}d\theta\right)^{\frac{q}{r}}dr.\label{eqn:int-Q-Gg}
\end{align}
\end{lemma}

\begin{proof}
Let $x\in B_{r_{2}}$ and $f\in C_{\rm c}^{\infty}((a,b)\times\mathbb{R}^{d};V)$.
By applying the arguments used in the proof of Lemma B.6 in \cite{Ji-Kim 2025-1}, we have
\begin{align}
\int_{Q_{r_{1},r_{2}}}|\mathcal{R}_{a,q}f(s,y)|^{q}dsdy
&\leq Cr_{2}^{d-\gamma_{\varphi}q}r_{1}^{\frac{q\gamma_{\varphi}}{\gamma_{\psi}}}
   \int_{-10r_{1}}^{0}
\mathbb{M}_{x}^{3r_{2}}\|f(r,x,\cdot)\|_{V}^{q}dr\label{eqn: Lemma A.6}
\end{align}
for some constant $C>0$.
Note that we have for any $x\in B_{r_{2}}$ and $r_{0}\geq 0$,
\begin{align}
\mathbb{M}_{x}^{r_{0}}\|f(r,x,\cdot)\|_{V}^{q}
\leq \left(\int_{\mathbb{R}}\left(\mathbb{M}_{x}^{r_{0}}
\|f(r,x,\theta)\|_{H}^{q}\right)^{\frac{r}{q}}d\theta\right)^{\frac{q}{r}}.
\label{eqn:ineq for maximal V valued}
\end{align}
Indeed, by applying the Minkowski's inequality, we obtain that for any $x\in B_{r_{2}}$, $r_{0}\geq 0$ and $r'>r_{0}$,
\begin{align*}
\int_{B_{r'}(x)}\|f(r,y,\cdot)\|_{V}^{q}dy
&=\int_{B_{r'}(x)}\left(\int_{\mathbb{R}}
    \|f(r,y,\theta)\|_{H}^{r}d\theta\right)^{\frac{q}{r}}dy\\
&\leq \left(\int_{\mathbb{R}}\left(\int_{B_{r'}(x)}
        \|f(r,y,\theta)\|_{H}^{q}dy\right)^{\frac{r}{q}}d\theta\right)^{\frac{q}{r}}\\
&\leq |B_{r'}(x)|\left(\int_{\mathbb{R}}\left(\mathbb{M}_{x}^{r_{0}}
            \|f(r,x,\theta)\|_{H}^{q}\right)^{\frac{r}{q}}d\theta\right)^{\frac{q}{r}},
\end{align*}
which implies \eqref{eqn:ineq for maximal V valued}.
Hence, by combining \eqref{eqn: Lemma A.6}
and \eqref{eqn:ineq for maximal V valued} with $r_{0}=3r_{2}$, we have
\begin{align*}
\int_{Q_{r_{1},r_{2}}}|\mathcal{R}_{a,q}f(s,y)|^{q}dsdy
&\leq Cr_{2}^{d-\gamma_{\varphi}q}r_{1}^{\frac{q\gamma_{\varphi}}{\gamma_{\psi}}}
   \int_{-10r_{1}}^{0}
\left(\int_{\mathbb{R}}\left(\mathbb{M}_{x}^{3r_{2}}
\|f(r,x,\theta)\|_{H}^{q}\right)^{\frac{r}{q}}d\theta\right)^{\frac{q}{r}}dr,
\end{align*}
which gives the inequality given in \eqref{eqn:int-Q-Gg}.
\end{proof}

For a real-valued differentiable function $f$ defined on $\mathbb{R}^{d}$,
we denote the gradient of $f$ by $\nabla f$.

\begin{lemma}\label{lem: bound of sup of gradient of mathcal G 2nd}
Let $r\geq 1$ and $q\geq \max\{2,r\}$.
Then there exists a constant $C>0$ depending on $d,q,\mu_{\varphi},\mu_{\psi}, \gamma_{\varphi},\gamma_{\psi}$,
$\kappa_{\varphi}$ and $\kappa_{\psi}$
such that for any $r_{1}, r_{2}>0 $ satisfying $r_{1}^{1/\gamma_{\psi}}=r_{2}$, $(t,x)\in Q_{r_{1},r_{2}}$ and
$f\in C_{\rm c}^{\infty}(\mathbb{R}\times\mathbb{R}^{d};V)$ satisfying that $f(t,x)=0$ for $t\geq -8r_{1}$,
\begin{align}
\sup_{(s,y)\in Q_{r_{1},r_{2}}}|\nabla\mathcal{R}_{a,q}f(s,y)|^{q}
&\leq
Cr_{1}^{-\frac{q}{\gamma_{\psi}}}\mathbb{M}_{t}^{8r_{1}}
        \left(\int_{\mathbb{R}}\left(\mathbb{M}_{x}^{2r_{2}}
        \|f(t,x,\theta)\|_{H}^{q}\right)^{\frac{r}{q}}d\theta\right)^{\frac{q}{r}},
\label{eqn: sup of gradient of mathcal G}
\end{align}
\begin{align}
\sup_{(s,y)\in Q_{r_{1},r_{2}}}\left|\frac{\partial}{\partial s}\mathcal{R}_{a,q}f(s,y)\right|^{q}
\leq C r_{1}^{-q}\mathbb{M}_{t}^{8r_{1}}\left(\int_{\mathbb{R}}
    (\mathbb{M}_{x}^{2r_{2}}\|f(t,x,\theta)\|_{H}^{q})^{\frac{r}{q}}d\theta\right)^{\frac{q}{r}},
\label{eqn: sup of derivative wrt s of mathcal G}
\end{align}
for any $a\in\mathbb{R}$.
\end{lemma}

\begin{proof}
Let $(t,x)\in Q_{r_{1},r_{2}}$ and
$f\in C_{\rm c}^{\infty}(\mathbb{R}\times\mathbb{R}^{d};V)$
satisfying that $f(t,x)=0$ for $t\geq -8r_{1}$ be given.
For \eqref{eqn: sup of gradient of mathcal G},
by applying the arguments used in the proof of Lemma B.7 in \cite{Ji-Kim 2025-1}, we have
\begin{align}
\sup_{(s,y)\in Q_{r_{1},r_{2}}}|\nabla\mathcal{R}_{a,q}f(s,y)|^{q}
\leq Cr_{1}^{-\frac{q}{\gamma_{\psi}}}\mathbb{M}_{t}^{8r_{1}}\mathbb{M}_{x}^{2r_{2}}\|f(t,x,\cdot)\|_{V}^{q}
\label{eqn:Lemma A.7}
\end{align}
for some constant $C>0$.
On the other hand, by \eqref{eqn:ineq for maximal V valued} with $r_{0}=2r_{2}$, we have
\begin{align}
\mathbb{M}_{x}^{2r_{2}}\|f(t,x,\cdot)\|_{V}^{q}
\leq \left(\int_{\mathbb{R}}\left(\mathbb{M}_{x}^{2r_{2}}
\|f(t,x,\theta)\|_{H}^{q}\right)^{\frac{r}{q}}d\theta\right)^{\frac{q}{r}}.
\label{eqn:ineq for maximal V valued 2}
\end{align}
Therefore, by combining \eqref{eqn:Lemma A.7} and \eqref{eqn:ineq for maximal V valued 2},
we have
\begin{align*}
\sup_{(s,y)\in Q_{r_{1},r_{2}}}|\nabla\mathcal{R}_{a,q}f(s,y)|^{q}
\leq Cr_{1}^{-\frac{q}{\gamma_{\psi}}}\mathbb{M}_{t}^{8r_{1}}
        \left(\int_{\mathbb{R}}\left(\mathbb{M}_{x}^{2r_{2}}
        \|f(r,x,\theta)\|_{H}^{q}\right)^{\frac{r}{q}}d\theta\right)^{\frac{q}{r}},
\end{align*}
which gives the inequality given in \eqref{eqn: sup of gradient of mathcal G}.

We now prove \eqref{eqn: sup of derivative wrt s of mathcal G}.
By applying the arguments used in the proof of Lemmas B.8, B.9 and B.10 in \cite{Ji-Kim 2025-1}, we have
\begin{align}
\sup_{(s,y)\in Q_{r_{1},r_{2}}}\left|\frac{\partial}{\partial s}\mathcal{R}_{a,q}f(s,y)\right|^{q}
\leq Cr_{1}^{-q}\mathbb{M}_{t}^{8r_{1}}\mathbb{M}_{x}^{2r_{2}}\|f(t,x,\cdot)\|_{V}^{q}
\label{eqn:Lemma A.8}
\end{align}
for some constant $C>0$.
Therefore, by combining \eqref{eqn:Lemma A.8} and \eqref{eqn:ineq for maximal V valued 2},
we have
\begin{align*}
\sup_{(s,y)\in Q_{r_{1},r_{2}}}\left|\frac{\partial}{\partial s}\mathcal{R}_{a,q}f(s,y)\right|^{q}
\leq Cr_{1}^{-q}\mathbb{M}_{t}^{8r_{1}}
        \left(\int_{\mathbb{R}}\left(\mathbb{M}_{x}^{2r_{2}}
        \|f(r,x,\theta)\|_{H}^{q}\right)^{\frac{r}{q}}d\theta\right)^{\frac{q}{r}},
\end{align*}
which gives the inequality given in \eqref{eqn: sup of derivative wrt s of mathcal G}.
\end{proof}

For $R>0$, we define
\begin{align*}
Q_{R}=(-2R,0)\times B_{R^{1/\gamma_{\psi}}}.
\end{align*}

\begin{lemma}\label{lem: bound of double average 2nd}
Let $r\geq 1$, $q\geq \max\{2,r\}$ and $R>0$. Then it holds that
\begin{itemize}
  \item [\rm (i)] if $-\infty< a\leq -2R<0\leq b< \infty$, then there exists a constant $C_{1}>0$ depending on $a,b ,q,\mu_{\varphi},\mu_{\psi},\kappa_{\varphi},\kappa_{\psi}$, $\gamma_{\varphi}$ and $\gamma_{\psi}$
      such that for any $(t,x)\in Q_{R}$ and $f\in C_{\rm c}^{\infty}(\mathbb{R}\times\mathbb{R}^{d};V)$,
\begin{align}
&\frac{1}{|Q_{R}|^{2}}\int_{Q_{R}}\int_{Q_{R}}
            |\mathcal{R}_{a,q}f(s,y)-\mathcal{R}_{a,q}f(w,z)|^{q}dsdydw dz\nonumber\\
&\qquad\leq C_{1}\mathbb{M}_{t}\left(\int_{\mathbb{R}}
        (\mathbb{M}_{x}\|f(t,x,\theta)\|_{H}^{q})^{\frac{r}{q}}
        d\theta\right)^{{\frac{q}{r}}},\label{eqn: double aver 1}
\end{align}

  \item [\rm (ii)] if $r\leq 2$ and $q=2$,
  then there exists a constant $C_{2}>0$ depending on $\mu_{\varphi},\mu_{\psi},\kappa_{\varphi},\kappa_{\psi}$,
  $\gamma_{\varphi}$ and $\gamma_{\psi}$
such that for any $(t,x)\in Q_{R}$ and $f\in C_{\rm c}^{\infty}(\mathbb{R}\times\mathbb{R}^{d};V)$,
\begin{align}
&\frac{1}{|Q_{R}|^{2}}\int_{Q_{R}}\int_{Q_{R}}
    |\mathcal{R}_{-\infty,2}f(s,y)-\mathcal{R}_{-\infty,2}f(w,z)|^{2}dsdydw dz\nonumber\\
&\qquad\leq C_{2}\mathbb{M}_{t}\left(\int_{\mathbb{R}}
        (\mathbb{M}_{x}\|f(t,x,\theta)\|_{H}^{q})^{\frac{r}{q}}
        d\theta\right)^{{\frac{q}{r}}}.\label{eqn: double aver 3}
\end{align}
\end{itemize}
\end{lemma}

\begin{proof}
Since the proofs of \eqref{eqn: double aver 1} and \eqref{eqn: double aver 3}
are similar, we only prove \eqref{eqn: double aver 1}.
Let $(t,x)\in Q_{R}$ and $f\in C_{\rm c}^{\infty}(\mathbb{R}\times\mathbb{R}^{d};V)$ be given.
We take a function $\zeta\in C_{\rm c}^{\infty}(\mathbb{R})$
such that $0\leq \zeta\leq 1$ and
\begin{align*}
\zeta=\left\{
        \begin{array}{ll}
          1 & \hbox{on } \quad [-8R,8R],\\
          0 & \hbox{on } \quad [-10R,10R]^{\rm c},
        \end{array}
      \right.
\end{align*}
where $A^{\rm c}$ is the complement of a set $A$.
Put
\begin{align*}
f_{1}(s,y)=f(s,y)\zeta(s),\quad f_{2}(s,y)=f(s,y)(1-\zeta(s,y)).
\end{align*}
For $f_{1}$, we put
\begin{align*}
f_{11}(s,y)=\eta(y) f_{1}(s,y),\quad f_{12}=(1-\eta(y))f_{1}(s,y).
\end{align*}
By applying the arguments used in the proof of Lemma B.11 in \cite{Ji-Kim 2025-1}, we have
\begin{align}
&\frac{1}{|Q_{R}|^{2}}\int_{Q_{R}}\int_{Q_{R}}
       \left|\mathcal{R}_{a,q}f(s,y)-\mathcal{R}_{a,q}f(w,z)\right|^{q}dsdydw dz \nonumber\\
&\qquad \leq 2^{2q-1}\frac{1}{|Q_{R}|}\int_{Q_{R}}|\mathcal{R}_{a,q}f_{1}(s,y)|^{q}dsdy \nonumber\\
&\qquad\qquad     +4^{2q-1}\sup_{(s,y)\in Q_{R}}
        \left(|R^{\frac{1}{\gamma_{\psi}}}\nabla \mathcal{R}_{a,q}f_{2}(s,y)|^{q}
        +\left|R\frac{\partial }{\partial s}\mathcal{R}_{a,q}f_{2}(s,y)\right|^{q}\right).
  \label{eqn:bound for LL}
\end{align}
Firstly, we show that
\begin{equation}\label{eqn: bound of mathcalG A1 on QR 2nd}
\int_{Q_{R}}|\mathcal{R}_{a,q}f_{1}(l, s,y)|^{q}dsdy
\leq C_{0}|Q_{R}|\mathbb{M}_{t}\left(\int_{\mathbb{R}}
        (\mathbb{M}_{x}\|f(t,x,\theta)\|_{H}^{q})^{\frac{r}{q}}
        d\theta\right)^{^{\frac{q}{r}}}
\end{equation}
for some constant $C_{0}>0$.
We take a function $\eta\in C_{\rm c}^{\infty}(\mathbb{R}^{d})$ such that $0\leq \eta\leq 1$ and
\begin{align*}
\quad\eta=\left\{
       \begin{array}{ll}
1\quad&\text{on}\quad B_{2R^{1/\gamma_{\psi}}},\\
0\quad&\text{on}\quad B_{3R^{1/\gamma_{\psi}}}^{\rm c},
       \end{array}
     \right.
\end{align*}
and we put $f_{11}(s,y)=\eta(y) f_{1}(s,y)$ and $f_{12}(s,y)=(1-\eta(y))f_{1}(s,y)$.
Then it holds that $\mathcal{R}_{a,q}f_{1}\leq \mathcal{R}_{a,q}f_{11}+\mathcal{R}_{a,q}f_{12}$
and
\begin{align}
\int_{Q_{R}}|\mathcal{R}_{a,q}f_{1}(s,y)|^{q}dsdy
\leq 2^{q-1}(I_{1}+I_{2}),\label{eqn:decomp f1 integral}
\end{align}
where
\begin{align*}
I_{1}:=\int_{Q_{R}}|\mathcal{R}_{a,q}f_{11}(s,y)|^{q}dsdy,\quad
I_{2}:=\int_{Q_{R}}|\mathcal{R}_{a,q}f_{12}(s,y)|^{q}dsdy.
\end{align*}
For $I_{1}$, by applying Lemma \ref{lem: g equal 0 outside of B3r1 2nd}, we obtain that
\begin{align}
\int_{Q_{R}}|\mathcal{R}_{a,q}f_{11}(s,y)|^{q}dsdy
&=\int_{-2R}^{0}\int_{B_{R^{1/\gamma_{\psi}}}}|\mathcal{R}_{a,q}f_{11}(s,y)|^{q}dsdy\nonumber\\
&\leq C^{(1)}\int_{a}^{0}\left[\int_{\mathbb{R}}\left(\int_{B_{3R^{1/\gamma_{\psi}}}}
        \|f_{11}(s,y,\theta)\|_{H}^{q}dy\right)^{\frac{r}{q}}
        d\theta\right]^{\frac{q}{r}}ds\nonumber\\
&\leq C^{(1)}\int_{-10R}^{0}\left[\int_{\mathbb{R}}\left(\int_{B_{3R^{1/\gamma_{\psi}}}}
        \|f_{11}(s,y,\theta)\|_{H}^{q}dy\right)^{\frac{r}{q}}
        d\theta\right]^{\frac{q}{r}}ds\nonumber\\
&\leq C^{(1)'}R^{\frac{d}{\gamma_{\psi}}}\int_{-10R}^{0}\left(\int_{\mathbb{R}}(\mathbb{M}_{x}
        \|f_{11}(s,x,\theta)\|_{H}^{q})^{\frac{r}{q}}
        d\theta\right)^{{\frac{q}{r}}}ds\nonumber\\
&\leq C^{(1)'}R^{\frac{d}{\gamma_{\psi}}+1}\mathbb{M}_{t}\left(\int_{\mathbb{R}}(\mathbb{M}_{x}
        \|f_{11}(t,x,\theta)\|_{H}^{q})^{\frac{r}{q}}
        d\theta\right)^{{\frac{q}{r}}}\nonumber\\
&\leq C^{(1)'}R^{\frac{d}{\gamma_{\psi}}+1}\mathbb{M}_{t}\left(\int_{\mathbb{R}}
        (\mathbb{M}_{x}\|f(t,x,\theta)\|_{H}^{q})^{\frac{r}{q}}d\theta\right)^{^{\frac{q}{r}}}
        \label{eqn:ineq f11}
\end{align}
for some constants $C^{(1)},C^{(1)'}>0$.
For $I_{2}$, by Lemma \ref{lem: bound of integral of mathcal G over Qr2r1 2nd}
with $(r_{1},r_{2})=(R^{1/\gamma_{\psi}},R)$, we have
\begin{align}
\int_{Q_{R}}|\mathcal{R}_{a,q}f_{12}(s,y)|^{q}dsdy
\leq C^{(2)}R^{\frac{d}{\gamma_{\psi}}+1}\mathbb{M}_{t}\left(\int_{\mathbb{R}}
        (\mathbb{M}_{x}\|f(t,x,\theta)\|_{H}^{q})^{\frac{r}{q}}d\theta\right)^{^{\frac{q}{r}}}
        \label{eqn:ineq f12}
\end{align}
for some constant $C^{(2)}>0$. Therefore, by combining  \eqref{eqn:decomp f1 integral},
\eqref{eqn:ineq f11} and \eqref{eqn:ineq f12}, we have \eqref{eqn: bound of mathcalG A1 on QR 2nd}.

On the other hand, by applying \eqref{eqn: sup of gradient of mathcal G}
with $(r_{1},r_{2})=(R^{1/\gamma_{\psi}},R)$, we have
\begin{equation}\label{eqn: bound of sup of R gamma gradient on QR 2nd}
\sup_{(s,y)\in Q_{R}}|R^{\frac{1}{\gamma_{\psi}}}\nabla \mathcal{R}_{a,q}f_{2}(s,y)|^{q}
\leq C^{(3)}\mathbb{M}_{t}\left(\int_{\mathbb{R}}
        (\mathbb{M}_{x}\|f(t,x,\theta)\|_{H}^{q})^{\frac{r}{q}}d\theta\right)^{\frac{q}{r}}
\end{equation}
for some constant $C^{(3)}>0$.
Also, by applying \eqref{eqn: sup of derivative wrt s of mathcal G}
with $(r_{1},r_{2})=(R^{1/\gamma_{\psi}},R)$, we have
\begin{equation}\label{eqn: bound of sup of R Dt G on QR 2nd}
\sup_{(s,y)\in Q_{R}}\left|R\frac{\partial}{\partial s} \mathcal{R}_{a,q}f_{2}(s,y)\right|^{q}
\leq C^{(4)}\mathbb{M}_{t}\left(\int_{\mathbb{R}}
        (\mathbb{M}_{x}\|f(t,x,\theta)\|_{H}^{q})^{\frac{r}{q}}d\theta\right)^{\frac{q}{r}}
\end{equation}
for some constant $C^{(4)}>0$. Hence, by combining
\eqref{eqn:bound for LL}, \eqref{eqn: bound of mathcalG A1 on QR 2nd},
\eqref{eqn: bound of sup of R gamma gradient on QR 2nd} and \eqref{eqn: bound of sup of R Dt G on QR 2nd},
we have the inequality given in \eqref{eqn: double aver 1}.
\end{proof}

If $f$ is a locally integrable function on $\mathbb{R}^{1+d}$, we define
\begin{align}\label{eqn: sharp fct}
f^{\#}(t,x):=\sup_{Q}\frac{1}{|Q|}\int_{Q}|f(r,z)-f_{Q}|drdz,
\end{align}
where the sup is taken all $Q$ containing $(t,x)$ of the type
\begin{align*}
Q=(t-R,t+R)\times B_{R^{1/\gamma_{\psi}}}(x),\quad R>0
\end{align*}
and $f_{Q}=\frac{1}{|Q|}\int_{Q}f(r,z)drdz$.

\textbf{A proof of Theorem \ref{cor: 2nd LP ineq}. }
Since the proofs of \eqref{ineq: 2nd LP ineq} and \eqref{ineq: 2nd LP ineq q=2} are similar,
we only prove \eqref{ineq: 2nd LP ineq}.
Let $-\infty<a<b<\infty$.
Let $r\geq 1$, $q\geq \max\{2,r\}$ and let $p\geq q$.
If $p=q$, then Theorem \ref{cor: 2nd LP ineq} holds
by Lemma \ref{lem: 2nd LP ineq for p equal lambda}.
Suppose that $p>q$.
Note that by applying Jensen's inequality and the definition of $(\mathcal{R}_{a,q}f)_{Q}$,
we obtain that
\begin{align}
&\left(\frac{1}{|Q|}\int_{Q}|\mathcal{R}_{a,q}f(s,y)-(\mathcal{R}_{a,q}f)_{Q}|dsdy\right)^{q}\nonumber\\
&\qquad\leq \frac{1}{|Q|}\int_{Q}|\mathcal{R}_{a,q}f(s,y)-(\mathcal{R}_{a,q}f)_{Q}|^{q}dsdy\nonumber\\
&\qquad\leq \frac{1}{|Q|^{2}}\int_{Q}\int_{Q}
            |\mathcal{R}_{a,q}f(s,y)-\mathcal{R}_{a,q}f(w,z)|^{q}dsdydw dz.\label{eqn: bound of sharp mathcalG 2}
\end{align}
For any $c_{1}\in\mathbb{R}$, $c_{2}\in\mathbb{R}^{d}$
and for $K_{\varphi,\psi}(t,s,x)=L_{\varphi}p_{\psi}(t,s,x)$,
we put
\begin{align*}
\bar{f}(t,x,\cdot)=f(t-c_{1},x-c_{2},\cdot),\quad
\bar{K}_{\varphi,\psi}(t,s,x)=K_{\varphi,\psi}(t-c_{1},s-c_{1},x).
\end{align*}
Then by applying the change of variables, we obtain that
\begin{align*}
&\mathcal{R}_{a,q}f(t-c_{1},x-c_{2})\\
&=\left(\int_{a}^{t-c_{1}}((t-c_{1})-s)^{\frac{q\gamma_{\varphi}}{\gamma_{\psi}}-1}
        \|K_{\varphi,\psi}(t-c_{1},s,\cdot)*f(s,\cdot,\cdot)(x-c_{2})\|_{V}^{q}ds\right)^{\frac{1}{q}}\\
&=\left(\int_{a+c_{1}}^{t}(t-s)^{\frac{q\gamma_{\varphi}}{\gamma_{\psi}}-1}
        \|\bar{K}_{\varphi,\psi}(t,s,\cdot)*\bar{f}(s,\cdot,\cdot)(x)\|_{V}^{q}ds\right)^{\frac{1}{q}}\\
&=\mathcal{R}_{a+c_{1},q}\bar{f}(t,x).
\end{align*}
Therefore, we may assume that $Q=Q_{R}=(-2R,0)\times B_{R^{1/\gamma_{\psi}}}$.
By applying \eqref{eqn: double aver 1}, we have
\begin{align}\label{eqn:esti double average 2}
&\frac{1}{|Q_{R}|^{2}}\int_{Q_{R}}\int_{Q_{R}}
                |\mathcal{R}_{a,q}f(s,y)-\mathcal{R}_{a,q}f(w,z)|^{q}dsdydw dz\nonumber\\
&\qquad\leq C_{1}\mathbb{M}_{t}\left(\int_{\mathbb{R}}
        (\mathbb{M}_{x}\|f(t,x,\theta)\|_{H}^{q})^{\frac{r}{q}}
        d\theta\right)^{{\frac{q}{r}}}
\end{align}
for some constant $C_{1}>0$.
By combining \eqref{eqn: bound of sharp mathcalG 2} and \eqref{eqn:esti double average 2}, we have
\begin{align*}
\frac{1}{|Q|}\int_{Q}|\mathcal{R}_{a,q}f(s,y)-(\mathcal{R}_{a,q}f)_{Q}|dsdy
\leq  C_{1}^{\frac{1}{q}}\left(\mathbb{M}_{t}\left(\int_{\mathbb{R}}
        (\mathbb{M}_{x}\|f(t,x,\theta)\|_{H}^{q})^{\frac{r}{q}}
        d\theta\right)^{{\frac{q}{r}}}\right)^{\frac{1}{q}},
\end{align*}
and then by taking the supremum in $Q$, we have
\begin{equation}\label{eqn: bound of sharp mathcalG 2nd}
(\mathcal{R}_{a,q}f)^{\#}(t,x)
\leq C_{1}^{\frac{1}{q}}\left(\mathbb{M}_{t}\left(\int_{\mathbb{R}}
    (\mathbb{M}_{x}\|f(t,x,\theta)\|_{H}^{q})^{\frac{r}{q}}
    d\theta\right)^{^{\frac{q}{r}}}\right)^{\frac{1}{q}}.
\end{equation}
Then by applying Theorem 5.1 in \cite{Ji-Kim 2025-1}, \eqref{eqn: bound of sharp mathcalG 2nd},
Hardy-Littlewood maximal theorem, we obtain that
 \begin{align}
\|\mathcal{R}_{a,q}f\|_{L^{p}((a,b)\times\mathbb{R}^{d})}^{p}
&\leq A_{1}\|(\mathcal{R}_{a,q}f)^{\#}\|_{L^{p}((a,b)\times\mathbb{R}^{d})}^{p}\nonumber\\
&= A_{1}\int_{\mathbb{R}^{d}}\int_{a}^{b}|(\mathcal{R}_{a,q}f)^{\#}(t,x)|^{p}dtdx\nonumber\\
&\leq A_{1}C_{1}^{\frac{p}{q}}\int_{\mathbb{R}^{d}}\int_{a}^{b}\left(\mathbb{M}_{t}\left(\int_{\mathbb{R}}
        (\mathbb{M}_{x}\|f(t,x,\theta)\|_{H}^{q})^{\frac{r}{q}}
        d\theta\right)^{\frac{q}{r}}\right)^{\frac{p}{q}}dtdx\nonumber\\
&\leq A_{1}C_{1}^{\frac{p}{q}}A_{2}\int_{\mathbb{R}^{d}}\int_{a}^{b}\left(\int_{\mathbb{R}}
        (\mathbb{M}_{x}\|f(t,x,\theta)\|_{H}^{q})^{\frac{r}{q}}
        d\theta\right)^{{\frac{p}{r}}}dtdx.\label{eqn: pf thm 5.2 1}
\end{align}
Note that by the assumption $p>q$ and $q\geq \max\{2,r\}$, we have $p>r$
and so by applying the Minkowski's inequality, we have
\begin{align}
&\int_{\mathbb{R}^{d}}\int_{a}^{b}\left(\int_{\mathbb{R}}
        (\mathbb{M}_{x}\|f(t,x,\theta)\|_{H}^{q})^{\frac{r}{q}}
        d\theta\right)^{{\frac{p}{r}}}dtdx\nonumber\\
&\qquad\leq \int_{a}^{b}\left[\int_{\mathbb{R}}\left(\int_{\mathbb{R}^{d}}
        (\mathbb{M}_{x}\|f(t,x,\theta)\|_{H}^{q})^{\frac{p}{q}}dx\right)^{\frac{r}{p}}
        d\theta\right]^{\frac{p}{r}}dt,\label{eqn:Minkow max}
\end{align}
from which, by applying \eqref{eqn: pf thm 5.2 1} and Hardy-Littlewood maximal theorem, we obtain that
 \begin{align*}
\|\mathcal{R}_{a,q}f\|_{L^{p}((a,b)\times\mathbb{R}^{d})}^{p}
&\leq A_{1}C_{1}^{\frac{p}{q}}A_{2}\int_{a}^{b}\left[\int_{\mathbb{R}}\left(\int_{\mathbb{R}^{d}}
        (\mathbb{M}_{x}\|f(t,x,\theta)\|_{H}^{q})^{\frac{p}{q}}dx\right)^{\frac{r}{p}}
        d\theta\right]^{\frac{p}{r}}dt\\
&\leq A_{1}C_{1}^{\frac{p}{q}}A_{2}A_{3}\int_{a}^{b}\left[\int_{\mathbb{R}}\left(\int_{\mathbb{R}^{d}}
        \|f(t,x,\theta)\|_{H}^{p}dx\right)^{\frac{r}{p}}
        d\theta\right]^{^{\frac{p}{r}}}dt,
\end{align*}
which proves the desired result. \hfill $\blacksquare$

\subsection{A Proof of Lemma \ref{lem:L psi potential}}\label{sec: potential}
Let $K$ be a separable Hilbert space.
Let $\varphi\in\mathfrak{M}_{\gamma}(\mathbb{R}^{d})$ for some $\gamma\equiv\gamma_{\varphi}>0$
satisfying the condition \textbf{(S2)} and let $\psi\in\mathfrak{S}$.
For $f\in C_{\rm c}^{\infty}(\mathbb{R}\times \mathbb{R}^{d};K)$, for a notational convenience, we put
\begin{align}\label{eqn: mathcal G def}
\mathcal{G}f(t,x)&=\int_{-\infty}^{t}L_{\varphi}\mathcal{T}_{\psi}(t,s)f(s,x)ds,\\
M(t,x,s,y)&=1_{(0,\infty)}(t-s)L_{\varphi}p_{\psi}(t,s,x-y).
 \label{eqn: kernel M}
\end{align}
Then for any $f\in C_{\rm c}^{\infty}(\mathbb{R}\times \mathbb{R}^{d};K)$, we obtain that
\begin{align*}
\mathcal{G}f(t,x)
&=\int_{-\infty}^{t}L_{\varphi}\mathcal{T}_{\psi}(t,s)f(s,x)ds\\
&=\int_{-\infty}^{\infty}\int_{\mathbb{R}^{d}}1_{(0,\infty)}(t-s)L_{\varphi}p_{\psi}(t,s,x-y)f(s,y)dyds\\
&=\int_{-\infty}^{\infty}\int_{\mathbb{R}^{d}}M(t,x,s,y)f(s,y)dyds.
\end{align*}
For a complex-valued measurable function $J(t,x,s,y)$ on $\mathbb{R}^{2d+2}$,
the authors in \cite{I. Kim S. Lim K.-H. Kim 2016} considered the following conditions:
there exist a constant $\gamma>0$ and a nonnegative nondecreasing functions $\varphi_{i}$ ($i=1,2,3$) on $(0,\infty)$ such that
\begin{itemize}
  \item [\textbf{(B1)}] for any $a>s$ and $y,z\in\mathbb{R}^{d}$,
  \begin{align*}
  \int_{a}^{\infty}\int_{\mathbb{R}^{d}}|J(t,x,s,y)-J(t,x,s,z)|dxdt
  \leq \varphi_{1}\left(\frac{|y-z|}{(a-s)^{1/\gamma}}\right),
  \end{align*}
  \item [\textbf{(B2)}] for any $a,b,s,r\in\mathbb{R}$ with $a>b\geq \max\{s,r\}$ and $y\in\mathbb{R}^{d}$,
    \begin{align*}
  \int_{a}^{\infty}\int_{\mathbb{R}^{d}}|J(t,x,s,y)-J(t,x,r,y)|dxdt
  \leq \varphi_{2}\left(\frac{|s-r|}{a-b}\right),
  \end{align*}
  \item [\textbf{(B3)}] for any $b>s$, $\rho>0$ and $y\in\mathbb{R}^{d}$,
    \begin{align*}
  \int_{s}^{b}\int_{|x-y|\geq \rho}|J(t,x,s,y)|dxdt\leq \varphi_{3}\left(\frac{(b-s)^{1/\gamma}}{\rho}\right).
  \end{align*}
\end{itemize}
\begin{lemma}\label{lem: esti of grad K}
Let $\varphi\in\mathfrak{M}_{\gamma}(\mathbb{R}^{d})$ for some $\gamma\equiv\gamma_{\varphi}>0$
satisfying the condition \textbf{(S2)} and let $\psi\in\mathfrak{S}$.
Let $N=\min\{N_{\varphi},N_{\psi}\}$.
Assume that $N> d+2+\lfloor\gamma_{\varphi}\rfloor+\lfloor\gamma_{\psi}\rfloor$.
 Then there exist constants $C_{1},C_{2},C_{3}>0$ depending on $d, \mu_{\varphi},\mu_{\psi},\gamma_{\varphi},\gamma_{\psi}$, $\kappa_{\varphi}$ and $\kappa_{\psi}$ such that for any $t>s$ and $x\in\mathbb{R}^{d}\setminus \{0\}$,
\begin{align}
| L_{\varphi}p_{\psi}(t,s,x)|
&\leq C_{1}\left(|x|^{-(\gamma_{\varphi}+d)}\wedge (t-s)^{-\frac{\gamma_{\varphi}+d}{\gamma_{\psi}}}\right) \label{eqn: bound of K},\\
|\nabla L_{\varphi}p_{\psi}(t,s,x)|
&\leq C_{1}\left(|x|^{-(\gamma_{\varphi}+1+d)}\wedge (t-s)^{-\frac{\gamma_{\varphi}+1+d}{\gamma_{\psi}}}\right) \label{eqn: bound of nabla K},\\
\left|\partial_{s}L_{\varphi}p_{\psi}(t,s,x)\right|
&\leq C_{3}\left(|x|^{-(\gamma_{\varphi}+\gamma_{\psi}+d)}\wedge
                      (t-s)^{-\frac{\gamma_{\varphi}+\gamma_{\psi}+d}{\gamma_{\psi}}}\right),\label{eqn: bound of dt nabla K}
\end{align}
where the absolute value of the left hand side means the Euclidean norm
and $a\wedge b=\min\{a,b\}$ for $a,b\in\mathbb{R}^{d}$.
\end{lemma}

\begin{proof}
The inequality \eqref{eqn: bound of nabla K} has been proved in Lemma 3.3 of \cite{Ji-Kim 2025}.
The inequalities \eqref{eqn: bound of K} and \eqref{eqn: bound of dt nabla K}
can be proved by the similar arguments used in Lemma A.2 and Lemma A.3 of \cite{Ji-Kim 2025-1}.
\end{proof}

\begin{lemma}\label{lem:A1-A3}
Let $\varphi\in\mathfrak{M}_{\gamma}(\mathbb{R}^{d})$ for some $\gamma\equiv\gamma_{\varphi}>0$
satisfying the condition \textbf{(S2)} and let $\psi\in\mathfrak{S}$.
Assume that $\gamma_{\varphi}=\gamma_{\psi}$ and $N_{\varphi},N_{\psi}>d+2+2\lfloor\gamma\rfloor$.
Let $M$ be the scalar valued function given in \eqref{eqn: kernel M}.
Then for $J=M$, the conditions \textbf{(B1)}-\textbf{(B3)} hold.
\end{lemma}

\begin{proof}
Firstly we check \textbf{(B1)}. By applying the mean value theorem, the change of variables and \eqref{eqn: bound of nabla K},
we obtain that
\begin{align}
&\int_{\mathbb{R}^{d}}|M(t,x,s,y)-M(t,x,s,z)|dx\nonumber\\
&=\int_{\mathbb{R}^{d}}|L_{\varphi}p_{\psi}(t,s,x-y)-L_{\varphi}p_{\psi}(t,s,x-z)|dx\nonumber\\
&=|z-y|\int_{\mathbb{R}^{d}}|\nabla L_{\varphi}p_{\psi}(t,s,x-(\theta y+(1-\theta)z)|dx\nonumber\\
&=|z-y|\int_{\mathbb{R}^{d}}|\nabla L_{\varphi}p_{\psi}(t,s,x)|dx\nonumber\\
&= |z-y|(t-s)^{\frac{d}{\gamma_{\psi}}}\int_{\mathbb{R}^{d}}|\nabla L_{\varphi}p_{\psi}(t,s,(t-s)^{\frac{1}{\gamma_{\psi}}}x)|dx\nonumber\\
&\leq C_{1} |z-y|(t-s)^{\frac{d}{\gamma_{\psi}}}\int_{\mathbb{R}^{d}}
       \left(|(t-s)^{\frac{1}{\gamma_{\psi}}}x|^{-(\gamma_{\varphi}+1+d)}\wedge (t-s)^{-\frac{\gamma_{\varphi}+1+d}{\gamma_{\psi}}}\right)dx\nonumber\\
&= C_{1} |z-y|(t-s)^{-\frac{\gamma_{\varphi}+1}{\gamma_{\psi}}}\int_{\mathbb{R}^{d}}
       \left(|x|^{-(\gamma_{\varphi}+1+d)}\wedge 1\right)dx.\label{eqn: A1 1}
\end{align}
Note that for any $\alpha\geq 0$ and $\beta>0$ with $\alpha<\beta$, it holds that
\begin{align}\label{eqn: int alpha beta}
\int_{\mathbb{R}^{d}}|x|^{\alpha}\left(|x|^{-(\beta+d)}\wedge 1\right)dx<\infty.
\end{align}
Indeed, by applying the polar coordinates, we obtain that
\begin{align*}
\int_{\mathbb{R}^{d}}|x|^{\alpha}\left(|x|^{-(\beta+d)}\wedge 1\right)dx
&=\int_{|x|\leq 1}|x|^{\alpha}dx+\int_{|x|>1}|x|^{\alpha-(\beta+d)}dx\\
&=\frac{2\pi^{d/2}}{\Gamma\left(\frac{d}{2}\right)}
    \left(\int_{0}^{1}\rho^{\alpha+d-1}d\rho+\int_{1}^{\infty}\rho^{\alpha-\beta-1}d\rho\right)\\
&=\frac{2\pi^{d/2}}{\Gamma\left(\frac{d}{2}\right)}
        \left(\frac{1}{\alpha+d}+\frac{1}{\beta-\alpha}\right).
\end{align*}
Then by applying \eqref{eqn: A1 1} and \eqref{eqn: int alpha beta} with $\alpha=0$ and $\beta=\gamma_{\varphi}+1$, we have
\begin{align*}
\int_{\mathbb{R}^{d}}|M(t,x,s,y)-M(t,x,s,z)|dx\leq C_{2}|z-y|(t-s)^{-\frac{\gamma_{\varphi}+1}{\gamma_{\psi}}},
\end{align*}
from which, by the assumption $\gamma_{\psi}=\gamma_{\varphi}$, we have
\begin{align*}
\int_{a}^{\infty}\int_{\mathbb{R}^{d}}|M(t,x,s,y)-M(t,x,s,z)|dxdt\leq C_{3}|z-y|(a-s)^{-\frac{1}{\gamma_{\psi}}}.
\end{align*}
By taking $\varphi_{1}(t)=C_{3}t$, we have the desired result.

Secondly we check \textbf{(B2)}. For $s\leq \max\{r,s\}\leq b<a<t$,
putting $\tau=\theta s+(1-\theta)r$, $\theta\in (0,1)$, by applying mean value theorem, the change of variables,
\eqref{eqn: bound of dt nabla K} and \eqref{eqn: int alpha beta}
with $\alpha=0$ and $\beta=\gamma_{\varphi}+\gamma_{\psi}$,
we obtain that
\begin{align*}
&\int_{\mathbb{R}^{d}}|M(t,x,s,y)-M(t,x,r,y)|dx\\
&= |s-r|\int_{\mathbb{R}^{d}}|\partial_{s}M(t,x,\tau,y)|dx\\
&=|s-r|(t-\tau)^{\frac{d}{\gamma_{\psi}}}\int_{\mathbb{R}^{d}}|\partial_{s} L_{\varphi}p_{\psi}(t,\tau,(t-\tau)^{\frac{1}{\gamma_{\psi}}}(x-y))|dx\\
&=|s-r|(t-\tau)^{\frac{d}{\gamma_{\psi}}}\int_{\mathbb{R}^{d}}|\partial_{s} L_{\varphi}p_{\psi}(t,\tau,(t-\tau)^{\frac{1}{\gamma_{\psi}}}x)|dx\\
&\leq C_{4}|s-r|(t-\tau)^{\frac{d}{\gamma_{\psi}}}\int_{\mathbb{R}^{d}}
    \left(|(t-\tau)^{\frac{1}{\gamma_{\psi}}}x|^{-(\gamma_{\varphi}+\gamma_{\psi}+d)}\wedge (t-\tau)^{-\frac{\gamma_{\varphi}+\gamma_{\psi}+d}{\gamma_{\psi}}}\right)dx\\
&=C_{4}|s-r|(t-\tau)^{-\frac{\gamma_{\varphi}+\gamma_{\psi}}{\gamma_{\psi}}}\int_{\mathbb{R}^{d}}
    \left(|x|^{-(\gamma_{\varphi}+\gamma_{\psi}+d)}\wedge 1\right)dx\\
&\leq  C_{5}|s-r|(t-b)^{-\frac{\gamma_{\varphi}}{\gamma_{\psi}}-1}.
\end{align*}
Therefore, we obtain that
\begin{align*}
\int_{a}^{\infty}\int_{\mathbb{R}^{d}}|M(t,x,s,y)-M(t,x,r,y)|dxdt
&\leq C_{5} |s-r|\int_{a}^{\infty}(t-b)^{-\frac{\gamma_{\varphi}}{\gamma_{\psi}}-1}dt\\
&=C_{6} |s-r|(a-b)^{-\frac{\gamma_{\varphi}}{\gamma_{\psi}}}.
\end{align*}
By the assumption $\gamma_{\varphi}=\gamma_{\psi}$ and taking $\varphi_{2}(t)=C_{6} t$, we have the desired result.

Finally we check \textbf{(B3)}. Take $\sigma>0$ such that $\sigma<\gamma_{\varphi}$.
Then for any $x,y\in\mathbb{R}^{d}$ satisfying $|x-y|\geq \rho$, we have $\frac{|x-y|}{\rho}\geq 1$ and we have
\begin{align}
\int_{s}^{b}\int_{|x-y|\geq \rho}|M(t,x,s,y)|dxdt
\leq \int_{s}^{b}\int_{|x-y|\geq \rho}\frac{|x-y|^{\sigma}}{\rho^{\sigma}}|M(t,x,s,y)|dxdt.\label{eqn: A3 1}
\end{align}
Note that by applying the change of variables, \eqref{eqn: bound of K} and
\eqref{eqn: int alpha beta} with $\alpha=\sigma$ and $\beta=\gamma_{\varphi}$,
we obtain that
\begin{align}
&\int_{|x-y|\geq \rho}|x-y|^{\sigma}|M(t,x,s,y)|dx\nonumber\\
&\leq \int_{\mathbb{R}^{d}}|x-y|^{\sigma}|L_{\varphi}p_{\psi}(t,s,x-y)|dx\nonumber\\
&=\int_{\mathbb{R}^{d}}|x|^{\sigma}|L_{\varphi}p_{\psi}(t,s,x)|dx\nonumber\\
&=(t-s)^{\frac{d+\sigma}{\gamma_{\psi}}}\int_{\mathbb{R}^{d}}|x|^{\sigma}
|L_{\varphi}p_{\psi}(t,s,(t-s)^{\frac{1}{\gamma_{\psi}}}x)|dx\nonumber\\
&\leq C_{7}(t-s)^{\frac{d+\sigma}{\gamma_{\psi}}}\int_{\mathbb{R}^{d}}|x|^{\sigma}
        \left(|(t-s)^{\frac{1}{\gamma_{\psi}}}x|^{-(\gamma_{\varphi}+d)}\wedge (t-s)^{\frac{\gamma_{\varphi}+d}{\gamma_{\psi}}}\right)dx\nonumber\\
&=C_{7}(t-s)^{\frac{\sigma-\gamma_{\varphi}}{\gamma_{\psi}}}\int_{\mathbb{R}^{d}}|x|^{\sigma}
        \left(|x|^{-(\gamma_{\varphi}+d)}\wedge 1 \right)dx\nonumber\\
&=C_{8}(t-s)^{\frac{\sigma-\gamma_{\varphi}}{\gamma_{\psi}}}.\label{eqn: A3 2}
\end{align}
Therefore, by combining \eqref{eqn: A3 1}, \eqref{eqn: A3 2} and the assumption $\gamma_{\varphi}=\gamma_{\psi}$,
we have
\begin{align*}
\int_{s}^{b}\int_{|x-y|\geq \rho}|M(t,x,s,y)|dxdt
\leq C_{8}\int_{s}^{b}(t-s)^{\frac{\sigma}{\gamma_{\psi}}-1}\rho^{-\sigma}dt
=\frac{C_{8}\gamma_{\psi}}{\sigma}
        (b-s)^{\frac{\sigma}{\gamma_{\psi}}}\rho^{-\sigma}.
\end{align*}
By taking $\varphi_{3}(t)=\frac{C_{8}\gamma}{\sigma}t^{\sigma}$, we have the desired result.
\end{proof}

On the other hand, by considering a multiplication operator,
the scalar-valued function $M$ can be considered as a bounded linear operator
on $K$.

\begin{proposition}\label{prop: mathcal G bdd L2}
Let $\varphi\in\mathfrak{M}_{\gamma}(\mathbb{R}^{d})$ for some $\gamma\equiv\gamma_{\varphi}>0$
satisfying the condition \textbf{(S2)} and let $\psi\in\mathfrak{S}$.
Assume that $\gamma_{\varphi}\le \gamma_{\psi}$.
The operator $\mathcal{G}$ given in \eqref{eqn: mathcal G def}
is bounded on $L^{2}(\mathbb{R}^{d+1};K)$.
\end{proposition}

\begin{proof}
The proof is similar to Part 2 of the proof of Theorem 6.1 in \cite{I. Kim S. Lim K.-H. Kim 2016}.
We denote by $\mathcal{F}_{n}$ the $n$-dimensional Fourier transform.
By applying the Plancherel theorem, the conditions \eqref{eqn: sym cond 2} and \textbf{(S1)}, we obtain that
\begin{align}
\|\mathcal{G}f\|_{L^{2}(\mathbb{R}^{d+1};K)}^{2}
&=\int_{-\infty}^{\infty}\int_{\mathbb{R}^{d}}
    \left\|\int_{-\infty}^{t}L_{\varphi}\mathcal{T}_{\psi}(t,s)f(s,x)ds\right\|_{K}^{2}dxdt\nonumber\\
&=\int_{-\infty}^{\infty}\int_{\mathbb{R}^{d}}
    \left\|\mathcal{F}_{d}\left(\int_{-\infty}^{t}L_{\varphi}
        \mathcal{T}_{\psi}(t,s)f(s,\cdot)ds\right)(\xi)\right\|_{K}^{2}d\xi dt\nonumber\\
&=\int_{-\infty}^{\infty}\int_{\mathbb{R}^{d}}
    \left\|\int_{-\infty}^{t}\varphi(\xi)\exp\left(\int_{s}^{t}\psi(r,\xi)dr\right)\mathcal{F}_{d}(f)(s,\xi)ds\right\|_{K}^{2}d\xi dt\nonumber\\
&\leq\int_{-\infty}^{\infty}\int_{\mathbb{R}^{d}}
    \left(\int_{-\infty}^{t}\left|\varphi(\xi)\exp\left(\int_{s}^{t}\psi(r,\xi)dr\right)\right|
    \left\|\mathcal{F}_{d}(f)(s,\xi)\right\|_{K}ds\right)^{2}d\xi dt\nonumber\\
&\leq \mu_{\varphi}^{2}\int_{-\infty}^{\infty}\int_{\mathbb{R}^{d}}|\xi|^{2\gamma_{\varphi}}
    \left(\int_{-\infty}^{t}e^{-(t-s)\kappa_{\psi}|\xi|^{\gamma_{\psi}}}\|\mathcal{F}_{d}(f)(s,\xi)\|_{K}ds\right)^{2}d\xi dt.
    \label{eqn: mathcal G L2 1}
\end{align}
Note that
\begin{align*}
\int_{-\infty}^{t}e^{-(t-s)\kappa_{\psi}|\xi|^{\gamma_{\psi}}}\|\mathcal{F}_{d}(f)(s,\xi)\|_{K}ds
&=\int_{-\infty}^{\infty}1_{(0,\infty)}(t-s)e^{-(t-s)\kappa_{\psi}|\xi|^{\gamma_{\psi}}}\|\mathcal{F}_{d}(f)(s,\xi)\|_{K}ds\\
&=q(\cdot,\xi)*h(\cdot,\xi)(t),\quad t\in\mathbb{R},
\end{align*}
where $q(t,\xi)=1_{(0,\infty)}(t)e^{-t\kappa_{\psi}|\xi|^{\gamma_{\psi}}}$ and $h(t,\xi)=\|\mathcal{F}_{d}(f)(t,\xi)\|_{K}$.
Then we obtain that for any $u\in\mathbb{R}$,
\begin{align}
\mathcal{F}_{1}\left(\int_{-\infty}^{t}e^{-(t-s)\kappa_{\psi}|\xi|^{\gamma_{\psi}}}\|\mathcal{F}_{d}(f)(s,\xi)\|_{K}ds\right)(u)
&=\mathcal{F}_{1}(q(\cdot,\xi)*h(\cdot,\xi))(u)\nonumber\\
&=\mathcal{F}_{1}q(u,\xi)\mathcal{F}_{1}h(u,\xi)\nonumber\\
&=\frac{1}{iu+\kappa_{\psi}|\xi|^{\gamma_{\psi}}}\mathcal{F}_{1}h(u,\xi).\label{eqn: 1dim FT}
\end{align}
Since $\frac{|\xi|^{2\gamma_{\varphi}}}{u^2+\kappa_{\psi}^2|\xi|^{2\gamma_{\psi}}}$ is bounded
by the assumption $\gamma_{\psi}\leq \gamma_{\varphi}$
(say $\frac{|\xi|^{2\gamma_{\varphi}}}{u^2+\kappa_{\psi}^2|\xi|^{2\gamma_{\psi}}}\leq C$),
by applying \eqref{eqn: mathcal G L2 1}, Fubini theorem, Planchrel theorem and \eqref{eqn: 1dim FT},
we obtain that
\begin{align*}
\|\mathcal{G}f\|_{L^{2}(\mathbb{R}^{d+1};K)}^{2}
&\leq \mu_{\varphi}^{2}\int_{-\infty}^{\infty}\int_{\mathbb{R}^{d}}|\xi|^{2\gamma_{\varphi}}
    \left(\int_{-\infty}^{t}e^{-(t-s)\kappa_{\psi}|\xi|^{\gamma_{\psi}}} \|\mathcal{F}_{d}(f)(s,\xi)\|_{K}ds\right)^{2}d\xi dt\\
&=\mu_{\varphi}^{2}\int_{\mathbb{R}^{d}}\int_{-\infty}^{\infty}|\xi|^{2\gamma_{\varphi}}\left|\mathcal{F}_{1}
    \left(\int_{-\infty}^{t}e^{-(t-s)\kappa_{\psi}|\xi|^{\gamma_{\psi}}}
        \|\mathcal{F}_{d}(f)(s,\xi)\|_{K}ds\right)(u)\right|^{2}du d\xi \\
&=\mu_{\varphi}^{2}\int_{\mathbb{R}^{d}}\int_{-\infty}^{\infty}|\xi|^{2\gamma_{\varphi}}
    \left|\frac{1}{iu+\kappa_{\psi}|\xi|^{\gamma_{\psi}}}\mathcal{F}_{1}h(u,\xi)\right|^{2}du d\xi \\
&=\mu_{\varphi}^{2}\int_{\mathbb{R}^{d}}\int_{-\infty}^{\infty}
        \frac{|\xi|^{2\gamma_{\varphi}}}{u^2+\kappa_{\psi}^2|\xi|^{2\gamma_{\psi}}}\left|\mathcal{F}_{1}h(u,\xi)\right|^{2}du d\xi\\
&\leq \mu_{\varphi}^{2}C\int_{\mathbb{R}^{d}}
        \int_{-\infty}^{\infty}\left|\mathcal{F}_{1}h(u,\xi)\right|^{2}du d\xi\\
&=\mu_{\varphi}^{2}C\int_{\mathbb{R}^{d}}
        \int_{-\infty}^{\infty}\left|h(t,\xi)\right|^{2}dt d\xi\\
&=\mu_{\varphi}^{2}C\int_{\mathbb{R}^{d}}
        \int_{-\infty}^{\infty}\left\|\mathcal{F}_{d}(f)(t,\xi)\right\|_{K}^{2}dt d\xi\\
&=\mu_{\varphi}^{2}C\int_{-\infty}^{\infty}\int_{\mathbb{R}^{d}}\left\|f(t,x)\right\|_{K}^{2}dxdt\\
&=\mu_{\varphi}^{2}C\|f\|_{L^{2}(\mathbb{R}^{d+1};K)}^{2},
\end{align*}
which implies that the operator $\mathcal{G}$ is bounded on $L^{2}(\mathbb{R}^{d+1};K)$.
\end{proof}

\textbf{A proof of Lemma \ref{lem:L psi potential}.}
Let $\mathcal{G}$ be the operator given in \eqref{eqn: mathcal G def}
and let $M$ be the kernel given in \eqref{eqn: kernel M}.
By Proposition \ref{prop: mathcal G bdd L2}, the operator $\mathcal{G}$ is bounded on $L^{2}(\mathbb{R}^{d+1};K)$.
On the other hand, by Lemma \ref{lem:A1-A3},
we see that the kernel $M$ satisfies Assumption 2.1 in \cite{I. Kim S. Lim K.-H. Kim 2016}
and then by Lemma 4.1 in \cite{I. Kim S. Lim K.-H. Kim 2016}, the kernel $M$ satisfies
the H\"{o}rmander condition (3.1) in \cite{I. Kim S. Lim K.-H. Kim 2016}.
Also, we can check that the kernel $M$ satisfies the conditions (i) and (ii) in Definition 3.3 in \cite{I. Kim S. Lim K.-H. Kim 2016}.
So $M$ is a Calder\'{o}n-Zygmund kernel with respect to some filtration of partitions
(see \cite{I. Kim S. Lim K.-H. Kim 2016,Krylov 2001}).
Therefore, by Theorem 4.1 in \cite{Krylov 2001}, the operator $\mathcal{G}$ is bounded on $L^{p}(\mathbb{R}^{d+1};K)$
for $1<p\leq 2$.

We now prove that the operator $\mathcal{G}$ is bounded on $L^{p}(\mathbb{R}^{d+1};K)$ for $p>2$.
We use the duality argument.
Put
\begin{align*}
\widetilde{M}(s,y,t,x):=M(-t,-x,-s,-y)=1_{(0,\infty)}(s-t)L_{\varphi}p_{\widetilde{\psi}}(s,t,y-x),
\end{align*}
where $\widetilde{\psi}(r,\xi):=\psi(-r,\xi)$.
Put
\begin{align*}
\widetilde{\mathcal{G}}g(s,y)=\int_{\mathbb{R}^{d+1}}\widetilde{M}(s,y,t,x)g(t,x)dtdx.
\end{align*}
Let $p>2$ and let $p'$ be the conjugate number of $p$.
Then by Fubini theorem and changing variables $(t,x,s,y)\rightarrow (-t,-x,-s,-y)$, we obtain that
for any $g\in L^{p'}(\mathbb{R}^{d+1};K)$,
\begin{align}
&\int_{\mathbb{R}^{d+1}}\bilin{g(t,x)}{\mathcal{G}f(t,x)}_{K}dtdx\nonumber\\
&=\int_{\mathbb{R}^{d+1}}\bilin{g(t,x)}{\int_{\mathbb{R}^{d+1}}M(t,x,s,y)f(s,y)dsdy}_{K}dtdx\nonumber\\
&=\int_{\mathbb{R}^{d+1}}\bilin{\int_{\mathbb{R}^{d+1}}M(t,x,s,y)g(t,x)dtdx}{f(s,y)}_{K}dsdy\nonumber\\
&=\int_{\mathbb{R}^{d+1}}\bilin{\int_{\mathbb{R}^{d+1}}M(-t,-x,-s,-y)g(-t,-x)dtdx}{f(s,y)}_{K}dsdy\nonumber\\
&=\int_{\mathbb{R}^{d+1}}\bilin{\widetilde{\mathcal{G}}g_{1}(s,y)}{f(s,y)}_{K}dsdy,\label{eqn:duality}
\end{align}
where $g_{1}(t,x)=g(-t,-x)$. Since $1<p'<2$, it holds that
\begin{align*}
\|\widetilde{\mathcal{G}}g_{1}\|_{L^{p'}(\mathbb{R}^{d+1};K)}
\leq C\|g_{1}\|_{L^{p'}(\mathbb{R}^{d+1};K)}
=C\|g\|_{L^{p'}(\mathbb{R}^{d+1};K)},
\end{align*}
from which, by applying \eqref{eqn:duality}, we obtain that for any $f\in L^{p}(\mathbb{R}^{d+1};K)$,
\begin{align*}
\|\mathcal{G}f\|_{L^{p}(\mathbb{R}^{d+1};K)}
&=\sup_{\|g\|_{L^{p'}(\mathbb{R}^{d+1};K)}=1}\left|\int_{\mathbb{R}^{d+1}}
        \bilin{g(t,x)}{\mathcal{G}f(t,x)}_{K}dtdx\right|\\
&=\sup_{\|g\|_{L^{p'}(\mathbb{R}^{d+1};K)}=1}\left|\int_{\mathbb{R}^{d+1}}
        \bilin{\widetilde{\mathcal{G}}g_{1}(s,y)}{f(s,y)}_{K}dsdy\right|\\
&\leq \sup_{\|g\|_{L^{p'}(\mathbb{R}^{d+1};K)}=1}\|\widetilde{\mathcal{G}}g_{1}\|_{L^{p'}(\mathbb{R}^{d+1};K)}
\|f\|_{L^{p}(\mathbb{R}^{d+1};K)}\\
&=C\|f\|_{L^{p}(\mathbb{R}^{d+1};K)}.
\end{align*}
Therefore, the operator $\mathcal{G}$ is bounded on $L^{p}(\mathbb{R}^{d+1};K)$ for $p>2$.
Hence the proof is complete. \hfill $\blacksquare$

\section*{Acknowledgment}
This work was supported by the National Research Foundation of Korea (NRF) grant
funded by the Korea government (MSIT) (RS-2025-00573210)
and the MSIT (Ministry of Science and ICT), Korea, under the ITRC(Information Technology Research Center) support program (IITP-RS-2024-00437284) supervised by the IITP(Institute for Information \& Communications Technology Planning \& Evaluation).

\section*{Data Availability }
Data sharing is not applicable to this article as no datasets were generated or analysed during the current study.

\section*{Conflict of Interest}
The authors declare that they have no conflict of interest regarding this work.


\end{document}